\documentclass[12pt]{amsart}

\usepackage{amsmath,amssymb}
\usepackage{txfonts}
\usepackage{eucal}
\usepackage{mathrsfs}
\usepackage{adjustbox}
\usepackage{tikz}
\usetikzlibrary{arrows.meta}

\numberwithin{equation}{section}

\usepackage{graphicx}
\usepackage[all]{xy}

\usepackage{xspace}
\usepackage{adjustbox}
\usepackage{xcolor}

\usepackage[
    colorlinks=true,
    linkcolor=blue,
    citecolor=blue,
    urlcolor=blue,
    backref=page
]{hyperref}

\newtheorem{theorem}{Theorem}[section]
\newtheorem{prop}[theorem]{Proposition}
\newtheorem{proposition}[theorem]{Proposition}
\newtheorem{lemma}[theorem]{Lemma}
\newtheorem{corollary}[theorem]{Corollary}

\theoremstyle{definition}
\newtheorem{example}[theorem]{Example}
\newtheorem{defn}[theorem]{Definition}
\newtheorem{definition}[theorem]{Definition}
\newtheorem{remark}[theorem]{Remark}

\newcommand{\kk}{\Bbbk}
\newcommand{\ZZ}{\mathbb{Z}}
\newcommand{\ot}{\otimes}
\newcommand{\id}{{\rlap{1}\hskip1.6pt\adjustbox{scale=1.1}{1}}}

\DeclareMathOperator{\Hom}{Hom}
\DeclareMathOperator{\ad}{ad}

\newcommand{\calg}{\mathcal{G}}
\newcommand{\calq}{\mathcal{Q}}

\newcommand{\mquiver}{mixed quiver\xspace}
\newcommand{\mquivers}{mixed quivers\xspace}
\newcommand{\mpair}{\langle 2\rangle}
\newcommand{\mset}[1]{\{\{#1\}\}}
\newcommand{\cocont}[1]{_{\prec #1}}

\renewcommand{\eqref}[1]{Eq.\,(\ref{#1})}

\newcommand{\QSym}{\mathrm{QSym}}
\newcommand{\OSP}{\operatorname{OSP}}
\newcommand{\Ori}{\operatorname{Ori}}
\newcommand{\ps}{\operatorname{ps}} 
\newcommand{\act}{\mathbin{\triangleleft}}

\begin{document}

\title[]{A Hopf Algebraic Theory of the Quantum Magnusian}

\author{Li Guo, Joon-Hwi Kim, Jung-Wook Kim, Sungsoo Kim, Sangmin Lee, Jian-Rong Li}
\thanks{Preprint numbers: CERN-TH-2026-234, KIAS-P26047}
 
\address{LG: Department of Mathematics and Computer Science,
Rutgers University, Newark, NJ 07102, United States} 
\email{liguo@rutgers.edu}

\address{JHK: Walter Burke Institute for Theoretical Physics,
California Institute of Technology, Pasadena, CA 91125, United States}
\email{joonhwi@caltech.edu}

\address{JWK: Theoretical Physics Department, CERN, 1211 Geneva 23, Switzerland}
\email{jung-wook.kim@cern.ch}

\address{SK: Department of Physics and Astronomy, Seoul National University, 
1 Gwanak-ro, Gwanak-gu, Seoul 08826, Korea}
\email{sooo4017@snu.ac.kr}

\address{SL: School of Physics and Quantum Universe Center, Korea Institute for Advanced Study, 
85 Hoegi-ro, Dongdaemun-gu, Seoul 02455, Korea}
\email{sangminlee@kias.re.kr}

\address{JL: Faculty of Mathematics, University of Seville, Calle Tarfia s/n, 41012 Seville, Spain}
\email{lijr07@gmail.com}

\date{}

\begin{abstract}
We develop a Hopf-algebraic theory of graph coefficients arising in the quantum Magnus expansion. At the classical level, the expansion is governed by directed trees, whereas its quantum counterpart involves loop graphs, multiple edges, and several types of edge data. We introduce a contraction Hopf algebra of mixed quivers containing these diagrams and closed under contraction.

On this algebra we define a character $e$ from normalized linear-extension data, directed-tadpole factors, and a prescription for undirected edges, and let $\omega$ be its convolution inverse. Our main result is a universal closed formula for the connected function $\omega_c$ on every finite mixed quiver, expressed as a finite sum over ordered set partitions of the vertices. We also construct a weighted vertex-splitting Hopf structure on the unlocalized mixed-quiver algebra, with convolution $\star_V$ and canonical edge-weight character $\zeta_V$, such that $\omega_c=\log_{\star_V}\zeta_V$. Compatibility of the two coproducts and an explicit identity between their characters explain this equality. We further prove that the two-vertex contraction identity, tadpole factorization, vanishing on disconnected graphs, and boundary data uniquely determine $\omega_c$.

For acyclic quivers, the formula reduces to a permutation formula with coefficients depending only on descent numbers. We derive this formula independently from the quantum Magnus expansion using operator products and Wick contractions, showing that the physical graph coefficients are governed by the convolution structure of the contraction Hopf algebra. Further consequences include a refined quantum Murua formula and orientation-sum identities related to Tutte and chromatic polynomials.
\end{abstract}



\maketitle

\setcounter{tocdepth}{1}

\vspace{-.8cm}
\tableofcontents

\vspace{-.8cm}

\setcounter{section}{0}


\section{Introduction}

\subsection{Physics motivation} 
The Magnus expansion \cite{Magnus} converts 
the solution to a time evolution equation  
into the exponential of a Lie-algebra valued quantity.  Given
\begin{align}
\label{magnus-setup}
   \frac{d}{dt}U(t)\,U(t)^{-1}=A(t),
    \qquad
    U(0)=1,  
\end{align}
one writes
\begin{align}
\label{magnusian-general}
     U(t)=\exp \Omega(t),
    \qquad
    \Omega(t)=\sum_{n\geq1}\Omega_n(t),
\end{align}
where $\Omega_n(t)$ is a
linear combination of $n$-fold time-ordered integrals of nested
commutators. The series $\Omega(t)$ is the Magnus expansion.

The main topic of this paper is the algebraic and combinatorial structures behind the Magnus expansion. 
Since it draws motivation from physics, we give a brief exposition of the quantum Magnusian and its diagrammatic expansion from a physicist's perspective. 

In quantum mechanics or quantum field theory, 
$U(t)$ is the time-evolution operator.  
For a time-independent problem, $A(t) \rightarrow -iH$, with $H$ the Hamiltonian operator, and $U(t)=e^{-iHt}$ solves \eqref{magnus-setup}.  
For a time-dependent problem, a closed-form formula for $U(t)$ is not easy to guess. In a pioneering work \cite{Magnus}, Magnus proposed a formula to compute $\Omega(t)$ defined in \eqref{magnusian-general}.
In the recent physics literature, 
the operator $(i\hbar)\Omega$ has been named ``Magnusian'' \cite{Kim:2025gis,Kim:2025sey}. In practice, an exact evaluation of $\Omega(t)$ is rarely possible, and one often relies on approximate methods. When a perturbation theory (convergent or asymptotic) is available, a formula known as the Magnus expansion \cite{Magnus} 
allows one to compute  $\Omega_{(n)}$ in \eqref{magnusian-general} order by order. 

At the classical level of quantum mechanics, ref.~\cite{Kim:2024svw} showed that the Magnusian admits a graphical expansion of the form 
\begin{align}
    (\Omega_\text{classical})_{n} = \sum_{|\tau|=n} \frac{\omega(\tau)}{\sigma(\tau)} I(\tau) \,. 
\end{align}
The sum runs over connected, directed tree graphs $\tau$ with $n$ vertices. 
The directed edges arise from a time ordering. 
The function $\sigma(\tau)$, called the symmetry factor in the  physics literature, is the order of the automorphism group of $\tau$. 
The integral $I(\tau)$ contains the physics content, the details of which will be specified later. 
More important for the current paper is the fact \cite{Kim:2024svw} that $\omega(\tau) \in \mathbb{Q}$ is a generalization of the Murua coefficient \cite{Murua} 
to non-rooted trees. 

Given the quantum origin of the Magnusian, it is natural to ask how the graph expansion extends to the full quantum theory. The usual
$\hbar$ counting in Feynman-diagram expansions suggests that the
quantum Magnusian should involve graphs of arbitrary loop number,
where loops are cycles in the underlying undirected graph.
The graph expansion of the quantum Magnusian for a relativistic quantum field theory was initiated in refs.~\cite{Brandhuber:2025igz,Kim:2025ebl,Guo:2026xaw}.
The graph sum of the quantum Magnusian takes the form 
\begin{align}
\label{quantum-magnusian}
    \Omega_{n} = \sum_{|G|=n} \frac{\omega(G)}{\sigma(G)} I(G) \,, 
\end{align}
where the sum runs over connected graphs with $n$ vertices. 
In the so-called BW basis, the graphs in (\ref{quantum-magnusian}) include directed and undirected edges. The undirected edges only affect loop graphs, 
so no change in the tree result is needed. In the RB basis, the graphs in (\ref{quantum-magnusian}) include blue and red directed edges.
To treat these
different graph realizations within a common algebraic framework, we
work throughout with mixed quivers, allowing directed and undirected edges,
directed cycles, multiple edges, tadpoles, and edge
decorations.
As noted in ref.~\cite{Guo:2026xaw}, 
the BW and RB bases
have natural origins in 
the two complementary approaches 
to the quantum Magnusian
due to
refs.~\cite{Kim:2025ebl,Brandhuber:2025igz},
with ref.~\cite{Brandhuber:2025igz} working in the standard perturbative language of physicists
while ref.~\cite{Kim:2025ebl} utilizes deformation quantization \cite{Fedosov1994,kontsevich,Karabegov1996,BordemannWaldmann1997}.

At the classical level, the resulting connected graphs are directed
trees.  
The Hopf algebra of
Calaque--Ebrahimi-Fard--Manchon (CEM) provides a natural algebraic framework
for the rooted-tree theory \cite{CEM}, 
later extended to non-rooted trees~\cite{Kim:2024svw}. 
The quantum problem is qualitatively different.  Loop graphs occur,
multiple contractions produce multiple edges, different diagrammatic
bases naturally involve both directed and undirected edges, and graph
contraction creates tadpoles and directed cycles.  
Thus the passage from the classical to the quantum Magnusian  
requires an algebraic structure capable of treating general graphs.
The purpose of this paper is to construct such a structure and to show
that it controls the graph coefficients on the physical classes treated
below, while extending intrinsically to arbitrary mixed quivers.

\subsection{Three realizations of the graph coefficient}

The central object of the paper is a single connected graph function
$\omega_c$.  Our main structural result gives two global
Hopf-algebraic realizations of this function: as a convolution inverse
for edge contraction and as a convolution logarithm for weighted vertex
splitting.  On the connected acyclic UC graphs relevant to the physical
expansion, the same function agrees with the coefficient produced by
the quantum Magnus expansion.

Hopf-algebraic structures on graphs have a long history. Schmitt's work on incidence Hopf algebras~\cite{schmitt1994incidence} provides an early framework for restriction--contraction constructions, while the Connes--Kreimer Hopf algebra~\cite{Connes:1998qv} established the importance of graph Hopf algebras in quantum field theory. Foissy's study of bialgebras in cointeraction~\cite{FoissyCointeraction} provides a general framework for the character-action mechanism that also appears in Section~\ref{sec:master-formulas}. More recently, Ebrahimi-Fard and Fl\o ystad~\cite{EbrahimiFardFloystad2024} developed an edge-subset extraction--contraction cointeraction for incidence hypergraphs.

In this paper, we construct a new contraction Hopf algebra on arbitrary mixed quivers, allowing directed and undirected edges, directed cycles, multiple edges, tadpoles, and independent edge decorations. The underlying class of mixed quivers arises naturally as the closure of the physical diagram class under graph contraction. Our coproduct is defined using spanning cut-sub-mixed-quivers and contraction of their connected components. This construction is specifically designed to capture the edge-weighted mixed-quiver contractions underlying the Magnus coefficients studied in this paper, providing a Hopf-algebraic foundation for the quantum Magnusian theory.

Let $\kk\mathcal G$ be the commutative algebra spanned by finite mixed
quivers, with multiplication given by disjoint union.  If $H\preceq G$
is a cut-sub-mixed quiver and $G\cocont H$ is obtained by contracting
the connected components of $H$, put
\begin{equation}
\label{eq:intro-coproduct}
    \Delta_c(G)
    =
    \sum_{H\preceq G} H\otimes(G\cocont H).
\end{equation}
After localizing at the grouplike one-vertex graph $\bullet$, this gives
the contraction Hopf algebra
$\mathcal H_c=\kk\mathcal G[\bullet^{-1}]$; see Theorem~\ref{t:ghopfm}.
On $\mathcal H_c$ we define a character $e$ from normalized linear
extensions, together with the directed- and undirected-tadpole
prescriptions of Section~\ref{sec:graph functions}, and define
\begin{equation}
\label{eq:intro-omega-definition}
    \omega*e=\epsilon,
\end{equation}
where $*$ denotes contraction convolution.  We write $\omega_c$ for
the connected part of $\omega$.

The same coefficient has a second Hopf-algebraic description.  On the
unlocalized algebra $\kk\mathcal G$, now graded by the number of
vertices, we define a weighted vertex-splitting coproduct $\Delta_V$.
Let $\star_V$ denote its convolution product and let $\zeta_V$ be the
edge-weight character.  The resulting Hopf algebra is connected graded by
Theorem~\ref{thm:weighted-vertex-hopf}, and Theorem~\ref{thm:two-hopf-omega} gives
\begin{equation}
\label{eq:intro-two-hopf}
    \omega_c(G)
    =
    \bigl(\log_{\star_V}\zeta_V\bigr)(G)
\end{equation}
for every finite mixed quiver $G$. The two coproducts encode genuinely different decompositions, namely edge-subset contraction and ordered vertex splitting, yet produce the same connected coefficient. Their agreement also has a direct structural explanation: on the common unlocalized graph algebra, the flipped contraction coproduct is compatible with the weighted vertex-splitting coproduct, and the characters $e$ and $\zeta_V$ satisfy an explicit bridge identity.  We record this short cointeraction argument in Section~\ref{sec:master-formulas} as a complementary explanation of \eqref{eq:intro-two-hopf}; the local differential characterization developed earlier remains useful independently. The
canonical morphism from the vertex-splitting Hopf algebra to $\QSym$
is precisely the quasisymmetric enumerator underlying the master
formula; see Proposition~\ref{prop:qsym-master}.  Consequently, the
ordered-set-partition coefficients
$(-1)^{r-1}/r$ arise as the universal coefficients of a convolution
logarithm rather than as an independent combinatorial ansatz.

The mixed-quiver formulation is essential for the physical application.
If $G_D$ is the spanning directed core of a mixed quiver $G$, with
parameters $\boldsymbol\beta$ on directed edges and
$\boldsymbol\gamma$ on undirected edges, then
\begin{equation}
\label{eq:intro-mixed-factorization}
    \omega_c(G;\boldsymbol\beta,\boldsymbol\gamma)
    =
    \left(\prod_{u\in E_U(G)}\gamma_u\right)
    \omega_c(G_D;\boldsymbol\beta).
\end{equation}
By Theorem~\ref{thm:mixed-convolution-factorization}, this identity holds
for arbitrary finite mixed quivers.  In the physical acyclic UC sector,
the RB-to-UC operator change of basis yields exactly the same coefficient;
see Corollary~\ref{cor:UC-mixed-Hopf-compatibility}.  Thus the algebraically defined
$\omega_c$ is not merely an extension of the tree coefficient: on the
physical graph class it is the graph coefficient of the quantum
Magnusian itself.

\subsection{Closed formulas and local characterization}

The two descriptions lead to a closed formula for the same
function.  Let $G$ be a finite mixed quiver with nonempty vertex set
$V_G$, and write
$\beta_{\varepsilon,\pm}=(\beta_\varepsilon\pm1)/2$ for directed
edges.  For an ordered set partition
$\mathcal B=(B_1|\cdots|B_r)$ of $V_G$ and a directed edge
$\varepsilon:j\to i$, let
$s_{\mathcal B}(\varepsilon)=+$ if the block containing $i$ precedes
the block containing $j$, and $s_{\mathcal B}(\varepsilon)=-$
otherwise.  Then
\begin{equation}
\label{eq:intro-general-master}
    \omega_c(G)
    =
    \left(\prod_{u\in E_U(G)}\gamma_u\right)
    \sum_{\mathcal B\in\OSP(V_G)}
    \frac{(-1)^{|\mathcal B|-1}}{|\mathcal B|}
    \prod_{\varepsilon\in E_D(G)}
    \beta_{\varepsilon,s_{\mathcal B}(\varepsilon)}.
\end{equation}
By Theorem~\ref{thm:general-master-formula}, this formula is valid without
acyclicity or simplicity assumptions.

The local differential calculus is likewise mixed.  For every directed
edge $\varepsilon$ and undirected edge $u$,
\begin{equation}
\label{eq:intro-edge-PDE}
    2\frac{\partial}{\partial\beta_\varepsilon}\omega_c(G)
    =\omega_c(G_{\setminus\varepsilon}),
    \qquad
    \frac{\partial}{\partial\gamma_u}\omega_c(G)
    =\omega_c(G_{\setminus u}).
\end{equation}
These identities are Proposition~\ref{prop:omega-PDE}.  For quivers, the
two-vertex contraction rule, tadpole factorization,
vanishing on disconnected graphs, and the boundary value at
$\boldsymbol\beta=\mathbf1$ give the uniqueness characterization of
Theorem~\ref{thm:master-uniqueness}.  Thus the contraction calculus provides the local recursive
description, while the vertex logarithm gives the global
ordered-set-partition expression \eqref{eq:intro-general-master}.

For an acyclic quiver on $n$ vertices, labeled so that every edge is
$j\to i$ with $i<j$, the general formula reduces to the permutation
formula
\begin{equation}
\label{eq:intro-acyclic-master}
    \omega_c(G)
    =
    \sum_{\sigma\in S_n}
    \frac{(-1)^{n-1-k(\sigma)}}
         {n\binom{n-1}{k(\sigma)}}
    \prod_{\varepsilon\in E_G}
    \beta_{\varepsilon,s_\sigma(\varepsilon)},
\end{equation}
where $k(\sigma)$ is the number of descents and
$s_\sigma(\varepsilon)$ records the relative positions of the endpoints
of $\varepsilon$ in $\sigma$.  This is Theorem~\ref{thm:special-master-formula}
and is the form that meets the operator-product Magnus expansion.

As an application of the general master formula, we establish
a $k$-vertex contraction rule for $\omega_c$, valid for every
integer $k\geq 2$ (Theorem~\ref{thm:higher-oriented-edge-contraction}).

\subsection{Operator realization and the quantum Magnusian}

The operator-product calculation independently reproduces the explicit
acyclic coefficient formula, while the RB convolution identity identifies
the Hopf series with the physical Magnusian.  The
Mielnik--Pleba\'nski--Strichartz permutation formula
expresses the $n$th Magnus term as a sum of ordered operator words with
coefficients depending only on their descent numbers.  In the
oscillator realization, Wick contractions turn these words into
red--blue (RB) quivers: time ordering fixes the edge direction, while
operator ordering distinguishes the two ordered contractions.
Parallel Wick contractions produce multiple edges and hence loop
graphs.

We prove an RB convolution identity whose graphwise form implies
\[
    \exp_\star(\Omega_{\mathrm{RB}})=U,
    \qquad
    \Omega_{\mathrm{RB}}=\Omega.
\]
This is Theorem~\ref{thm:global-RB-Hopf-operator-compatibility}.  Thus the
connected RB series built from the Hopf $\omega$-function is
exactly the Magnusian, not merely a combinatorial model for it.

To pass to the UC basis, for a reference
orientation $\varepsilon:j\to i$ and underlying unoriented edge
$u=|\varepsilon|$ we use the local change of basis
\[
    W_{ji}\theta_{ij}
    \longmapsto
    \beta_{\varepsilon,+}\mathrm U_\varepsilon\theta_{ij}
    +\gamma_u\mathrm C_u\theta_{ij},
    \qquad
    W_{ij}\theta_{ij}
    \longmapsto
    \beta_{\varepsilon,-}\mathrm U_\varepsilon\theta_{ij}
    +\gamma_u\mathrm C_u\theta_{ij}.
\]
For every connected acyclic UC graph $G$,
Theorem~\ref{thm:UC-directed-core-factorization-reorg} and
Corollary~\ref{cor:UC-mixed-Hopf-compatibility} give
\begin{equation}
\label{eq:intro-black-core}
    \omega_{\mathrm{UC},c}(G;\boldsymbol\beta,\boldsymbol\gamma)
    =
    \omega_c(G;\boldsymbol\beta,\boldsymbol\gamma)
    =
    \left(\prod_{u\in E_U(G)}\gamma_u\right)
    \omega_c(G_D;\boldsymbol\beta).
\end{equation}
The first equality is the compatibility between the operator basis
change and direct mixed-quiver convolution; the second is the
intrinsic directed-core factorization.  The physical black--white
basis is recovered at
$(\boldsymbol\beta,\boldsymbol\gamma)=(\mathbf0,\mathbf1)$.

Finally, restricting the basis change to the purely directed sector
and combining it with the Magnus permutation formula gives an
independent operator-product derivation of
\eqref{eq:intro-acyclic-master}.  If $\Omega_n^{\mathrm{op}}(G)$ denotes
the coefficient obtained from operator products, then for every
connected acyclic quiver
\begin{equation}
\label{eq:intro-two-realizations}
    \Omega_n^{\mathrm{op}}(G)=\omega_c(G).
\end{equation}
This is Theorem~\ref{thm:operator-special-master-formula}.  This second
derivation is the direct bridge between the abstract Hopf
coefficient and the coefficients of the quantum Magnus expansion.

\subsection{Further consequences}

The acyclic formula yields a refined quantum Murua formula.  For every
connected acyclic oriented multigraph $G$, every semi-root $s$, and
arbitrary edge parameters $\boldsymbol\beta$, we construct a spider-web
expression $\mathcal M_s(G;\boldsymbol\beta)$ and prove
\[
    \omega_c(G;\boldsymbol\beta)
    =
    \mathcal M_s(G;\boldsymbol\beta).
\]
This is Theorem~\ref{thm:refined-quantum-murua-reorg}.  It simultaneously
recovers the classical rooted-tree Murua formula,
its non-rooted tree extension, and the RB/BW quantum specializations
\cite{Murua,Kim:2024svw,Guo:2026xaw}.

We also obtain orientation-sum identities.  For a tadpole-free
undirected graph $\widehat G$, orient each edge independently and put
$\overline\beta_u=(\beta_\varepsilon+\beta_{\bar\varepsilon})/2$.
Then
\begin{equation}
\label{eq:intro-orientation-sum}
    \sum_{G\in\Ori(\widehat G)}
    \omega_c(G;\boldsymbol\beta)
    =
    \sum_{\substack{F\subseteq E\\(V,F)\ {\rm connected}}}
    (-1)^{|F|}
    \prod_{u\in E\setminus F}\overline\beta_u.
\end{equation}
This is Theorem~\ref{thm:omega-orientation-sum}.  The right-hand side is
the linear coefficient of a multivariate graph
polynomial; its uniform specializations recover expressions involving
the Tutte and chromatic polynomials; see
Subsection~\ref{subsec:sum-rule-Tutte}.

 \begin{figure}[htbp]
\centering
\adjustbox{scale=0.93}{
\begin{tikzpicture}[
  >={Stealth[length=2.1mm,width=1.3mm]},
  line width=0.5pt,
  every node/.style={font=\small,inner sep=1.5pt,outer sep=0pt},
  step label/.style={font=\scriptsize,inner sep=1pt},
  flow/.style={->,shorten <=2pt,shorten >=2pt}
] 
  \node (hopf)     at (0,0)      {Hopf algebra};
  \node (magnus)   at (6,0)      {Magnus expansion};
  \node (general)  at (0,-1.1)   {General master formula};
  \node (operator) at (6,-1.1)   {Operator products};
  \node (sum)      at (-1.2,-2.8){Sum rules};
  \node (acyclic)  at (3,-2.8)   {Acyclic master formula};
  \node (murua)    at (3,-3.9)   {Quantum Murua formula};
 
  \draw[flow] (hopf.south) -- node[step label,left=2pt] {(Sections \ref{sec:hopf}--\ref{sec:master-formulas})} (general.north);
  \draw[flow] (magnus.south) -- node[step label,right=2pt] {(Sections \ref{sec:operator-products-Hopf}--\ref{sec:RB-BW-bases-reorg})} (operator.north);
  \draw[flow] (general.south) -- node[step label,left=1pt,pos=0.55] {(Section \ref{sec:sum-rules})} (sum);
  \draw[flow] (general.south) -- node[step label,right=5pt,pos=0.52] {(Section \ref{sec:master-formulas})} (acyclic);
  \draw[flow] (operator.south) -- node[step label,right=5pt,pos=0.52] {(Section \ref{sec:operator-acyclic-master})} (acyclic);
  \draw[flow] (acyclic.south) -- node[step label,right=2pt] {(Section \ref{sec:quantum-Murua-reorg})} (murua.north);
\end{tikzpicture}
}
\caption{Overall flow of the paper.}
\label{fig:overview}
\end{figure} 

\subsection{Organization}
The architecture of the paper is summarized in Figure~ \ref{fig:overview}. 
Sections~\ref{sec:hopf}--\ref{sec:differential-equations} construct the
contraction Hopf algebra, define $e$ and $\omega$, and develop the local
contraction and differential calculus.  Section~\ref{sec:master-formulas}
introduces the weighted vertex-splitting Hopf structure, proves
\eqref{eq:intro-two-hopf}, and derives the general and acyclic master
formulas together with their canonical $\QSym$ realization.
Sections~\ref{sec:operator-products-Hopf}--\ref{sec:operator-acyclic-master}
develop the operator realization: Wick contractions produce RB graphs,
the RB convolution identity identifies their connected series with the
Magnusian, and the RB-to-UC basis change yields the physical mixed-graph
coefficients.  In the purely directed sector, the Magnus permutation
formula gives an independent second derivation of the explicit acyclic
formula.  The final Sections \ref{sec:quantum-Murua-reorg} and \ref{sec:sum-rules} establish the refined quantum Murua formula
and the orientation-sum identities.

\section{Hopf algebra}
\label{sec:hopf}

We construct the contraction Hopf algebra of mixed quivers.  The
coproduct is defined by spanning edge cuts followed by contraction;
localizing at the grouplike one-vertex graph $\bullet$ gives
$\mathcal H_c=\kk\calg[\bullet^{-1}]$.

\subsection{Mixed quivers}
\label{subsec:mixed quivers}
A mixed-edge graph may have directed and undirected edges; we also
allow multiple edges, decorations, and tadpoles.  We call a self-loop
a \emph{tadpole}, reserving \emph{loop} for a cycle in the underlying
undirected graph.  A \emph{quiver} means a purely directed mixed
quiver, with multiple edges and tadpoles allowed.  We use \emph{graph}
generically for directed, undirected, or mixed graphs.  By convention,
the empty graph is regarded as disconnected; in particular, every
connected graph is nonempty.

Let $X$ be a set.  We write $X^{\mset{2}}$ for the set of multisets
of size two with entries in $X$, and denote the disjoint union
\[
    X^{\mpair}:=X^2\sqcup X^{\mset{2}}.
\]

\begin{defn}
A \emph{\mquiver} is a triple $G=(V_G,E_G,\rho_G)$,
where $V_G$ is the set of vertices, $E_G$ is the set of edges, and
$\rho_G:E_G\longrightarrow V_G^{\mpair}$
is the incidence map. If $\Omega$ is a nonempty set, an $\Omega$-decorated \mquiver is a
\mquiver together with a decomposition $E_G=\bigsqcup_{\omega\in\Omega}E_{G,\omega}$.
\end{defn}

Decorated \mquiver isomorphisms preserve decorations; we suppress
the decoration set when harmless.

The two summands in $V_G^{\langle 2\rangle}=V_G^2\sqcup V_G^{\mset{2}}$ encode the two types of edges: an element $(i,j)\in V_G^2$ represents a directed edge from $i$ to $j$, whereas an element $\mset{i,j}\in V_G^{\mset{2}}$ represents an undirected edge joining $i$ and $j$. 

A mixed quiver is an ordinary quiver precisely when $\operatorname{im}\rho_G\subseteq V_G^2$. We say that a quiver is \emph{acyclic} if it contains no directed
cycle; in particular, an acyclic quiver has no tadpoles.

An edge-sub-\mquiver determined by $E_H\subseteq E_G$ is obtained by
retaining only the vertices incident with edges of $E_H$.  For the
Hopf-algebraic construction, however, we shall use the following
spanning version.

\begin{defn}
\label{def:cut-sub-mquiver}
A \emph{cut-sub-\mquiver} of $G=(V_G,E_G,\rho_G)$
is a \mquiver $H=(V_G,E_H,\rho_H)$,
where $E_H\subseteq E_G$, $\rho_H=\rho_G|_{E_H}$.
Thus all vertices of $G$ are retained, including isolated vertices.
We write $H\preceq G$ if $H$ is a cut-sub-\mquiver of $G$. If $G$ is a quiver, we also call $H$ a cut-subquiver. 
\end{defn}

A cut-sub-\mquiver is uniquely determined by its edge set, so there
are $2^{|E_G|}$ cut-sub-\mquivers of $G$.  We denote the one with empty edge set by $G_{\emptyset}:=(V_G,\emptyset)$.
Thus $G_{\emptyset}=\bullet^{|V_G|}$,
where $\bullet^m$ denotes the edgeless \mquiver on $m$ vertices and
$\bullet^0=\emptyset$.

We next define the contraction.  Let $H\preceq G$.  Let $\sim_H$ be the
equivalence relation on $V_G$ generated by $x\sim_H y$
whenever $x$ and $y$ are the endpoints of an edge of $H$, and let
\[
    \pi_H:V_G\longrightarrow V_G/{\sim_H}
\]
be the quotient map.  Define
\begin{align*}
{G}\cocont{H}
    :=
    \left(
        V_G/{\sim_H},
        E_G\setminus E_H,
        \rho_{{G}\cocont{H}}
    \right), \text{where } \rho_{{G}\cocont{H}}
    :=
    \pi_H^{\mpair}\circ
    \rho_G|_{E_G{\setminus} E_H}.
\end{align*}
Here, for a map $f:X\to Y$,
\[
    f^{\mpair}(x,y)=(f(x),f(y)),
    \qquad
    f^{\mpair}(\mset{x,y})
    =
    \mset{f(x),f(y)}.
\]
In particular,
\begin{equation}
\label{eq:cut-extreme-contractions}
    {G}\cocont{G_{\emptyset}}=G.
\end{equation}
If $G$ is connected, then
\begin{equation}
\label{eq:full-contraction-connected}
    {G}\cocont{G}=\bullet.
\end{equation}
See Fig.~\ref{fig:cocont-example} for a simple nontrivial example of $G\cocont{H}$. 
\begin{figure}[htbp]
    \centering
    \includegraphics[width=0.6\linewidth]{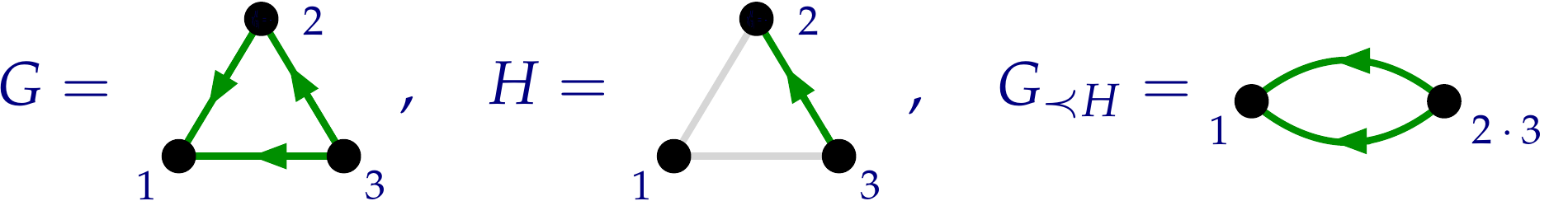}
    \caption{An example of $G$, $H$ and $G\cocont{H}$. The vertex label $(2\cdot 3)$ indicates that the two vertices have collapsed into one.}
    \label{fig:cocont-example}
\end{figure}

\subsection{Hopf algebra} \label{subsec:Hopf algebra}

Fix a nonempty decoration set $\Omega$.  In this subsection,
$\calg$ denotes the set of isomorphism classes of $\Omega$-decorated
\mquivers and $\calq$ the set of isomorphism classes of connected
$\Omega$-decorated \mquivers; the undecorated theory is recovered by
taking $\Omega$ to be a singleton.  All cuts and contractions retain
the decorations of surviving edges.
Under disjoint union, $\calg$ is the free commutative monoid
generated by $\calq$. Let $\kk$ be a field of characteristic zero and let $\kk\calg$
be the monoid algebra of $\calg$.  
Since $\mathcal G$ is the free commutative monoid generated by $\mathcal Q$, we may identify $\kk\calg \cong \kk[\calq]$, 
the free commutative $\kk$-algebra generated by $\calq$.
Its multiplication is induced by disjoint union:
\[
    G_1G_2:=G_1\sqcup G_2.
\]
The unit is the empty \mquiver:
\[
    \eta:\kk\longrightarrow\kk\calg,
    \qquad
    \eta(1)=\emptyset.
\]
We also write $1=\emptyset$ when no confusion can arise.

Since an isomorphism of \mquivers induces a bijection between their
cut-sub-\mquivers and commutes with contraction, the following
construction is well defined on isomorphism classes. Define
\begin{equation}
\label{eq:graphcoprodm}
    \Delta(G)
    :=
    \sum_{H\preceq G}
    H\otimes({G}\cocont{H})
\end{equation}
for every \mquiver $G$, and extend it linearly to $\kk\calg$.
In particular,
\[
    \Delta(\emptyset)
    =
    \emptyset\otimes\emptyset,
    \qquad
    \Delta(\bullet^m)
    =
    \bullet^m\otimes\bullet^m.
\]
If $G$ is connected, $|V_G|=n$, and $E_G\neq\emptyset$, then
\eqref{eq:cut-extreme-contractions} and
\eqref{eq:full-contraction-connected} give
\begin{equation}
\label{eq:coproduct-extreme-terms}
\begin{split}
    \Delta(G)
    &=
    \bullet^n\otimes G
    +
    G\otimes\bullet
    +
    \sum_{\substack{
        H\preceq G\\
         E_H \neq \emptyset, H \neq G
    }}
    H\otimes({G}\cocont{H}).
\end{split}
\end{equation}

\begin{remark}[Relation with incidence Hopf algebras]
The restriction--contraction form of \eqref{eq:graphcoprodm} is reminiscent
of Schmitt's incidence Hopf algebras of graphs~\cite[\S 14]{schmitt1994incidence}.
Our coproduct differs in two features needed for the quantum problem: every
edge subset is allowed, because tadpoles and parallel edges are retained, and
every cut-sub-\mquiver is spanning, so isolated vertices remain visible under
contraction.  These conventions are also responsible for the nontrivial
grouplike element $\bullet$ below.
\end{remark}

\begin{proposition}
\label{prop:coproduct-multiplicative}
The coproduct is multiplicative: for any two \mquivers $G_1$ and $G_2$,
\[
    \Delta(G_1 G_2)=\Delta(G_1)\Delta(G_2).
\]
\end{proposition}

\begin{proof}
Every cut-sub-\mquiver of $G_1\sqcup G_2$ is uniquely
$H_1\sqcup H_2$ with $H_i\preceq G_i$, and
\[
(G_1\sqcup G_2)\cocont{(H_1\sqcup H_2)}
\cong ({G_1}\cocont{H_1})\sqcup({G_2}\cocont{H_2}).
\]
Substitution into \eqref{eq:graphcoprodm} gives the claim.
\end{proof}

\begin{prop}
\label{p:coass}
The coproduct $\Delta$ is coassociative.
\end{prop}

\begin{proof}
Expanding the left iterated coproduct gives
\begin{equation}
\label{eq:coassl}
(\Delta\ot\id)\Delta(G)
=
\sum_{I\preceq H\preceq G}
I\ot(H\cocont I)\ot(G\cocont H).
\end{equation}
For the other parenthesization,
\[
(\id\ot\Delta)\Delta(G)
=
\sum_{I\preceq G}\ \sum_{J\preceq G\cocont I}
I\ot J\ot((G\cocont I)\cocont J).
\]
Fix $I\preceq G$.  Since $E_{G\cocont I}=E_G\setminus E_I$, the map
\[
J\longmapsto H,\qquad E_H:=E_I\sqcup E_J,
\]
is a bijection from the cut-sub-\mquivers of $G\cocont I$ to the
cut-sub-\mquivers $H$ with $I\preceq H\preceq G$.  Under this bijection
$J\cong H\cocont I$, while contraction is transitive:
\[
(G\cocont I)\cocont(H\cocont I)\cong G\cocont H.
\]
Indeed, both quotients identify exactly the vertices lying in the same
connected component generated by the edges of $H$, and both retain the edge
set $E_G\setminus E_H$ with the induced incidence map and the same
surviving edge decorations.  Hence the right
iterated coproduct is exactly the sum in \eqref{eq:coassl}.
\end{proof}

Define $\epsilon:\kk\calg\longrightarrow\kk$ by
\begin{equation}
\label{eq:counit-definition}
\epsilon(G)
:=
\begin{cases}
1, & E_G=\emptyset,\\
0, & E_G\neq\emptyset.
\end{cases}
\end{equation}

\begin{proposition}
\label{prop:mquiver-counit}
The map $\epsilon$ is a counit for $\Delta$, that is, $(\epsilon\otimes\id)\Delta
=
\id
=
(\id\otimes\epsilon)\Delta$.
\end{proposition}

\begin{proof}
For any $G$, there is exactly one cut-sub-\mquiver with no edges,
namely $G_{\emptyset}=\bullet^{|V_G|}$. Hence
$(\epsilon\otimes\id)\Delta(G)
=G\cocont{G_{\emptyset}}=G$.
On the other hand, $G\cocont{H}$ is edgeless if and only if
$E_H=E_G$, equivalently $H=G$. Therefore
$(\id\otimes\epsilon)\Delta(G)=G$.
\end{proof}

The map $\epsilon$ is multiplicative, and
Propositions~\ref{prop:coproduct-multiplicative},
\ref{p:coass}, and~\ref{prop:mquiver-counit} show that 
$(\kk\calg,m,\eta,\Delta,\epsilon)$
is a commutative bialgebra. 

We next equip $\kk\calg\cong\kk[\calq]$ with a grading.
For a connected \mquiver $Q\in\calq$, define
$d(Q):=|E_Q|$.
Extending the degree additively under disjoint union gives $d(G_1\sqcup G_2)=d(G_1)+d(G_2)$
for arbitrary \mquivers $G_1,G_2$, and hence
$d(G)=|E_G|$ for every \mquiver $G$.
For $n\geq0$, let $\calg^{(n)}$ be the set of isomorphism classes of \mquivers of degree $n$, and let $(\kk\calg)^{(n)}$ be the spanned $\kk$-subspace. Then
\[
    \kk\calg
    =
    \bigoplus_{n\geq0}(\kk\calg)^{(n)}.
\]
Since
\begin{align*}
|E_H|
    +
    |E_{{G}\cocont{H}}|
    =
    |E_G|, \quad 
    |E_{G_1\sqcup G_2}|
    =
    |E_{G_1}|+|E_{G_2}|,
\end{align*} 
the product and coproduct are compatible with this grading.  Thus $\kk\calg$ is a
graded bialgebra.
The degree-zero part is
\[
    (\kk\calg)^{(0)}
    =
    \kk[\bullet].
\]
Recall that a graded $\kk$-algebra $A=\bigoplus_{n\geq0}A^{(n)}$
is called \emph{connected} if $A^{(0)}=\kk\cdot 1$.
Since the unit of $\kk\calg$ is the empty \mquiver $\emptyset$, whereas
\[
    \kk[\bullet]
    =
    \operatorname{span}_\kk
    \{1=\emptyset,\bullet,\bullet^2,\ldots\},
\]
the graded algebra $\kk\calg$ is not connected.
The element $\bullet$ is grouplike:
\[
    \Delta(\bullet)
    =
    \bullet\otimes\bullet,
    \qquad
    \epsilon(\bullet)=1.
\]
If $\kk\calg$ admitted an antipode, then
$S(\bullet)\bullet=1=\emptyset$,
so $\bullet$ would have to be invertible.  It is not invertible in
$\kk\calg$.  We therefore localize at the multiplicative set generated by $\bullet$ and set
\begin{equation}
\label{eq:localized-mquiver-Hopf-algebra}
    \mathcal{H}_c
    :=
    \kk\calg[\bullet^{-1}].
\end{equation}
Since $\kk\calg\cong\kk[\calq]$ and $\bullet$ is one of its connected
generators, if $\calq_+:=\{Q\in\calq\mid E_Q\neq\emptyset\}$,
then
\begin{equation}
\label{eq:localized-polynomial-description}
    \mathcal H_c
    \cong
    \kk[\bullet,\bullet^{-1},\calq_+].
\end{equation}
We construct the antipode explicitly below, so no general antipode
criterion for the localization is needed.

\begin{remark}
In contrast with the construction in \cite{CEM}, which is formulated for rooted trees and rooted forests and imposes the relation $\bullet=1$, we work with arbitrary mixed quivers and instead invert the grouplike element $\bullet$. 
\end{remark}

Because $\Delta(\bullet)=\bullet\otimes\bullet$ is invertible in
$\mathcal H_c\otimes\mathcal H_c$ and $\epsilon(\bullet)=1$ is invertible
in $\kk$, the universal property of localization extends the coproduct
and counit uniquely to algebra maps on $\mathcal H_c$, with
\begin{equation}
\label{eq:localized-coproduct-counit}
\Delta(\bullet^{-1})
=
\bullet^{-1}\otimes\bullet^{-1},
\qquad
\epsilon(\bullet^{-1})=1.
\end{equation}
The two sides of coassociativity, and likewise the two counit composites,
are algebra maps on $\mathcal H_c$.  They agree on $\kk\calg$ and on the
additional generator $\bullet^{-1}$, so the bialgebra identities extend
to all of $\mathcal H_c$.  Hence $\mathcal H_c$ is again a graded
commutative bialgebra, with
\[
    \mathcal{H}_c^{(0)}
    =
    \kk[\bullet,\bullet^{-1}].
\]

The coproduct defines convolution products on both
$\Hom_\kk(\kk\calg,\kk)$ and $\Hom_\kk(\mathcal H_c,\kk)$.
For linear maps $f,g$ in either space, define
\begin{equation}
\label{eq:convolution-product}
    f*g
    :=
    m_\kk\circ(f\otimes g)\circ\Delta.
\end{equation}
In particular, for every \mquiver $G$,
\begin{equation}
\label{eq:convolution-explicit}
    (f*g)(G)
    =
    \sum_{H\preceq G}
    f(H)\,
    g({G}\cocont{H}).
\end{equation}
Thus $(f*g)(\bullet^m)
    =
    f(\bullet^m)g(\bullet^m)$,
for $m\geq0$ in $\kk\calg$, and for every $m\in\mathbb Z$ in
$\mathcal H_c$, since $\bullet$ is grouplike.
If $f$ and $g$ are multiplicative with respect to disjoint union,
then so is $f*g$, because the coproduct is multiplicative.

\begin{theorem}
\label{t:ghopfm}
The localized bialgebra $\mathcal{H}_c=\kk\calg[\bullet^{-1}]$
is a graded commutative Hopf algebra.  Its unit is
\[
    \eta(1)=\emptyset,
\]
its counit is given by \eqref{eq:counit-definition}, together with
$\epsilon(\bullet^{-1})=1$, its antipode satisfies
\[
    S(\bullet^m)=\bullet^{-m},
    \qquad m\in\mathbb Z,
\]
and, for every connected \mquiver $G$ with $E_G\neq\emptyset$,
\begin{equation}
\label{eq:antipode-recursion}
    S(G)
    =
    -
    \left(
        \sum_{\substack{
            H\preceq G\\
            H\neq G
        }}
        S(H)({G}\cocont{H})
    \right)
    \bullet^{-1}.
\end{equation}
The antipode is extended multiplicatively to arbitrary \mquivers.
\end{theorem}

\begin{proof}
By \eqref{eq:localized-polynomial-description}, it is enough to define
$S$ on the algebra generators $\bullet^{\pm1}$ and the connected
positive-edge generators $Q\in\calq_+$, and then extend multiplicatively.
We construct the latter values by induction on the number of edges.  Set
\[
    S(\bullet^m)=\bullet^{-m},
    \qquad
    m\in\mathbb Z.
\]
Let $G$ be connected with $n=|V_G|$, $|E_G|>0$.
The edgeless cut-sub-\mquiver of $G$ is $G_{\emptyset}=\bullet^n$.
Using \eqref{eq:coproduct-extreme-terms}, the equation $(S*\id)(G)=0$
is
\[
    0
    =
    \bullet^{-n}G
    +
    S(G)\bullet
    +
    \sum_{\substack{
        H\preceq G\\
        E_H\neq\emptyset,\,
        H\neq G
    }}
    S(H)({G}\cocont{H}).
\]
Solving for $S(G)$ gives precisely
\eqref{eq:antipode-recursion}.  This is a valid induction because
every proper cut-sub-\mquiver appearing in the sum has fewer edges
than $G$.  Extend $S$ multiplicatively to all of $\mathcal H_c$.  Since
$\mathcal H_c$ is commutative and $\Delta$ is multiplicative, the
convolution $S*\id$ is an algebra map.  The preceding calculation checks
$S*\id=\eta\epsilon$ on every algebra generator in
\eqref{eq:localized-polynomial-description}; hence it holds on all of
$\mathcal H_c$.

For completeness, construct a right convolution inverse $T$ in the same
way.  Set $T(\bullet^m)=\bullet^{-m}$, $m\in\ZZ$.
For connected $G$ as above, define
\begin{equation}
\label{eq:right-antipode-recursion}
\begin{split}
    T(G)
    :=
    -\bullet^{-n}
    \Bigg(
        G\bullet^{-1}
        +
        \sum_{\substack{
            H\preceq G\\
            E_H\neq\emptyset,\,
            H\neq G
        }}
        H\,T({G}\cocont{H})
    \Bigg).
\end{split}
\end{equation}
Extend $T$ multiplicatively to $\mathcal H_c$.  Since
\[
    |E_{{G}\cocont{H}}|
    =
    |E_G|-|E_H|
    <
    |E_G|
\]
for every nonempty $H$, this is again well defined by induction.  As
above, $\id*T$ is an algebra map, and the defining calculation checks
$\id*T=\eta\epsilon$ on the algebra generators; hence the equality holds
on all of $\mathcal H_c$.  By associativity of convolution,
\[
    S
    =
    S*(\id*T)
    =
    (S*\id)*T
    =
    T.
\]
Hence $S*\id
    =
    \eta\epsilon
    =
    \id*S$,
and $S$ is the antipode.  The recursion preserves the edge grading,
so the antipode is degree-preserving.
\end{proof}

The convolution unit in both algebras $\Hom_\kk(\kk\calg,\kk)$
and $\Hom_\kk(\mathcal H_c,\kk)$ is $\epsilon$.
\begin{proposition}
\label{prop:convolution-unit}
For every linear map $f$ in either convolution algebra,
\[
    \epsilon*f=f=f*\epsilon.
\]
\end{proposition}

\section{Graph functions:
\texorpdfstring{$e$}{e} and \texorpdfstring{$\omega$}{omega}} \label{sec:graph functions}

We define $e$ and $\omega$ on the full mixed-quiver Hopf algebra
$\mathcal H_c$.  Their restriction to the Hopf subalgebra
$\mathcal H_c^{\mathrm{dir}}
=\kk\mathcal G_{\mathrm{dir}}[\bullet^{-1}]$
gives the directed specialization used for the order-theoretic and
acyclic results below.

\subsection{\texorpdfstring{$e$}{e}-function}

For the Hopf-algebraic statements, the edge decorations include fixed
parameter values in $\kk$: a directed edge $\varepsilon$ carries
$\beta_\varepsilon\in\kk$ and an undirected edge $u$ carries
$\gamma_u\in\kk$.  Isomorphisms preserve these decorations, and cuts and
contractions inherit them on surviving edges.  Polynomial identities and
derivatives are interpreted universally: for a fixed edge-labeled graph we
may replace the scalar parameters by independent indeterminates in
$\kk[\beta_\varepsilon,\gamma_u]$, since all identities involved are
polynomial and surviving edge labels are inherited by cuts and contractions.

We first recall the directed linear-extension part of the definition.
Let $Q=(V_Q,E_Q,\rho_Q)$ be a finite quiver and let
$E_{\mathrm{tad}}(Q)$ denote its set of directed tadpoles.  Write
\[
    Q_{\mathrm{pruned}}
    :=
    \left(
        V_Q,\,
        E_Q\setminus E_{\mathrm{tad}}(Q),\,
        \rho_Q|_{E_Q\setminus E_{\mathrm{tad}}(Q)}
    \right).
\]
For every directed edge $\varepsilon$, set
\begin{align}
\label{eq:def of beta epsilon plus minus}
\beta_{\varepsilon,\pm}
:=
\frac{\beta_\varepsilon\pm1}{2}.
\end{align}
If $Q$ is acyclic, let $\operatorname{LE}(Q)$ denote the set of
linear extensions of the partial order induced by its oriented edges
and put $\phi(Q):=|\operatorname{LE}(Q)|$.
For a tadpole-free quiver $Q$, define
\begin{equation}
\label{eq:e-function_and_linear_extension}
    e_0(Q)
    :=
    \begin{cases}
    \displaystyle
    \frac{\phi(Q)}{|V_Q|!},
    & \text{if $Q$ is acyclic},\\[6pt]
    0,
    & \text{if $Q$ contains a directed cycle}.
    \end{cases}
\end{equation}
\begin{definition}[$e$-function] \label{def:e function}
Let $G$ be a finite mixed quiver.  Write
$E_G=E_D(G)\sqcup E_U(G)$, let $G_D$ be the spanning directed core,
and let $E_U^{\mathrm{tad}}(G)$ be the set of undirected tadpoles.  Give
each directed edge $\varepsilon$ a parameter $\beta_\varepsilon$ and
each undirected edge $u$ a parameter $\gamma_u$.  Define
\begin{equation}
\label{eq:e-mixed-prescription}
    e(G)
    :=
    \begin{cases}
    \displaystyle
    (-1)^{|E_U(G)|}
    \left(\prod_{u\in E_U(G)}\gamma_u\right)
    \left(
        \prod_{\varepsilon\in E_{\mathrm{tad}}(G_D)}
        (-\beta_{\varepsilon,-})
    \right)
    e_0((G_D)_{\mathrm{pruned}}),
    & E_U(G)=E_U^{\mathrm{tad}}(G),\\[8pt]
    0,
    & \text{otherwise}.
    \end{cases}
\end{equation}
We extend $e$ $\kk$-linearly to $\kk\mathcal G$.
\end{definition}

For a purely directed quiver $Q$, Definition~\ref{def:e function}
reduces to
\begin{equation}
\label{eq:e-prescription}
    e(Q)
    =
    \left(
        \prod_{\varepsilon\in E_{\mathrm{tad}}(Q)}
        (-\beta_{\varepsilon,-})
    \right)
    e_0(Q_{\mathrm{pruned}}).
\end{equation}

On purely directed quivers, the $\beta$-dependence of $e(G)$ enters
only through directed tadpoles.  In the mixed theory, every undirected
tadpole contributes a factor $-\gamma_u$, while the presence of any
non-tadpole undirected edge forces $e(G)=0$.  Notice also that
\begin{equation}
\label{eq:e-edgeless}
    e(\bullet^m)=1
    \qquad
    (m\geq0),
\end{equation}
where $\bullet^0=\emptyset$.

\begin{remark}
\label{rk:quiver-vs-poset-e}
For an acyclic quiver, linear extensions are those of its reachability
poset \cite{Chan-Pak-2025,richard2011enumerative,HR25}; parallel and
transitively redundant arrows therefore do not change $e_0$.  The
tadpole factors in Definition~\ref{def:e function}, however, are
essential for closure under contraction.
\end{remark}

\begin{remark}
We display $(\boldsymbol\beta,\boldsymbol\gamma)$ only when needed.
The scalar and universal-polynomial interpretations above use the same
formulas, because surviving edge labels and parameters are retained under
cuts, deletion, and contraction.
\end{remark}

\begin{proposition}
\label{prop:e-multiplicative}
The $e$-function is multiplicative with respect to disjoint union:
\begin{equation}
\label{eq:e-multiplicative}
    e(G_1\sqcup G_2)
    =
    e(G_1)e(G_2)
\end{equation}
for finite mixed quivers $G_1,G_2$.
\end{proposition}

\begin{proof}
If one of $G_1,G_2$ contains a non-tadpole undirected edge, then the
disjoint union contains the same edge, so both sides of
\eqref{eq:e-multiplicative} vanish.  Suppose therefore that every
undirected edge of each $G_i$ is a tadpole.  The undirected factors
$-\gamma_u$ and the directed-tadpole factors $-\beta_{\varepsilon,-}$
split multiplicatively between the two components.

It remains to check the normalized linear-extension factor.  Put
$Q_i:=((G_i)_D)_{\mathrm{pruned}}$ and $n_i:=|V_{G_i}|$.  If one $Q_i$
contains a directed cycle, then both $e_0(Q_1\sqcup Q_2)$ and
$e_0(Q_1)e_0(Q_2)$ vanish.  If both are acyclic, every linear extension
of $Q_1\sqcup Q_2$ is obtained uniquely by choosing a linear extension
of each $Q_i$ and interleaving the two resulting words.  Hence
\[
    \phi(Q_1\sqcup Q_2)
    =
    \binom{n_1+n_2}{n_1}\phi(Q_1)\phi(Q_2),
\]
and therefore
\[
    e_0(Q_1\sqcup Q_2)
    =
    \frac{\phi(Q_1\sqcup Q_2)}{(n_1+n_2)!}
    =
    e_0(Q_1)e_0(Q_2).
\]
Combining these factors with
$(G_1\sqcup G_2)_D=(G_1)_D\sqcup(G_2)_D$ proves
\eqref{eq:e-multiplicative}.
\end{proof}

Since $e$ is multiplicative and $e(\bullet)=1$, the universal
property of localization gives the following.
\begin{lemma} \label{lem:extend e to Hc}
The $e$-function extends uniquely to a character
\[
    e:\mathcal H_c=\kk\mathcal G[\bullet^{-1}]\longrightarrow\kk
\]
by setting $e(\bullet^{-1})=1$.
\end{lemma}

\subsection{\texorpdfstring{$\omega$}{omega}-function}
\label{subsec:omega function}

\begin{definition} [$\omega$-function] \label{def:omega function}
The \(\omega\)-function is the convolution inverse of $e$ on the full
mixed-quiver Hopf algebra:
\begin{equation}
\label{eq:omega-convolution-definition}
    \omega
    :=
    e\circ S
    =
    e^{*(-1)}.
\end{equation}
Equivalently,
\[
    \omega*e=e*\omega=\epsilon.
\]
\end{definition}

We make the definition explicit.  Since $\bullet$ is grouplike and
$e(\bullet)=1$, we have
\begin{equation}
\label{eq:omega-edgeless}
    \omega(\bullet^m)=1,
    \qquad m\in\mathbb Z.
\end{equation}
In particular, $\omega(G)=1$ for every edgeless mixed quiver $G$.
Now suppose that $E_G\neq\emptyset$, and let
$G_{\emptyset}
    =
    (V_G,\emptyset)
    =
    \bullet^{|V_G|}$
be the unique cut-sub-\mquiver of $G$ with no edges.  Since
\[
    {G}\cocont{G_{\emptyset}}=G,
    \qquad
    {G}\cocont{G}
    =
    \bullet^{c(G)},
\]
where $c(G)$ is the number of connected components of $G$, the
convolution equation gives
\begin{equation}
\label{omega-Hopf-cycle}
    \omega(G)
    =
    -e(G)
    -
    \sum_{\substack{
        H\preceq G\\
        0<|E_H|<|E_G|
    }}
    \omega(H)e({G}\cocont{H}).
\end{equation}
Indeed, the contribution of $H=G_{\emptyset}$ is $e(G)$, whereas
the contribution of $H=G$ is $\omega(G)$.

The recursion~\eqref{omega-Hopf-cycle}, together with
\eqref{eq:omega-edgeless}, computes the graph values of $\omega$
recursively by the number of edges; see Fig.~\ref{fig:conv-example} for
an example.  
\begin{figure}[htbp]
    \centering
    \includegraphics[width=0.84\linewidth]{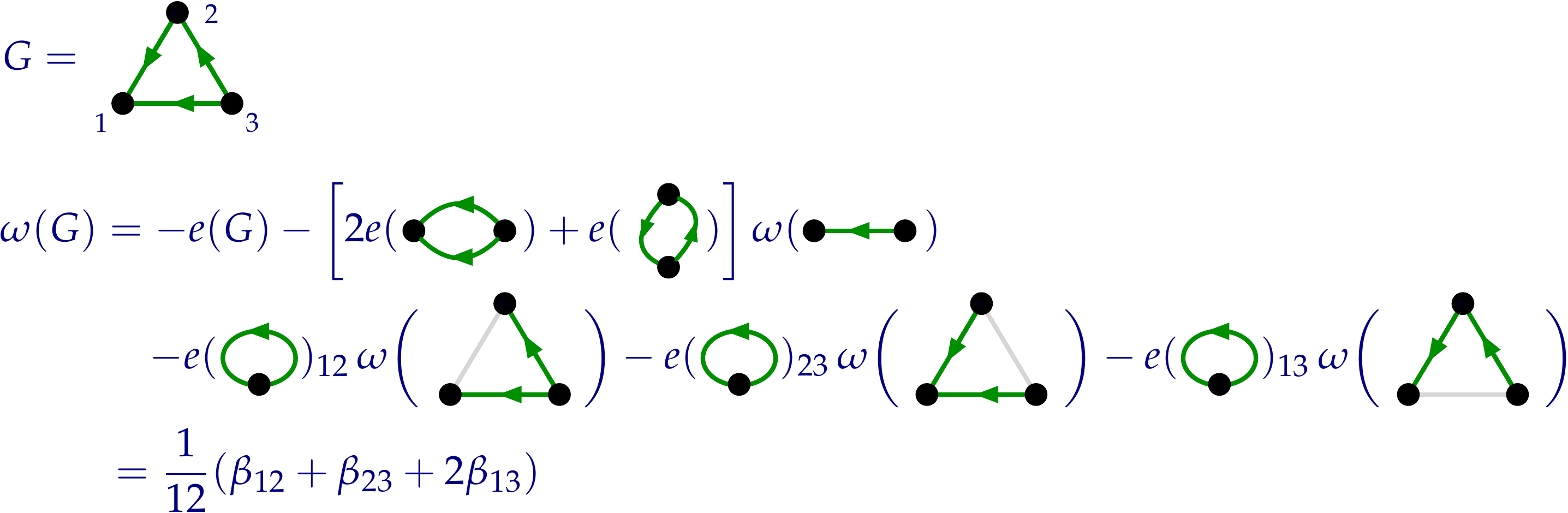}
    \caption{How the convolution formula computes $\omega(G)$, where $\beta_{\epsilon}$ is denoted by $\beta_{ij}$ if $\epsilon$ is the edge $j \to i$.}
    \label{fig:conv-example}
\end{figure}
 
\begin{proposition} \label{prop:convolution identities}
The function
$\omega\in\Hom_\kk(\mathcal H_c,\kk)$
is the two-sided convolution inverse of $e$:
\begin{equation}
\label{eq:omega-two-sided}
    \omega*e
    =
    e*\omega
    =
    \epsilon.
\end{equation}
Moreover,
\begin{equation}
\label{eq:omega-antipode}
    \omega
    =
    e\circ S,
\end{equation}
where $S$ denotes the antipode of $\mathcal H_c$.
Consequently, $\omega$ is multiplicative: for finite mixed quivers
$G_1,G_2$,
\begin{equation}
\label{eq:omega-multiplicative}
\omega(G_1\sqcup G_2)
=
\omega(G_1)\omega(G_2).
\end{equation}
Since $\omega(\emptyset)=1$, the function
$\omega:\mathcal H_c\longrightarrow\kk$
is therefore also a character.
\end{proposition}

\begin{proof}
Equations~\eqref{eq:omega-two-sided} and \eqref{eq:omega-antipode} are
immediate from Definition~\ref{def:omega function} and the antipode
identities.  Since the antipode is an anti-algebra morphism and $e$ is
multiplicative,
\[
\begin{aligned}
    \omega(G_1\sqcup G_2)
    &=
    e\bigl(S(G_1\sqcup G_2)\bigr)
    =
    e\bigl(S(G_2)S(G_1)\bigr)
    =
    \omega(G_1)\omega(G_2).
 \qquad   \qquad
\end{aligned} \qedhere
\]
\end{proof}

For a tadpole-free directed quiver, $e$ is parameter-independent,
whereas its convolution inverse $\omega$ generally is not.

\subsection{Tadpole factorization}

The following lemma is the tadpole factorization for the $\omega$-function. 

\begin{lemma} 
\label{lem:omega-tadpole-factorization}
Let $G$ be a finite mixed quiver. If $\varepsilon$ is a directed tadpole, then
\begin{equation}
\label{eq:omega-one-tadpole-factorization}
    \omega(G)
    =
    \beta_{\varepsilon,-}\,
    \omega(G_{\setminus\varepsilon}).
\end{equation}
If $u$ is an undirected tadpole, then
\begin{equation}
\label{eq:omega-undirected-tadpole-factorization}
    \omega(G)
    =
    \gamma_u\,
    \omega(G_{\setminus u}).
\end{equation}
Consequently, for a quiver $G$,
\begin{equation}
\label{omega-tadpole}
    \omega(G)
    =
    \left(
        \prod_{\varepsilon\in E_{\mathrm{tad}}(G)}
        \beta_{\varepsilon,-}
    \right)
    \omega(G_{\mathrm{pruned}}).
\end{equation}
\end{lemma}

\begin{proof}
We prove both statements by induction on $|E_G|$.  Let $a$ be the
distinguished tadpole and put $\overline G:=G_{\setminus a}$.  Every
cut-sub-\mquiver of $G$ is uniquely either $H$ or $H\cup\{a\}$ with
$H\preceq\overline G$.  Moreover,
\[
G\cocont H=(\overline G\cocont H)\cup\{a\},
\qquad
G\cocont(H\cup\{a\})=\overline G\cocont H,
\]
because contracting a tadpole identifies no vertices.  If
$a=\varepsilon$ is directed, Definition~\ref{def:e function} gives
\[
    e(H\cup\{\varepsilon\})
    =-\beta_{\varepsilon,-}e(H),
\]
whereas for an undirected tadpole $a=u$ it gives
\[
    e(H\cup\{u\})=-\gamma_u e(H).
\]
Since $E_G\neq\emptyset$, the identity $e*\omega=\epsilon$ therefore
becomes, in the directed case,
\[
0=
\sum_{H\preceq\overline G}e(H)
\left[
\omega\bigl((\overline G\cocont H)\cup\{\varepsilon\}\bigr)
-\beta_{\varepsilon,-}\omega(\overline G\cocont H)
\right],
\]
and analogously with $\gamma_u$ in the undirected case.  If
$|E_G|=1$, only the edgeless $H$ occurs and the desired identity follows
immediately.  For $|E_G|>1$, every nonempty $H$ gives graphs with
strictly fewer edges than $G$, so the corresponding bracket vanishes by
induction.  The edgeless term is exactly
$\omega(G)-\beta_{\varepsilon,-}\omega(\overline G)$, respectively
$\omega(G)-\gamma_u\omega(\overline G)$.  This proves the two
factorizations, and iterating the directed one gives
\eqref{omega-tadpole}.
\end{proof}

For a one-vertex mixed quiver with one directed tadpole $\varepsilon$,
$\omega=\beta_{\varepsilon,-}$; with one undirected tadpole $u$,
$\omega=\gamma_u$.

The specialization $\boldsymbol\beta=\mathbf1$ has an especially simple consequence in
the purely directed theory.

\begin{proposition}
\label{prop:omega-cyclic-vanishing}
If a quiver $G$ contains a directed cycle, then
$\omega(G)|_{\boldsymbol\beta=\mathbf1}=0$.
\end{proposition}

\begin{proof}
Suppose that $G$ contains a directed cycle and
$\beta_\varepsilon=1$ for every $\varepsilon\in E_G$. We prove that
$\omega(G)=0$ by induction on $|E_G|$. If $G$ contains a tadpole, the
claim follows immediately from
Lemma~\ref{lem:omega-tadpole-factorization}, since
$\beta_{\varepsilon,-}=0$.

Assume therefore that $G$ is tadpole-free and contains a directed
cycle $C$. Then $e(G)=0$, and \eqref{omega-Hopf-cycle} gives
\[
    \omega(G)
    =-
    \sum_{\substack{H\preceq G\\0<|E_H|<|E_G|}}
    \omega(H)e(G\cocont H).
\]
If $H$ contains every edge of $C$, then $H$ contains a directed cycle,
so $\omega(H)=0$ by induction. Otherwise the image of $C$ in
$G\cocont H$ contains a directed cycle or a tadpole; in either case
$e(G\cocont H)=0$ at $\boldsymbol\beta=\mathbf1$. Hence every term
vanishes.
\end{proof}

\subsection{Connected \texorpdfstring{$\omega$}{omega}-function}
\label{subsec:connected-omega}

Since $\omega$ is multiplicative with respect to disjoint unions, its
values are determined by its values on connected graphs.  For later
use, we define the connected $\omega$-function
\begin{equation}
\label{eq:omega-connected}
    \omega_c(G)
    :=
    \begin{cases}
        \omega(G), & \text{if $G$ is connected},\\
        0, & \text{if $G$ is disconnected}.
    \end{cases}
\end{equation}
In particular,
\[
    \omega_c(\bullet)=1,
    \qquad
    \omega_c(\emptyset)=0,
\]
whereas $\omega(\emptyset)=1$.
If $G=G_1\sqcup\cdots\sqcup G_r$
is the decomposition of a mixed quiver $G$ into connected components, then
\[
    \omega(G)
    =
    \prod_{a=1}^r\omega_c(G_a),
\]
where for $G=\emptyset$ the product on the right-hand side is understood
as the empty product, equal to $1$.

\begin{lemma}[Connected-projected convolution identity]
\label{lem:connected-projected-convolution}
Let $G$ be a finite mixed quiver with $E_G\neq\emptyset$. Then
\[
    \sum_{K\preceq G}
    e(K)\,\omega_c(G\cocont K)
    =0.
\]
\end{lemma}

\begin{proof}
If $G$ is connected, then $G\cocont K$ is connected for every
$K\preceq G$. Hence
\[
    \sum_{K\preceq G}
    e(K)\,\omega_c(G\cocont K)
    =
    (e*\omega)(G)
    =
    \epsilon(G)
    =
    0,
\]
where the last equality follows from $E_G\neq\emptyset$.
If $G$ is disconnected, contracting edges cannot join distinct
connected components. Hence $G\cocont K$ is disconnected for every
$K\preceq G$, and every term in the sum vanishes by the definition of
$\omega_c$.
\end{proof}

\subsection{Directed-core factorization}
\label{subsec:mixed-convolution-extension}

The full mixed definition of $e$ is arranged so that its convolution
inverse has a particularly simple dependence on undirected edges.

\begin{theorem}[Directed-core factorization]
\label{thm:mixed-convolution-factorization}
For every finite mixed quiver $G$,
\begin{equation}
\label{eq:mixed-convolution-factorization}
    \omega_c
    (G;\boldsymbol\beta,\boldsymbol\gamma)
    =
    \left(
        \prod_{u\in E_U(G)}\gamma_u
    \right)
    \omega_c(G_D;\boldsymbol\beta).
\end{equation}
\end{theorem}

\begin{proof}
Set
\[
    F(G)
    :=
    \left(\prod_{u\in E_U(G)}\gamma_u\right)
    \omega_c(G_D;\boldsymbol\beta).
\]
Then $F(\bullet)=1$, and $F(G)=0$ whenever $G$ is disconnected.  By
Lemma~\ref{lem:connected-projected-convolution}, $\omega_c$ is
characterized recursively by
\begin{equation}
\label{eq:mixed-connected-projected}
    \sum_{H\preceq G}
    e(H)\,\omega_c(G\cocont H)
    =0
    \qquad (E_G\neq\emptyset),
\end{equation}
together with its value on $\bullet$ and its vanishing on disconnected
graphs.  Indeed, if $G$ is connected, the edgeless cut-subquiver
$G_{\emptyset}$ contributes $\omega_c(G)$, while every other
cut-subquiver contains at least one edge and hence produces a quotient
with strictly fewer edges.  Thus \eqref{eq:mixed-connected-projected}
determines the connected values by induction on $|E_G|$.  It is
therefore enough to verify the same relation for $F$.

A term can contribute only if every undirected edge selected in $H$ is
already an undirected tadpole of $G$.  Such an $H$ is uniquely
described by a directed cut-subquiver $K\preceq G_D$ and a subset
$T\subseteq E_U^{\mathrm{tad}}(G)$, with
\[
    E_H=E_K\sqcup T.
\]
This is an edge-set decomposition.  Since the edges in $T$ are
tadpoles, they do not identify vertices.  Consequently,
\[
    (G\cocont H)_D={G_D}\cocont{K},
    \qquad
    E_U(G\cocont H)=E_U(G)\setminus T.
\]
Hence
\begin{align*}
e(H)F(G\cocont H) =
(-1)^{|T|}
\left(\prod_{u\in T}\gamma_u\right)e(K)
\left(\prod_{u\in E_U(G)\setminus T}\gamma_u\right)
\omega_c({G_D}\cocont{K})=
\left(\prod_{u\in E_U(G)}\gamma_u\right)
(-1)^{|T|}e(K)\omega_c({G_D}\cocont{K}).
\end{align*}
Summing over all contributing $H$ gives
\begin{equation}
\label{eq:mixed-factorized-convolution-sum}
\left(\prod_{u\in E_U(G)}\gamma_u\right)
\left(
\sum_{T\subseteq E_U^{\mathrm{tad}}(G)}(-1)^{|T|}
\right)
\left(
\sum_{K\preceq G_D}e(K)\omega_c({G_D}\cocont{K})
\right).
\end{equation}
If $E_U^{\mathrm{tad}}(G)\neq\emptyset$, the middle factor is zero.
If $E_U^{\mathrm{tad}}(G)=\emptyset$ and $E_D(G)\neq\emptyset$, the
last factor is zero by
Lemma~\ref{lem:connected-projected-convolution}, applied to $G_D$.
Finally, if both sets are empty while $E_G\neq\emptyset$, every edge of
$G$ is a non-tadpole undirected edge.  Thus $|V_G|\geq2$ and
$G_D=\bullet^{|V_G|}$ is disconnected.  Since $G_D$ has no directed
edges, the only directed cut-subquiver $K\preceq G_D$ is $G_D$ itself,
so the last factor in \eqref{eq:mixed-factorized-convolution-sum} is
\[
    e(G_D)\,\omega_c(G_D)=0.
\]
Hence the required convolution relation holds in this final case as
well.  Therefore $F=\omega_c$.
\end{proof}

\begin{remark}
The connected projection in
Theorem~\ref{thm:mixed-convolution-factorization} is essential.  A mixed
quiver may be connected because of its undirected edges while its
directed core is disconnected; in that situation $\omega_c(G)=0$ by
the theorem, whereas the multiplicative value $\omega(G_D)$ need not
vanish.  Thus one cannot in general replace $\omega_c$ by $\omega$ in
\eqref{eq:mixed-convolution-factorization}.
\end{remark}

\section{Contraction rules}
\label{sec:contraction}

We develop the local contraction calculus for $e$ and $\omega_c$ on
finite mixed quivers.  The distinguished edges in the local relations
are directed; the remaining background edges may be directed or
undirected.

\subsection{Notation for local modifications}

Let $G=(V_G,E_G,\rho_G)$ be a mixed quiver and let
$\varepsilon=[i\to j]\in E_D(G)$ be a distinguished directed edge
with distinct endpoints. We use the following three local modifications.
\begin{enumerate}
\item
$G_{\bar{\varepsilon}}$ denotes the quiver obtained from $G$ by
replacing the distinguished edge
$\varepsilon:i\longrightarrow j$
by the oppositely oriented distinguished edge
$\bar{\varepsilon}:j\longrightarrow i$. 
The new edge $\bar{\varepsilon}$ is given a new edge variable $\beta_{\bar{\varepsilon}}$ independent of the original $\beta_{\varepsilon}$. 
All other vertices,
edges, incidences, orientations, decorations, and edge parameters are
left unchanged. In particular, if other edges join $i$ and $j$, they
are not modified.

\item
$G_{\setminus\varepsilon}$ denotes the quiver obtained from $G$ by
removing the distinguished edge $\varepsilon$ while retaining both
of its endpoints:
\[
    G_{\setminus\varepsilon}
    :=
    \left(
        V_G,\,
        E_G\setminus\{\varepsilon\},\,
        \rho_G\big|_{E_G\setminus\{\varepsilon\}}
    \right).
\]
Thus vertices which become isolated after removing $\varepsilon$ are
retained, and $G_{\setminus\varepsilon}$ may be connected or
disconnected.

\item
$G\cocont{\varepsilon}$ denotes the contraction of the distinguished
edge $\varepsilon$, using the quotient notation introduced in Section~\ref{subsec:mixed quivers}.
More precisely, $\varepsilon$ is regarded as the cut-sub-\mquiver with
edge set $\{\varepsilon\}$ and the full vertex set $V_G$. Thus
$\varepsilon$ is removed, its endpoints $i$ and $j$ are identified,
and every other edge is retained with the induced incidence. In
particular, every other edge joining $i$ and $j$ becomes a tadpole at
the merged vertex.

\end{enumerate}

The notation extends to a specified set $F\subseteq E_G$ of
distinguished edges. We write
$G_{\setminus F}$
for the graph obtained by removing all edges in $F$, while $G\cocont{F}$
denotes their simultaneous contraction. If $R\subseteq F$ is a set of
oriented distinguished edges to be reversed, we write $G_{\bar R}$
for the graph obtained by reversing precisely the edges in $R$. These operations are edge-based and therefore
remain unambiguous when parallel edges are present.

\begin{remark}[Vertex-based shorthand]
\label{rem:old-local-notation-dictionary}
For a background mixed quiver $H$ with marked vertices $i,j$, write
$H_{[i\to j]}$ for adjoining one new distinguished edge $i\to j$,
$H_{[i|j]}:=H$, and $H_{[i\cdot j]}$ for identifying $i$ and $j$.
Thus, if $\varepsilon:i\to j$ is distinguished in $G$ and
$H=G_{\setminus\varepsilon}$,
\begin{equation}
\label{eq:background-fixed-edge-dictionary}
H_{[i\to j]}=G,\qquad
H_{[i\leftarrow j]}=G_{\bar\varepsilon},\qquad
H_{[i|j]}=G_{\setminus\varepsilon},\qquad
H_{[i\cdot j]}=G\cocont{\varepsilon}.
\end{equation}
For several marked vertices, vertical bars keep groups distinct, dots
identify vertices, and
$H_{[i_1\leftarrow\cdots\leftarrow i_k]}$ means that the new
distinguished chain
$i_k\to\cdots\to i_1$ is adjoined.  All pre-existing edges are retained.
Figure~\ref{fig:cut-cont-notation} illustrates this shorthand.
\begin{figure}[htbp]
    \centering
    \includegraphics[width=0.56\linewidth]{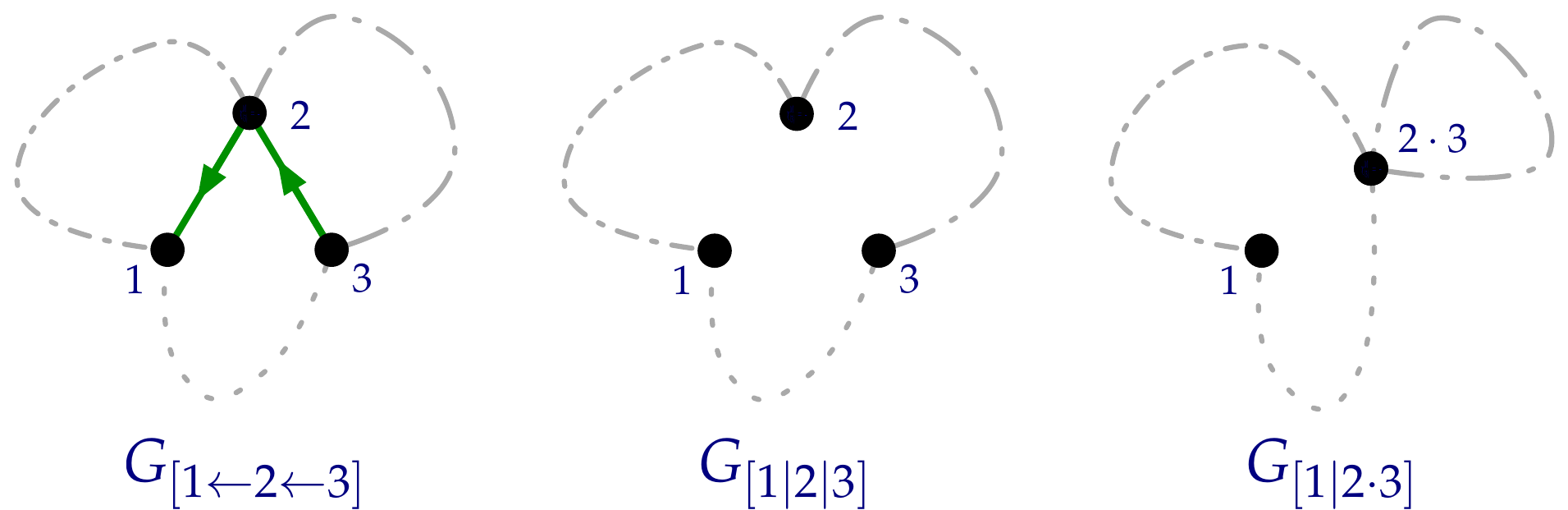}
    \caption{Vertex-based shorthand for local modifications. Vertical
    bars indicate marked vertices, or groups of marked vertices, which
    remain distinct, while dots indicate identifications. Arrows
    indicate newly adjoined distinguished edges.}
    \label{fig:cut-cont-notation}
\end{figure}
\end{remark}

\subsection{Contraction rule for the \texorpdfstring{$e$}{e}-function}

\begin{proposition}[Multi-vertex $e$-contraction]
\label{prop:e-multi-vertex-contraction}
Let $G$ be a finite mixed quiver with distinct marked vertices
$i_1,\ldots,i_m$. Then
\begin{equation}
\label{eq:e-multi-vertex-contraction}
    \sum_{\sigma\in S_m}
    e\bigl(
        G_{[
        i_{\sigma(1)}
        \leftarrow
        i_{\sigma(2)}
        \leftarrow\cdots\leftarrow
        i_{\sigma(m)}
        ]}
    \bigr)
    =
    e(G).
\end{equation}
\end{proposition}

\begin{proof}
If $G$ contains a non-tadpole undirected edge, then every graph in
\eqref{eq:e-multi-vertex-contraction} contains the same edge, so all
terms vanish by Definition~\ref{def:e function}.  Otherwise every
undirected edge is a tadpole, and the common factor
\[
(-1)^{|E_U(G)|}\prod_{u\in E_U(G)}\gamma_u
\]
may be removed.  The identity is then exactly the directed
multi-vertex contraction identity for $G_D$: linear extensions of the
pruned directed core are partitioned according to the relative order
of the marked vertices.  Directed tadpole factors are common to all
terms as well.  This proves the result.
\end{proof}

For later use, we record the cases $m=2$ and $m=3$.  If
$\varepsilon$ is a distinguished directed edge of a mixed quiver $G$
with reverse $\bar\varepsilon$, apply
\eqref{eq:e-multi-vertex-contraction} to the background
$H=G_{\setminus\varepsilon}$ to obtain
\begin{equation}
\label{eq:e-two-vertex-contraction}
    e(G)+e(G_{\bar\varepsilon})
    =
    e(G_{\setminus\varepsilon}).
\end{equation}
For three distinct marked vertices $1,2,3$, \eqref{eq:e-multi-vertex-contraction} gives
\begin{equation}
\label{eq:e-three-vertex-contraction}
    \sum_{\sigma\in S_3}
    e\bigl(
        G_{[\sigma(1)\leftarrow\sigma(2)\leftarrow\sigma(3)]}
    \bigr)
    =
    e(G).
\end{equation}

For a tree, \eqref{eq:e-two-vertex-contraction} reduces to the
$e$-contraction rule of \cite[Eq.~(4.22)]{Kim:2024svw}, since
$G_{\setminus\varepsilon}$ is disconnected and $e$ is multiplicative.

\subsection{Two-vertex contraction rule for the
\texorpdfstring{$\omega$}{omega}-function}

The proof below uses the connected-projected convolution identity
(Lemma~\ref{lem:connected-projected-convolution}) together with the
$e$-contraction rule and tadpole factorization.  The point requiring
care is that, after contracting a cut-sub-\mquiver, the distinguished
directed edge may either retain distinct endpoints or become a
tadpole.

\begin{theorem}[Two-vertex $\omega$-contraction rule]
\label{thm:omega-two-vertex-contraction}
Let $G$ be a finite mixed quiver and let $\varepsilon$ be a
directed edge of $G$ with distinct endpoints.
Then
\begin{equation}
\label{eq:omega_contraction-beta} 
    \omega_c(G)
    +
    \omega_c(G_{\bar{\varepsilon}})
    +
    \omega_c(G\cocont{\varepsilon})
    =
    \frac{
        \beta_{\varepsilon}
        +
        \beta_{\bar{\varepsilon}}
    }{2}
    \omega_c(G_{\setminus\varepsilon}). 
\end{equation} 
\end{theorem}

\begin{proof}
If $G$ is disconnected, then $G_{\bar\varepsilon}$,
$G\cocont\varepsilon$, and $G_{\setminus\varepsilon}$ are all
disconnected, so every connected coefficient in
\eqref{eq:omega_contraction-beta} vanishes.  Assume henceforth that
$G$ is connected and argue by induction on $m:=|E_G|$.

If $m=1$, then $G$ and $G_{\bar\varepsilon}$ are connected one-edge
quivers,
$G\cocont\varepsilon=\bullet$, and
$G_{\setminus\varepsilon}=\bullet\sqcup\bullet$.  The convolution
relation gives
$\omega_c(G)=\omega_c(G_{\bar\varepsilon})=-\frac12$, while
$\omega_c(\bullet)=1$ and
$\omega_c(\bullet\sqcup\bullet)=0$.  Thus
\eqref{eq:omega_contraction-beta} holds.

Let $m>1$ and assume the theorem for all mixed quivers with fewer than
$m$ edges.  Put
\[
    H:=G_{\setminus\varepsilon},
    \qquad
    C:=G\cocont\varepsilon.
\]
Applying Lemma~\ref{lem:connected-projected-convolution} to the
connected mixed quivers $G$ and $G_{\bar\varepsilon}$ and adding the
two identities gives
\begin{equation}
\label{contraction-induction}
0=
\sum_{A\preceq G}e(A)\omega_c(G\cocont A)
+
\sum_{B\preceq G_{\bar\varepsilon}}
e(B)\omega_c({G_{\bar\varepsilon}}\cocont B).
\end{equation}

Every cut-sub-\mquiver of $G$ is uniquely either $K$ or
$L:=K\cup\{\varepsilon\}$ for a unique $K\preceq H$; likewise every
cut-sub-\mquiver of $G_{\bar\varepsilon}$ is uniquely either $K$ or
$\bar L:=K\cup\{\bar\varepsilon\}$.  All of these are spanning
cut-sub-\mquivers on $V_G$.

Let
$q_\varepsilon:V_G\to V_C$ be the quotient map identifying the two
endpoints of $\varepsilon$.  Since contraction removes the selected
edge occurrence but does not identify distinct surviving edge
occurrences, $E_C=E_H$.  For $K\preceq H$, define
$K^\varepsilon\preceq C$ by
\[
    E_{K^\varepsilon}:=E_K,
    \qquad
    \rho_{K^\varepsilon}
    :=q_\varepsilon^{\mpair}\circ\rho_K.
\]
Thus $K\mapsto K^\varepsilon$ is a bijection between the
cut-sub-\mquivers of $H$ and those of $C$, with all surviving edge
decorations unchanged.  Transitivity of contraction gives canonical
isomorphisms
\begin{equation}
\label{eq:paired-contractions}
    G\cocont L
    \cong
    C\cocont{K^\varepsilon}
    \cong
    {G_{\bar\varepsilon}}\cocont{\bar L}.
\end{equation}
Moreover, applying \eqref{eq:e-two-vertex-contraction} to the
distinguished edge of $L$ gives
\begin{equation}
\label{eq:paired-e}
    e(L)+e(\bar L)=e(K).
\end{equation}
Pairing in \eqref{contraction-induction} the four terms associated
with a fixed $K\preceq H$ therefore yields
\begin{equation}
\label{eq:paired-J-fixed-edge}
0=
\sum_{K\preceq H}e(K)
\Bigl[
    \omega_c(G\cocont K)
    +\omega_c({G_{\bar\varepsilon}}\cocont K)
    +\omega_c(C\cocont{K^\varepsilon})
\Bigr].
\end{equation}

Fix now a nonempty $K\preceq H$.  First suppose that contracting $K$
does not identify the endpoints of $\varepsilon$, and let
$\varepsilon_K$ be the image of the distinguished edge in
$G\cocont K$.  Then
${G_{\bar\varepsilon}}\cocont K$ is obtained by reversing
$\varepsilon_K$ (with parameter $\beta_{\bar\varepsilon}$), while
\[
    (G\cocont K)_{\setminus\varepsilon_K}
    \cong H\cocont K,
    \qquad
    (G\cocont K)\cocont{\varepsilon_K}
    \cong C\cocont{K^\varepsilon}.
\]
Because edge occurrences are retained under contraction unless they
are selected, $|E_{G\cocont K}|=m-|E_K|<m$.  The induction hypothesis
applied to $G\cocont K$ gives
\begin{equation}
\label{eq:inductive-bracket}
\begin{split}
    &\omega_c(G\cocont K)
    +\omega_c({G_{\bar\varepsilon}}\cocont K)
    +\omega_c(C\cocont{K^\varepsilon}) =
    \frac{\beta_\varepsilon+\beta_{\bar\varepsilon}}{2}
    \,\omega_c(H\cocont K).
\end{split}
\end{equation}

Suppose instead that contracting $K$ already identifies the endpoints
of $\varepsilon$.  Then the images of $\varepsilon$ and
$\bar\varepsilon$ are directed tadpoles with parameters
$\beta_\varepsilon$ and $\beta_{\bar\varepsilon}$, respectively.
Deleting either distinguished tadpole gives $H\cocont K$, and
transitivity of contraction gives
$C\cocont{K^\varepsilon}\cong H\cocont K$.
All graphs here are connected.  Tadpole factorization therefore gives
\[
    \omega_c(G\cocont K)
    =\beta_{\varepsilon,-}\,\omega_c(H\cocont K),
    \qquad
    \omega_c({G_{\bar\varepsilon}}\cocont K)
    =\beta_{\bar\varepsilon,-}\,\omega_c(H\cocont K).
\]
Since
$\beta_{\varepsilon,-}+\beta_{\bar\varepsilon,-}+1
=(\beta_\varepsilon+\beta_{\bar\varepsilon})/2$,
\eqref{eq:inductive-bracket} follows in this case as well.

The term $K=H_\emptyset$ in \eqref{eq:paired-J-fixed-edge} is
\[
    \omega_c(G)+\omega_c(G_{\bar\varepsilon})+\omega_c(C),
\]
because $e(H_\emptyset)=1$.  Since $m>1$, the graph $H$ has a
nonempty edge set.  Separating the edgeless cut-sub-\mquiver in
Lemma~\ref{lem:connected-projected-convolution} applied to $H$ gives
\begin{equation}
\label{eq:cut-convolution}
    \sum_{\substack{K\preceq H\\E_K\neq\emptyset}}
    e(K)\,\omega_c(H\cocont K)
    =-\omega_c(H).
\end{equation}
Substituting \eqref{eq:inductive-bracket} and
\eqref{eq:cut-convolution} into
\eqref{eq:paired-J-fixed-edge}, and recalling
$H=G_{\setminus\varepsilon}$ and $C=G\cocont\varepsilon$, proves
\eqref{eq:omega_contraction-beta}.
\end{proof}

\begin{example}
\label{ex:omega-parallel-edge-contraction}
Let $G$ have two vertices and two parallel edges
\[
    \varepsilon:i\longrightarrow j,
    \qquad
    \delta:i\longrightarrow j,
\]
where $\varepsilon$ is the distinguished edge.  We have that
\begin{align*}
\omega_c(G\cocont{\varepsilon})
    =
   \frac{\beta_{\delta} - 1}{2}, \quad  \omega_c(G)
    =
    -\frac{\beta_{\varepsilon}+\beta_{\delta}}{4}, \quad \omega_c(G_{\bar{\varepsilon}})
    =
    -\frac{\beta_{\delta}+ \beta_{\bar{\varepsilon}} -2}{4},  \quad \omega_c(G_{\setminus\varepsilon})
    =
    -\frac12.
\end{align*}  
Therefore \eqref{eq:omega_contraction-beta} is satisfied. 
\end{example}

\begin{figure}[htbp]
    \centering
    \includegraphics[width=0.84\linewidth]{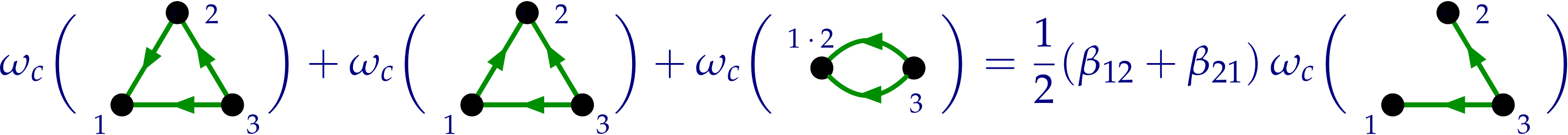}
    \caption{The two-vertex contraction rule for a triangle
    graph. Here $\beta_{\epsilon}$ is denoted by $\beta_{ij}$ if $\epsilon$ is the edge $j \to i$.}
    \label{fig:2-vertex-ex}
\end{figure}
 
\begin{example}[Triangle graph]
\label{ex:omega-triangle-contraction}
Let $G$ be the oriented triangle with vertex set $\{1,2,3\}$ and
edges
\[
    \varepsilon:2\longrightarrow1,
    \qquad
    \delta:3\longrightarrow2,
    \qquad
    \gamma:3\longrightarrow1,
\]
where $\varepsilon$ is the distinguished edge. 
Fig.~\ref{fig:2-vertex-ex} illustrates the two-vertex contraction rule
for this graph, with the distinguished edge $\varepsilon$ indicated
in the figure.
A direct convolution calculation gives
\begin{align*}
    \omega_c(G)
    =
    \frac{
        \beta_{\varepsilon}
        +
        \beta_{\delta}
        +
        2\beta_{\gamma}
    }{12}, \ 
    \omega_c(G_{\bar{\varepsilon}})
    =
    \frac{
        \beta_{\bar{\varepsilon}}
        +
        2\beta_{\delta}
        +
        \beta_{\gamma}
    }{12}, \ 
    \omega_c(G\cocont{\varepsilon})
    =
    -\frac{
        \beta_{\gamma}
        +
        \beta_{\delta}
    }{4}, \ 
    \omega_c(G_{\setminus\varepsilon})
    =
    \frac16.
\end{align*}
Consequently, \eqref{eq:omega_contraction-beta} is satisfied.
\end{example}

\begin{remark}[Trees]
For a tree, $\omega_c(\tau_{\setminus\varepsilon})=0$, so
Theorem~\ref{thm:omega-two-vertex-contraction} reduces to
\begin{equation}
\label{eq:omega_contraction-tree}
\omega_c(\tau)+\omega_c(\tau_{\bar\varepsilon})
=-\omega_c(\tau\cocont{\varepsilon}),
\end{equation}
recovering \cite[Proposition~4.3]{Chartier_2010} and
\cite[Eq.~(4.23)]{Kim:2024svw}.
\end{remark}

\subsection{An application of the contraction rule: removing composite edges}

An edge $\varepsilon=[v_1\to v_2]$ is called \emph{composite} if
there exists a directed path from $v_1$ to $v_2$ which does not use
$\varepsilon$. Removing a composite edge does not change the
$e$-function, since the edge is already implied by the induced partial
order. If $\beta_{\varepsilon}=1$ for every $\varepsilon\in E_G$, we write
$\boldsymbol{\beta}=\mathbf{1}$. For the $\omega$-function, we have the following. 
\begin{corollary}
\label{cor:removable-edge}
Let $\varepsilon=[v_1\to v_2]$ be a directed edge of a finite mixed
quiver $G$, and suppose that $G_D$ contains a directed path from
$v_1$ to $v_2$ which does not use $\varepsilon$.  When all directed
edge parameters are specialized to
$\boldsymbol\beta=\mathbf1$ (with the undirected parameters arbitrary),
removing $\varepsilon$ does not change the $\omega$-function:
\[
\left.\omega(G)\right|_{\boldsymbol\beta=\mathbf1}
=
\left.\omega(G_{\setminus\varepsilon})
\right|_{\boldsymbol\beta=\mathbf1}.
\]
\end{corollary}

\begin{proof}
Let $G_0$ be the connected component containing $\varepsilon$.  The
alternative directed path shows that deleting $\varepsilon$ does not
change the mixed connected components, so multiplicativity reduces the
claim to connected $G$.  Theorem~\ref{thm:omega-two-vertex-contraction}
at $\boldsymbol\beta=\mathbf1$ gives
\[
\omega_c(G)+\omega_c(G_{\bar\varepsilon})
+\omega_c(G\cocont\varepsilon)
=\omega_c(G_{\setminus\varepsilon}).
\]
The directed cores of $G_{\bar\varepsilon}$ and
$G\cocont\varepsilon$ contain a directed cycle, possibly a tadpole,
unless they are disconnected.  In either case their connected
coefficients vanish at $\boldsymbol\beta=\mathbf1$ by
Theorem~\ref{thm:mixed-convolution-factorization} and
Proposition~\ref{prop:omega-cyclic-vanishing}.  The result follows.
\end{proof}

\section{Differential equations} 
\label{sec:differential-equations}

The two-vertex contraction rule yields first-order differential
equations for $\omega_c$; their consequences include mixed derivative
formulas, multiaffinity, and a uniqueness characterization.

\subsection{Differential equations for mixed quivers}

For a specified subset $F\subseteq E_G$, let $G_{\setminus F}$ denote
the mixed quiver obtained by deleting the edges in $F$ while retaining
all vertices.

\begin{proposition}[Edge-deletion differential equations]
\label{prop:omega-PDE}
Let $G$ be a finite mixed quiver with independent parameters
$\beta_\varepsilon$ on directed edges and $\gamma_u$ on undirected
edges.  Then
\begin{align}
\label{eq:omega-PDE}
    2\frac{\partial}{\partial\beta_\varepsilon}\omega_c(G)
    &=\omega_c(G_{\setminus\varepsilon}),
    &&\varepsilon\in E_D(G),\\
\label{eq:omega-PDE-undirected}
    \frac{\partial}{\partial\gamma_u}\omega_c(G)
    &=\omega_c(G_{\setminus u}),
    &&u\in E_U(G).
\end{align}
\end{proposition}

\begin{proof}
Let $\varepsilon$ be directed.  If it is not a tadpole,
differentiate the mixed two-vertex contraction rule,
Theorem~\ref{thm:omega-two-vertex-contraction}, with respect to
$\beta_\varepsilon$.  The reverse edge carries the independent
parameter $\beta_{\bar\varepsilon}$, while the deleted and contracted
graphs do not contain $\beta_\varepsilon$; hence
\eqref{eq:omega-PDE} follows.  If $\varepsilon$ is a directed tadpole,
the same identity follows directly from
Lemma~\ref{lem:omega-tadpole-factorization}.

For an undirected edge $u$, Theorem~\ref{thm:mixed-convolution-factorization}
gives
\[
\omega_c(G)
=\gamma_u
\left(\prod_{v\in E_U(G)\setminus\{u\}}\gamma_v\right)
\omega_c(G_D).
\]
The factor after $\gamma_u$ is exactly
$\omega_c(G_{\setminus u})$ by the same theorem, proving
\eqref{eq:omega-PDE-undirected}.  This also covers the case in which
deleting $u$ disconnects the mixed graph, since then the common
directed-core coefficient vanishes.
\end{proof}

\begin{example}
Consider the four-vertex, three-loop graph $G$ in
Fig.~\ref{fig:refined-beta3}. Its connected 
$\omega$-function is
\begin{align}
\label{max-4-omega}
\begin{split}
\omega_c(G)
={}&
-\frac{1}{48}
(\beta_{13}+\beta_{24}+\beta_{23})
-\frac{1}{48}
(\beta_{12}\beta_{14}\beta_{24}
+\beta_{13}\beta_{14}\beta_{34})
\\
&-\frac{1}{96}
\bigl(
3\beta_{13}\beta_{24}\beta_{14}
+\beta_{12}\beta_{23}\beta_{34}
+\beta_{13}\beta_{23}\beta_{24}
+\beta_{12}\beta_{14}\beta_{34}
\\
&\hspace{1.1cm}
+\beta_{13}\beta_{24}\beta_{34}
+\beta_{12}\beta_{14}\beta_{23}
+\beta_{12}\beta_{13}\beta_{24}
+\beta_{14}\beta_{23}\beta_{34}
\\
&\hspace{1.1cm}
+\beta_{13}\beta_{23}\beta_{14}
+\beta_{12}\beta_{24}\beta_{34}
+\beta_{14}\beta_{23}\beta_{24}
+\beta_{12}\beta_{13}\beta_{34}
\bigr).
\end{split}
\end{align}

\begin{figure}[htbp]
    \centering
    \includegraphics[width=0.5\linewidth]{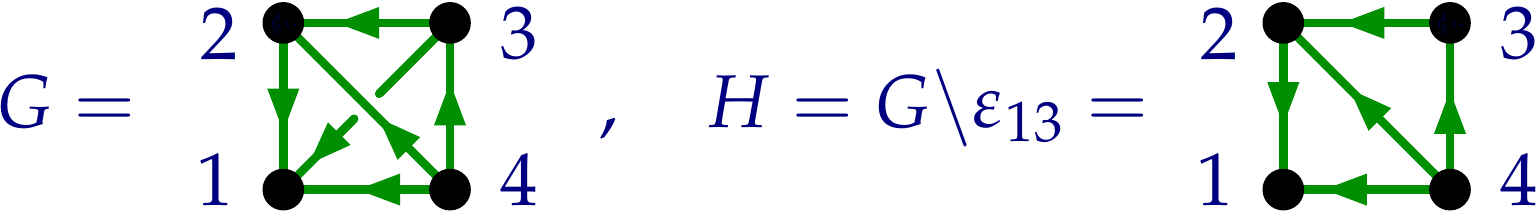}
    \caption{A graph with 4 vertices and 3 loops, and one of its subgraphs. The edge $j \to i$ is denoted by $\varepsilon_{ij}$ and $\beta_{ij} = \beta_{\varepsilon_{ij}}$.}
    \label{fig:refined-beta3}
\end{figure}

Let $H=G_{\setminus\varepsilon_{13}}$.
Direct differentiation of \eqref{max-4-omega} gives
\begin{align}
\begin{split}
2\frac{\partial}{\partial\beta_{13}}\omega(G)
={}&
-\frac{1}{48}
\bigl(
2
+3\beta_{14}\beta_{24}
+2\beta_{14}\beta_{34}
+\beta_{14}\beta_{23}
+\beta_{12}\beta_{24}
+\beta_{23}\beta_{24}
+\beta_{12}\beta_{34}
+\beta_{24}\beta_{34}
\bigr).
\end{split}
\end{align}
This agrees with $\omega_c(H)$. 
\end{example}

\subsection{Mixed derivatives and edge deletions}

\begin{corollary}
\label{cor:omega-multiple-PDE}
Let $A_D\subseteq E_D(G)$ and $A_U\subseteq E_U(G)$ be sets of
distinct edges.  Then, in any order of differentiation,
\begin{equation}
\label{eq:omega-multiple-PDE}
    2^{|A_D|}
    \left(\prod_{\varepsilon\in A_D}
    \frac{\partial}{\partial\beta_\varepsilon}\right)
    \left(\prod_{u\in A_U}
    \frac{\partial}{\partial\gamma_u}\right)
    \omega_c(G)
    =
    \omega_c\bigl(G_{\setminus(A_D\cup A_U)}\bigr).
\end{equation}
\end{corollary}

\begin{proof}
Apply Proposition~\ref{prop:omega-PDE} successively.  Each directed
edge contributes a factor $2$, each undirected edge contributes a
factor $1$, and deletion of distinct edges is independent of order.
\end{proof}

\subsection{Multiaffinity and the loop number}

\begin{corollary}
\label{cor:omega-multiaffine}
For every finite mixed quiver $G$, the polynomial
$\omega_c(G;\boldsymbol\beta,\boldsymbol\gamma)$ is multiaffine in all
edge parameters.
\end{corollary}

\begin{proof}
Differentiate either identity in Proposition~\ref{prop:omega-PDE} once
more with respect to the same edge parameter.  The deleted graph no
longer contains that parameter, so the second derivative vanishes.
\end{proof}
This multiaffinity will be used in the proof of the refined quantum Murua formula in Section \ref{sec:quantum-Murua-reorg}.

For a connected mixed quiver $G$, let
\[
    L(G):=|E_G|-|V_G|+1
\]
be its loop number, and let
$\deg_{\boldsymbol\beta,\boldsymbol\gamma}$ denote total degree in all
edge parameters.

\begin{corollary}
\label{cor:omega-loop-degree}
If $G$ is connected, then
\begin{equation}
\label{eq:omega-loop-degree}
    \deg_{\boldsymbol\beta,\boldsymbol\gamma}\omega_c(G)
    \leq L(G).
\end{equation}
\end{corollary}

\begin{proof}
By multiaffinity there is nothing to prove if no set of $L(G)+1$
distinct edge parameters can occur in a monomial.  Otherwise let
$A\subseteq E_G$ have $L(G)+1$ edges, and put
$A_D=A\cap E_D(G)$, $A_U=A\cap E_U(G)$.  Since deletion retains all
vertices,
\[
|E_{G_{\setminus A}}|=|V_G|-2,
\]
so $G_{\setminus A}$ is disconnected.  Corollary~\ref{cor:omega-multiple-PDE}
therefore shows that the mixed derivative involving all parameters in
$A$ vanishes.  By multiaffinity, any monomial of total degree greater
than $L(G)$ would contain the parameters of such a set $A$ and would
give a nonzero derivative, a contradiction.
\end{proof}

\subsection{Rigidity and characterization of
\texorpdfstring{$\omega_c$}{omega-c}}
\label{subsec:uniqueness-from-contraction}

The following elementary rigidity principle isolates the precise role of
the deletion differential equations.

\begin{lemma}[PDE rigidity]
\label{lem:pde-rigidity}
Let $F_1$ and $F_2$ be polynomial-valued functions on finite quivers,
with
\[
 F_i(G)\in
 \kk[\beta_\varepsilon:\varepsilon\in E_G]
\]
for every quiver $G$.  Assume that:
\begin{enumerate}
\item
$F_1,F_2$ vanish on disconnected graphs;
\item
for every edge $\varepsilon\in E_G$,
\begin{equation}
\label{eq:F-rigidity-PDE}
 2\frac{\partial}{\partial\beta_\varepsilon}F_i(G)
 =F_i(G_{\setminus\varepsilon}),
 \qquad i=1,2;
\end{equation}
\item
for every connected quiver $G$,
\begin{equation}
\label{eq:F-rigidity-boundary}
 F_1(G;\mathbf1)=F_2(G;\mathbf1).
\end{equation}
\end{enumerate}
Then $F_1(G)=F_2(G)$ for every quiver $G$.
\end{lemma}

\begin{proof}
Set $D:=F_1-F_2$ and argue by induction on $|E_G|$.  If $E_G=\emptyset$, then either $G$ is disconnected (including
$G=\emptyset$ by convention), in which case $D(G)=0$ by assumption,
or $G=\bullet$, in which case
\eqref{eq:F-rigidity-boundary} gives $D(G)=0$.

Suppose $|E_G|>0$.  There is nothing to prove when $G$ is disconnected,
so assume that $G$ is connected.  For every edge
$\varepsilon\in E_G$, subtracting the two identities
\eqref{eq:F-rigidity-PDE} gives
\[
 2\frac{\partial}{\partial\beta_\varepsilon}D(G)
 =D(G_{\setminus\varepsilon}).
\]
The graph on the right has fewer edges, so the induction hypothesis
makes it zero.  Hence every first partial derivative of $D(G)$ vanishes.
Since $\kk$ has characteristic zero and $D(G)$ is a polynomial,
$D(G)$ is independent of all edge parameters.  Evaluating at
$\boldsymbol\beta=\mathbf1$ and using
\eqref{eq:F-rigidity-boundary} gives $D(G)=0$.
\end{proof}

This rigidity lemma turns the local contraction identities into a
complete characterization of the directed connected coefficient.  The
mixed extension is then fixed by
Theorem~\ref{thm:mixed-convolution-factorization}.  Recall that
$\beta_{\varepsilon,-}=(\beta_\varepsilon-1)/2$.

\begin{theorem}[Local characterization]
\label{thm:master-uniqueness}
Let $F_1$ and $F_2$ be polynomial-valued functions on finite quivers,
with independent parameters attached to the edges.  Assume that:
\begin{enumerate}
\item
$F_1,F_2$ vanish on disconnected graphs;
\item
both functions satisfy the two-vertex contraction rule;
\item
for every tadpole $\varepsilon\in E_G$,
\begin{equation}
\label{eq:F-master-tadpole}
 F_1(G)=\beta_{\varepsilon,-}F_1(G_{\setminus\varepsilon}),
 \qquad
 F_2(G)=\beta_{\varepsilon,-}F_2(G_{\setminus\varepsilon});
\end{equation}
\item
for every connected quiver $G$,
\begin{equation}
\label{eq:F-master-boundary}
 F_1(G;\mathbf1)=F_2(G;\mathbf1).
\end{equation}
\end{enumerate}
Then $F_1(G)=F_2(G)$ for every quiver $G$.
\end{theorem}

\begin{proof}
It is enough to verify the deletion PDEs in
Lemma~\ref{lem:pde-rigidity}.  Fix an edge
$\varepsilon:j\to i$ of $G$.

If $\varepsilon$ is not a tadpole, put
$H:=G_{\setminus\varepsilon}$, so that
$G=H_{[i\leftarrow j]}$.  Let
$\bar\varepsilon:i\to j$ be the distinguished reverse edge, carrying
its own independent parameter $\beta_{\bar\varepsilon}$.  We view the
contraction identity in the common polynomial ring obtained by adjoining
the two fresh variables $\beta_\varepsilon$ and
$\beta_{\bar\varepsilon}$ to the parameter ring of $H$.  Applying
assumption~(2) to $F_i$ gives
\[
 F_i(H_{[i\leftarrow j]})
 +F_i(H_{[i\to j]})
 +F_i(H_{[i\cdot j]})
 =\frac{\beta_\varepsilon+\beta_{\bar\varepsilon}}{2}F_i(H).
\]
Only the first term on the left depends on $\beta_\varepsilon$.
Differentiation therefore gives
\[
 2\frac{\partial}{\partial\beta_\varepsilon}F_i(G)
 =F_i(G_{\setminus\varepsilon}).
\]

If $\varepsilon$ is a tadpole, assumption~(3) gives
$F_i(G)=\beta_{\varepsilon,-}F_i(G_{\setminus\varepsilon})$, and the
second factor is independent of $\beta_\varepsilon$.  Hence the same
PDE follows from
$\partial\beta_{\varepsilon,-}/\partial\beta_\varepsilon=1/2$.
Thus both $F_1$ and $F_2$ satisfy the hypotheses of
Lemma~\ref{lem:pde-rigidity}, while assumption~(4) supplies its boundary
condition.  The lemma gives $F_1=F_2$.
\end{proof}

\subsection{The boundary value at
\texorpdfstring{$\boldsymbol\beta=\mathbf1$}{beta=1}}
\label{ss:beta-boundary}

We now relate the boundary values of $\omega_c$ at $\beta=1$ with the strict order polynomial of quivers, a notion generalized here from the usual case of posets. For the latter, we refer the reader to \cite{Chan-Pak-2025, richard2011enumerative} for details. 

Throughout this subsection, let $G$ be a nonempty quiver with $|V_G|=n\geq1$.
For a positive integer $r$, let $\operatorname{Sur}_r(G)$
denote the set of surjective maps
$f:V_G\longrightarrow[r]$
satisfying
\begin{equation}
\label{eq:strict-order-preserving-map}
    f(i)<f(j)
    \qquad
    \text{for every edge }j\to i.
\end{equation}
Set
\[
    N_r(G):=|\operatorname{Sur}_r(G)|.
\]
For every positive integer $q$, define the {\it strict order polynomial} of $G$ by 
\begin{equation}
\label{eq:strict-order-polynomial}
    \chi_G(q)
    :=
    \#\left\{
        f:V_G\longrightarrow[q]
        \,\middle|\,
        f(i)<f(j)
        \text{ for every edge }j\to i
    \right\}.
\end{equation}
The polynomial is often denoted by $\Omega^0_G(q)$ for a poset $G$ \cite{Chan-Pak-2025}. We avoid this notation since $\Omega$ has been reserved as the master function for $\omega_c$ in Theorem~\ref{thm:general-master-formula}. As noted in Remark~\ref{rk:quiver-vs-poset-e}, if $G$ is acyclic, its oriented
edges induce a partial order on $V_G$, and the linear extensions and
strict order polynomial of $G$ are precisely those of the underlying
reachability poset $P_G$.

As in the case of posets, every map counted by $\chi_G(q)$ has an image of some cardinality
$r$. Identifying its image increasingly with $[r]$ gives an element
of $\operatorname{Sur}_r(G)$. Hence
\begin{equation}
\label{eq:strict-order-polynomial-surjections}
    \chi_G(q)
    =
    \sum_{r=1}^n
    \binom{q}{r}N_r(G).
\end{equation} 
We write
\begin{equation}
\label{eq:chi-prime-definition}
    \chi_G'(0)
    :=
    \left.
    \frac{d}{dq}\chi_G(q)
    \right|_{q=0}.
\end{equation}
This is the coefficient of the linear term in $\chi_G$.  For posets it is
closely related to $P$-partitions and quasisymmetric functions; one has
the classical formula \cite{richard2011enumerative}
\[
    \chi'_G(0)=\frac{1}{n}\sum_{\sigma\in\operatorname{LE}(G)}
    \frac{(-1)^{n-1-k(\sigma)}}{\binom{n-1}{k(\sigma)}},
\]
where
\begin{equation}
\label{eq:descent-number-master}
    k(\sigma):=\#\{a\in\{1,\ldots,n-1\}\mid \sigma_a>\sigma_{a+1}\}
\end{equation}
is the number of descents of $\sigma$.  Recent work on related linear
coefficients includes \cite{FMP26,Kah26,LL26}.  Lemma~\ref{lem:omega-beta-one}
below gives a convolutional, hence Hopf-algebraic, interpretation.  We
prove it by elementary combinatorics; Proposition~\ref{prop:qsym-master}
provides the quasisymmetric interpretation.

Since
$\left.
    \frac{d}{dq}\binom{q}{r}
    \right|_{q=0}
    =
    \frac{(-1)^{r-1}}{r}$,
we obtain
\begin{equation}
\label{eq:strict-order-polynomial-derivative}
    \chi_G'(0)
    =
    \sum_{r=1}^n
    \frac{(-1)^{r-1}}{r}N_r(G).
\end{equation}
If $G$ contains a directed cycle, including a tadpole, then
\eqref{eq:strict-order-preserving-map} cannot hold along that cycle.
Consequently,
\begin{equation}
\label{eq:chi-cyclic-zero}
    \chi_G(q)=0,
    \quad
    \chi_G'(0)=0, \quad \text{if $G$ contains a directed cycle.}
\end{equation}

\begin{lemma}
\label{lem:omega-beta-one}
For every quiver $G$ with $V_G\neq\emptyset$,
$\omega_c(G;\mathbf1)=\chi_G'(0)$.
\end{lemma}

\begin{proof}
Set
\[
    \mu(G):=\chi_G'(0)
    =\sum_{r=1}^{|V_G|}\frac{(-1)^{r-1}}{r}N_r(G).
\]
For connected acyclic quivers, we first establish the convolution
relation satisfied by $\omega_c$ at $\boldsymbol\beta=\mathbf1$.  If $G$ is
connected, then $G\cocont{K}$ is connected for every cut-subquiver
$K\preceq G$, so the convolution relation for $\omega$ may be written with
$\omega_c$ in the second factor.

Let $G$ be connected and acyclic, and put
$n:=|V_G|$. Consider $\sum_{K\preceq G}
    e(K;\mathbf1)\,\mu(G\cocont{K})$.
Every cut-subquiver $K\preceq G$ is acyclic and has the same vertex set
as $G$, so $e(K;\mathbf1)=\frac{\phi(K)}{n!}$.
After expanding $\mu(G\cocont{K})$, a contribution is specified by a
cut-subquiver $K\preceq G$ together with a surjective strictly
order-preserving map
$\bar f:V_{G\cocont{K}}\to[r]$. Let
$\pi_K:V_G\to V_{G\cocont{K}}$ be the contraction map and set
$f:=\bar f\circ\pi_K$. If $j\to i$ belongs to $K$, then
$f(i)=f(j)$. If it does not belong to $K$, its image is an oriented
edge of $G\cocont{K}$, possibly a tadpole; the latter is excluded by
the strict order-preserving property of $\bar f$, and hence
$f(i)<f(j)$. Thus
\[
    f(i)\leq f(j)
    \qquad
    \text{for every }j\to i\in E_G.
\]

This construction gives a bijection between such pairs $(K,\bar f)$ and
surjective maps $f:V_G\to[r]$ satisfying the above weak
order-preserving condition. Indeed, for such an $f$, define the
cut-subquiver $K_f\preceq G$ by
\[
    E_{K_f}
    :=
    \left\{
        \varepsilon:j\to i\in E_G
        \,\middle|\,
        f(i)=f(j)
    \right\}.
\]
Starting from $(K,\bar f)$, every edge of $K$ belongs to $K_f$.
Conversely, if $j\to i$ belonged to $K_f\setminus K$, then its image in
$G\cocont{K}$ would either join two vertices with the same $\bar f$-value
or be a tadpole, contradicting the strict order-preserving property of
$\bar f$. Hence $K=K_f$.

Conversely, let $f:V_G\to[r]$ be surjective and weakly
order-preserving, and define $K_f$ as above. Since $f$ is constant along
every edge of $K_f$, it is constant on every connected component of
$K_f$ and hence descends to a surjective map
$\bar f:V_{G\cocont{K_f}}\to[r]$. If an edge $j\to i$ survives in
$G\cocont{K_f}$, then $f(i)\neq f(j)$; together with
$f(i)\leq f(j)$, this gives $f(i)<f(j)$. Thus $\bar f$ is strictly
order-preserving, and the two constructions are inverse.
For such an $f$, put $B_a:=f^{-1}(a)$ and $s_a:=|B_a|$ for
$a=1,\ldots,r$. Then
$    K_f=G[B_1]\sqcup\cdots\sqcup G[B_r]$.
By multiplicativity of $e$,
\[
    e(K_f;\mathbf1)
    =
    \prod_{a=1}^r
    \frac{\phi(G[B_a])}{s_a!}.
\]
Choose a linear extension of each $G[B_a]$ and concatenate the resulting
words in the order 
\[B_1,B_2,\ldots,B_r\].
Since $f$ is weakly
order-preserving, the resulting word is a linear extension of $G$.
Conversely, fix a linear extension $\tau$ of $G$ and a composition
$s_1+\cdots+s_r=n$ with $s_a\geq1$. Divide $\tau$ into consecutive
blocks of lengths $s_1,\ldots,s_r$ and let $f$ take the value $a$ on the
$a$-th block. Then $f$ is surjective and weakly order-preserving, and
the restriction of $\tau$ to each block is a linear extension of the
corresponding induced graph $G[B_a]$.
Therefore, for each fixed linear extension $\tau$ of $G$, the total
coefficient contributed to the convolution sum is
\[
    A_n
    :=
    \sum_{r=1}^n\frac{(-1)^{r-1}}{r}
    \sum_{\substack{s_1+\cdots+s_r=n\\ s_a\geq1}}
    \frac{1}{s_1!\cdots s_r!}.
\]
Equivalently, coefficient extraction from
$\log(1+(e^x-1))=x$ gives
$A_n=[x^n]x$, hence $A_1=1$ and $A_n=0$ for $n\geq2$.
Summing over the
$\phi(G)$ linear extensions of $G$, we obtain
\[
    \sum_{K\preceq G}
    e(K;\mathbf1)\,\mu(G\cocont{K})
    =
    \phi(G)A_n
    =
    \epsilon(G)
\]
for every connected acyclic graph $G$.

We now compare this relation with the convolution identity for
$\omega_c(\,\cdot\,;\mathbf1)$.  For connected acyclic quivers we argue
by induction on the number of edges.  The edgeless case is
$G=\bullet$, where $\mu(\bullet)=1=\omega_c(\bullet)$.  For a nonempty
cut-subquiver $K\preceq G$, the quotient $G\cocont K$ is connected and
has strictly fewer edges than $G$.  If this quotient contains a directed
cycle, then both $\mu(G\cocont K)$ and
$\omega_c(G\cocont K;\mathbf1)$ vanish by
\eqref{eq:chi-cyclic-zero} and
Proposition~\ref{prop:omega-cyclic-vanishing}; otherwise they agree by
the induction hypothesis.  The edgeless cut-subquiver contributes
$\mu(G)$ in the relation above and $\omega_c(G;\mathbf1)$ in the
convolution identity for $\omega_c$.  Comparing the two relations gives
$\mu(G)=\omega_c(G;\mathbf1)$.

If instead $G$ is connected and contains a directed cycle, then
\eqref{eq:chi-cyclic-zero} gives $\mu(G)=0$, while
Proposition~\ref{prop:omega-cyclic-vanishing} gives
$\omega_c(G;\mathbf1)=\omega(G;\mathbf1)=0$.  Hence the equality holds
for every connected nonempty quiver.

Finally, suppose that
$G=G_1\sqcup\cdots\sqcup G_s$ is disconnected, with $s\geq2$ and
$V_{G_a}\neq\emptyset$ for every $a$. Then
$\chi_G(q)=\prod_{a=1}^s\chi_{G_a}(q)$. Since
$\chi_{G_a}(0)=0$ for every $a$, we have $\chi_G'(0)=0$, which agrees
with $\omega_c(G;\mathbf1)=0$. Therefore
$\omega_c(G;\mathbf1)=\chi_G'(0)$ for every quiver $G$ with
$V_G\neq\emptyset$.
\end{proof}

\begin{remark}
\label{rk:quiver-vs-poset-o}
At $\boldsymbol\beta=\mathbf1$, the tadpole contributions vanish.
Hence Lemma~\ref{lem:omega-beta-one} reduces to the strict order
polynomial of the reachability poset in the acyclic case and to zero
for quivers containing a directed cycle.
\end{remark}

The following corollary gives the boundary value of $\omega_c$ on a finite mixed quiver. 
\begin{corollary} 
\label{cor:mixed-omega-beta-one}
Let $G$ be a finite mixed quiver with nonempty vertex set. Then
\begin{equation}
\label{eq:mixed-omega-beta-one}
\omega_c(G;\mathbf1,\boldsymbol\gamma)
=
\left(\prod_{u\in E_U(G)}\gamma_u\right)
\chi_{G_D}'(0).
\end{equation}
\end{corollary}

\begin{proof}
Combine Theorem~\ref{thm:mixed-convolution-factorization} with
Lemma~\ref{lem:omega-beta-one} applied to the directed core.
\end{proof}

\section{Master formulas for the \texorpdfstring{$\omega$}{omega}-function}
\label{sec:master-formulas}

We derive the ordered-set-partition master formula for arbitrary
finite mixed quivers and its descent-number permutation form for
acyclic quivers.

\subsection{General and acyclic master formulas}
\label{subsec:general-special-master-formulas}

\subsubsection{The general master function}
Let $G$ be a finite mixed quiver with nonempty vertex set $V_G$,
$|V_G|=n$. Denote by $\operatorname{OSP}(V_G)$ the set of ordered set
partitions of $V_G$.  For
$\mathcal B=(B_1|\cdots|B_\ell)$ and $v\in V_G$, let
$b_{\mathcal B}(v)=a$ if $v\in B_a$.  For a directed edge
$\varepsilon:j\to i$, define
\begin{equation}
\label{eq:s-B-master}
    s_{\mathcal B}(\varepsilon)
    :=
    \begin{cases}
        +,& b_{\mathcal B}(i)<b_{\mathcal B}(j),\\[1mm]
        -,& b_{\mathcal B}(i)\geq b_{\mathcal B}(j).
    \end{cases}
\end{equation}
Thus a directed tadpole always has sign $-$.  Set
\begin{equation}
\label{eq:MGB-master}
    M_G(\mathcal B)
    :=
    \left(\prod_{u\in E_U(G)}\gamma_u\right)
    \prod_{\varepsilon\in E_D(G)}
    \beta_{\varepsilon,s_{\mathcal B}(\varepsilon)},
\end{equation}
and $c_\ell:=(-1)^{\ell-1}/\ell$.

\begin{definition}[General master function]
\label{def:general-master-function}
For a finite mixed quiver $G$ with $n\geq1$ vertices, define
\begin{equation}
\label{eq:general-master-function}
    \Omega_n(G)
    :=
    \sum_{\mathcal B\in\operatorname{OSP}(V_G)}
    c_{|\mathcal B|}M_G(\mathcal B).
\end{equation}
\end{definition}
\begin{example}
Let $r\geq0$, and suppose that $G$ consists of a single vertex carrying directed tadpoles $\varepsilon_1,\ldots,\varepsilon_r$. There is only one ordered set partition of $V_G$, consisting of a single block. Since $c_1=1$ and every tadpole has sign $-$, we obtain
\begin{equation}
\label{eq:Omega1-master}
    \Omega_1(G)
    =
    \prod_{a=1}^r\beta_{\varepsilon_a,-}.
\end{equation}
For $r=0$, the graph is $G=\bullet$, so the product is empty and $\Omega_1(\bullet)=1$.
\end{example}
\begin{example}
Let $V_G=\{1,2\}$ and let $r_{12},r_{21}\geq0$. Suppose that $G$ has edges
$\varepsilon_{12}^{(a)}:2\to1$, $1\leq a\leq r_{12}$, and
$\varepsilon_{21}^{(b)}:1\to2$, $1\leq b\leq r_{21}$, and no tadpoles. The three ordered set partitions of $V_G$ are $(\{1,2\})$, $(\{1\}|\{2\})$, and $(\{2\}|\{1\})$. The one-block partition contributes the product of the minus factors. For $(\{1\}|\{2\})$, the edges $2\to1$ have sign $+$ and the edges $1\to2$ have sign $-$, whereas for $(\{2\}|\{1\})$ the signs are reversed. Since $c_1=1$ and $c_2=-1/2$, we obtain
\begin{equation}
\label{eq:Omega2-master}
\begin{split}
    \Omega_2(G)
    =
    \prod_{a=1}^{r_{12}}\beta_{\varepsilon_{12}^{(a)},-}
    \prod_{b=1}^{r_{21}}\beta_{\varepsilon_{21}^{(b)},-}
    -\frac12\left(
    \prod_{a=1}^{r_{12}}\beta_{\varepsilon_{12}^{(a)},+}
    \prod_{b=1}^{r_{21}}\beta_{\varepsilon_{21}^{(b)},-}
    +
    \prod_{a=1}^{r_{12}}\beta_{\varepsilon_{12}^{(a)},-}
    \prod_{b=1}^{r_{21}}\beta_{\varepsilon_{21}^{(b)},+}
    \right).
\end{split}
\end{equation}
In particular, if $r_{12}=r_{21}=0$, then $G=\bullet\sqcup\bullet$ and
$\Omega_2(G)=1-\frac12-\frac12=0$.
\end{example}

\begin{theorem}[General master formula]
\label{thm:general-master-formula}
Let $G$ be a finite mixed quiver with $n\geq1$ vertices. Then
\begin{equation}
\label{eq:general-master-equals-omega}
    \omega_c(G)=\Omega_n(G).
\end{equation}
\end{theorem}

We will prove Theorem~\ref{thm:general-master-formula} in Section~\ref{subsec:proof-general-master-formula}.

\subsubsection{The acyclic master function}
\label{sss:acyclic-master}
We now state the simpler formula for acyclic graphs.  Its proof will
be obtained later as a specialization of the general master formula in Section~\ref{subsec:proof-special-master-formula}.

Let $G$ be an acyclic quiver with $n$ vertices. Since $G$ is
acyclic, we may label its vertices as
$V_G=\{1,\ldots,n\}$
so that every edge is of the form
\begin{equation}
\label{eq:acyclic-topological-orientation}
    j\longrightarrow i,
    \qquad i<j.
\end{equation}

For a permutation $\sigma=(\sigma_1,\ldots,\sigma_n)\in S_n$,
let $p_\sigma(i):=\sigma^{-1}(i)$
be the position of $i$ in $\sigma$ and let $k(\sigma)$ be the number of descents defined in \eqref{eq:descent-number-master}.
For an edge $\varepsilon:j\longrightarrow i$,
$i<j$,
define
\begin{equation}
\label{eq:s-sigma-master}
    s_\sigma(\varepsilon)
    :=
    \begin{cases}
        +,
        & p_\sigma(i)<p_\sigma(j),\\[2mm]
        -,
        & p_\sigma(i)>p_\sigma(j).
    \end{cases}
\end{equation}
Also, define
\begin{equation}
\label{eq:d-nk-master}
    d_{n,k}
    :=
    \frac{(-1)^{n-1-k}}
    {n\binom{n-1}{k}},
    \qquad
    0\le k\le n-1.
\end{equation}

Recall that a quiver is called acyclic if it contains no directed cycles, including tadpoles.

\begin{definition}[Acyclic master function]
\label{def:special-master-function}
For an acyclic quiver $G$ labeled as above, define
\begin{equation}
\label{eq:special-master-function}
    \overline{\Omega}_n(G)
    :=
    \sum_{\sigma\in S_n}
    d_{n,k(\sigma)}
    \prod_{\varepsilon\in E_G}
    \beta_{\varepsilon,s_\sigma(\varepsilon)}.
\end{equation}
\end{definition}

\begin{theorem}[Acyclic master formula]
\label{thm:special-master-formula}
Let $G$ be an acyclic quiver with $n\geq1$ vertices. Then
\begin{equation}
\label{eq:special-master-equals-omega}
    \omega_c(G)
    =
    \overline{\Omega}_n(G).
\end{equation}
\end{theorem}

We will prove Theorem~\ref{thm:special-master-formula} in Section~\ref{subsec:proof-special-master-formula}.

\begin{example}
Let $G$ be an acyclic quiver with vertex set
$V_G=\{1,2\}$, labeled so that every edge is oriented from $2$ to
$1$. Suppose that $G$ has $m\geq0$ parallel edges
$\varepsilon_a:2\to1,
    \qquad 1\leq a\leq m$.
There are two permutations in $S_2$. For
$\sigma=(1,2)$, we have $k(\sigma)=0$ and
$s_\sigma(\varepsilon_a)=+$ for every $a$, whereas for
$\sigma=(2,1)$, we have $k(\sigma)=1$ and
$s_\sigma(\varepsilon_a)=-$ for every $a$. Since
$d_{2,0}=-\frac12$,
$d_{2,1}=\frac12$,
Definition~\ref{def:special-master-function} gives
\begin{equation}
\label{eq:acyclic-master-n2}
    \overline{\Omega}_2(G)
    =
    -\frac12
    \left(
        \prod_{a=1}^{m}\beta_{\varepsilon_a,+}
        -
        \prod_{a=1}^{m}\beta_{\varepsilon_a,-}
    \right).
\end{equation}
If $m=0$, then $G=\bullet\sqcup\bullet$, and the two empty products
are equal to $1$, so $\overline{\Omega}_2(G)=0$, as expected.
\end{example}

We record a parity property which follows directly from the acyclic master formula.

\begin{corollary}[Parity under $\boldsymbol\beta$-reversal]
\label{cor:beta-reversal}
Let $G$ be a connected acyclic quiver. Then
\begin{equation}
\label{beta-flip}
    \omega_c(G;-\boldsymbol\beta)
    =
    (-1)^{L(G)}
    \omega_c(G;\boldsymbol\beta),
\end{equation}
where $L(G)=|E_G|-|V_G|+1$, and $-\boldsymbol\beta$ denotes the simultaneous replacement
$\beta_\varepsilon\mapsto-\beta_\varepsilon$ for all
$\varepsilon\in E_G$.
\end{corollary}
\begin{proof}
Set $n=|V_G|$ and $m=|E_G|$. For
$\sigma=(\sigma_1,\ldots,\sigma_n)\in S_n$, let
$\sigma^{\mathrm{rev}}=(\sigma_n,\ldots,\sigma_1)$. Then
\[
    s_{\sigma^{\mathrm{rev}}}(\varepsilon)
    =
    -s_\sigma(\varepsilon),
    \qquad
    k(\sigma^{\mathrm{rev}})
    =
    n-1-k(\sigma),
\]
where the first identity means that the two signs are opposite. Since
$(-\beta_\varepsilon)_+
    =
    -\beta_{\varepsilon,-}$, 
$(-\beta_\varepsilon)_-
    =
    -\beta_{\varepsilon,+}$,
the acyclic master formula gives
\[
    \omega_c(G;-\boldsymbol\beta)
    =
    (-1)^m
    \sum_{\sigma\in S_n}
    d_{n,k(\sigma)}
    \prod_{\varepsilon\in E_G}
    \beta_{\varepsilon,s_{\sigma^{\mathrm{rev}}}(\varepsilon)}.
\]
Reindexing by $\sigma\mapsto\sigma^{\mathrm{rev}}$ and using
$
    d_{n,n-1-k}=(-1)^{n-1}d_{n,k}$,
we obtain
\[
    \omega_c(G;-\boldsymbol\beta)
    =
    (-1)^{m+n-1}\omega_c(G;\boldsymbol\beta).
\]
Since
$m+n-1\equiv m-n+1=L(G)\pmod 2$, this proves
\eqref{beta-flip}.
\end{proof}

\subsection{Weighted vertex splitting and the general master formula}
\label{subsec:proof-general-master-formula}

We now explain the structural origin of
Definition~\ref{def:general-master-function}.  The unlocalized
mixed-quiver algebra carries a second Hopf structure, graded by vertices
and encoding ordered vertex splittings.  In this subsection, each directed
edge $\varepsilon$ is decorated by a scalar $\beta_\varepsilon\in\kk$
and each undirected edge $u$ by a scalar $\gamma_u\in\kk$; induced
subquivers inherit these decorations.  The corresponding formal polynomial
identities are obtained by universal specialization.

For $A\subseteq V_G$, let $G|_A$ denote the induced mixed subquiver.  If
$V_G=A\sqcup B$ is an ordered decomposition, define
\begin{equation}
\label{eq:vertex-crossing-weight}
\begin{split}
 w_G(A,B):=
 \prod_{\substack{\varepsilon:j\to i\in E_D(G)\\ i\in A,\ j\in B}}
 \beta_{\varepsilon,+}
 \prod_{\substack{\varepsilon:j\to i\in E_D(G)\\ i\in B,\ j\in A}}
 \beta_{\varepsilon,-}
 \prod_{\substack{u=\{i,j\}\in E_U(G)\\
                    i\in A,\ j\in B\ \text{or}\ i\in B,\ j\in A}}
 \gamma_u.
\end{split}
\end{equation}
Edges internal to one block do not contribute to this factor.  Put
\begin{equation}
\label{eq:vertex-coproduct}
 \Delta_V(G)
 :=
 \sum_{A\sqcup B=V_G}
 w_G(A,B)\,G|_A\otimes G|_B,
\end{equation}
and let $\epsilon_V(G)=1$ for the empty graph and $0$ otherwise.

\begin{theorem}[Weighted vertex-splitting Hopf algebra]
\label{thm:weighted-vertex-hopf}
With disjoint union as product, $\Delta_V$ and $\epsilon_V$ make
$\kk\mathcal G$ into a connected graded commutative Hopf algebra
$\mathcal H_V$, graded by $|V_G|$.
\end{theorem}

\begin{proof}
A vertex decomposition of a disjoint union splits uniquely into vertex
decompositions of its two factors, and there are no crossing edges
between the two factors.  Hence $w_{G_1\sqcup G_2}$ factors and
$\Delta_V(G_1\sqcup G_2)=\Delta_V(G_1)\Delta_V(G_2)$.  For the
counit identities, $\epsilon_V(G|_A)$ is nonzero only for $A=\emptyset$,
while $\epsilon_V(G|_B)$ is nonzero only for $B=\emptyset$; the two
surviving terms have weight~$1$.

For coassociativity, fix an ordered decomposition
$V_G=A\sqcup B\sqcup C$, allowing empty blocks.  In
$(\Delta_V\otimes\id)\Delta_V(G)$ the coefficient of
$G|_A\otimes G|_B\otimes G|_C$ is
\[
 w_G(A\sqcup B,C)\,w_{G|_{A\sqcup B}}(A,B),
\]
whereas in $(\id\otimes\Delta_V)\Delta_V(G)$ it is
\[
 w_G(A,B\sqcup C)\,w_{G|_{B\sqcup C}}(B,C).
\]
An edge whose endpoints lie in two different blocks is separated exactly
once on either side.  A directed edge contributes
$\beta_{\varepsilon,+}$ or $\beta_{\varepsilon,-}$ according only
to the relative order of the blocks containing its head and tail, while an
undirected edge contributes $\gamma_u$.  Edges internal to one block,
including all tadpoles, contribute on neither side.  Thus the two
coefficients agree edge by edge.  Since the vertex grading is preserved
and the degree-zero part consists only of the empty graph, the bialgebra is
connected graded and therefore has an antipode.
\end{proof}

The algebra $\mathcal H_V$ is deliberately not localized at $\bullet$.
Indeed,
\[
 \Delta_V(\bullet)=\bullet\otimes1+1\otimes\bullet,
\]
so $\bullet$ is primitive for vertex splitting, in contrast with its
grouplike behavior for the contraction coproduct.

Define
\begin{equation}
\label{eq:vertex-zeta}
 \zeta_V(G)
 :=
 \left(\prod_{\varepsilon\in E_D(G)}\beta_{\varepsilon,-}\right)
 \left(\prod_{u\in E_U(G)}\gamma_u\right).
\end{equation}
The empty products give $\zeta_V(\emptyset)=1$, and the edge sets of a
disjoint union are disjoint.  Hence
$\zeta_V(G_1\sqcup G_2)=\zeta_V(G_1)\zeta_V(G_2)$, so $\zeta_V$
is a character of $\mathcal H_V$.
Let $\star_V$ denote convolution for $\Delta_V$.  Since $\mathcal H_V$
is connected graded, the convolution logarithm
\begin{equation}
\label{eq:vertex-log}
 \log_{\star_V}\zeta_V
 :=
 \sum_{r\ge1}\frac{(-1)^{r-1}}{r}
 (\zeta_V-\epsilon_V)^{\star_V r}
\end{equation}
is finite on every graph.

For a map $f:V_G\to\mathbb Z_{>0}$ and a directed edge
$\varepsilon:j\to i$, put
\[
 s_f(\varepsilon)=
 \begin{cases}
 +,&f(i)<f(j),\\
 -,&f(i)\ge f(j),
 \end{cases}
\]
and define
\begin{equation}
\label{eq:KG-map-short}
 \mathcal K_G(X;\boldsymbol\beta,\boldsymbol\gamma)
 :=
 \sum_{f:V_G\to\mathbb Z_{>0}}
 \left(\prod_{u\in E_U(G)}\gamma_u\right)
 \left(\prod_{\varepsilon\in E_D(G)}
 \beta_{\varepsilon,s_f(\varepsilon)}\right)
 \prod_{v\in V_G}x_{f(v)}.
\end{equation}

Let $\QSym$ be the Hopf algebra of quasisymmetric functions and let
$$\zeta_Q: \QSym \to \mathbb{Q}, \quad \zeta_Q(F):=F(1,0,0,\ldots)$$
be the canonical character of $\QSym$.
Applying the universal property of the terminal combinatorial Hopf algebra $(\QSym,\zeta_Q)$~\cite{ABS06}, we obtain the following interpretations of $\mathcal K_G$ and $\Omega_{V_G}$.
\begin{proposition}
\label{prop:qsym-master}
Let $\zeta_Q(F)=F(1,0,0,\ldots)$ be the canonical character of $\QSym$.
The canonical morphism of combinatorial Hopf algebras
\[
 \Psi_V:(\mathcal H_V,\zeta_V)\longrightarrow(\QSym,\zeta_Q)
\]
satisfies
\begin{equation}
\label{eq:canonical-qsym-map}
 \Psi_V(G)=\mathcal K_G.
\end{equation}
Moreover, if $V_G\neq\emptyset$, then
\begin{equation}
\label{eq:vertex-log-master}
 \bigl(\log_{\star_V}\zeta_V\bigr)(G)
 =
 (\log_*\zeta_Q)(\mathcal K_G)
 =
 [q] \,\ps_q\mathcal K_G
 =
 \Omega_{|V_G|}(G),
\end{equation}
where $\ps_q(M_\alpha)=\binom{q}{\ell(\alpha)}$.  On the empty graph,
both convolution logarithms vanish.
\end{proposition}

\begin{proof}
Fix an ordered set partition
$\mathcal B=(B_1|\cdots|B_r)$ of $V_G$ and consider the corresponding
summand of the iterated vertex coproduct.  If the endpoints of a directed
edge lie in different blocks, that edge is separated exactly once and
contributes $\beta_{\varepsilon,+}$ or
$\beta_{\varepsilon,-}$ according to the relative order of those
blocks.  If its endpoints lie in the same block, the edge survives in one
tensor factor and $\zeta_V$ contributes $\beta_{\varepsilon,-}$.
Likewise, an undirected edge contributes $\gamma_u$ whether it is
separated or remains internal.  Tadpoles are necessarily internal and are
covered by the latter cases.  Thus, after applying $\zeta_V$ to the tensor
factors, the ordered fiber contributes exactly $M_G(\mathcal B)$ from
\eqref{eq:MGB-master}.

Consequently, for every composition $\alpha$ of $|V_G|$, the
coefficient of $M_\alpha$ in the canonical
Aguiar--Bergeron--Sottile morphism~\cite{ABS06} is the sum of these weights over the
ordered set partitions with block sizes $\alpha$.  Grouping the
colorings in \eqref{eq:KG-map-short} by their ordered fibers gives the
same coefficient.  This proves both that $\mathcal K_G\in\QSym$ and
that $\Psi_V(G)=\mathcal K_G$.

Since $\zeta_Q\circ\Psi_V=\zeta_V$ and $\Psi_V$ is a Hopf morphism,
\[
 (\log_*\zeta_Q)\circ\Psi_V=\log_{\star_V}\zeta_V.
\]
If $\ell(\alpha)=r$, then the deconcatenation coproduct of $\QSym$
and the identity
$\zeta_Q(M_\alpha)=1$ for $r=1$ and $0$ for $r>1$ give directly
\[
 (\log_*\zeta_Q)(M_\alpha)
 =\frac{(-1)^{r-1}}{r}
 =[q]\binom qr
 =[q]\ps_q(M_\alpha).
\]
Applying this to $\mathcal K_G$ and grouping its colorings by ordered
fibers yields the last two equalities in \eqref{eq:vertex-log-master} by
Definition~\ref{def:general-master-function}.
\end{proof}

We now compare the contraction and weighted vertex-splitting descriptions of the connected $\omega$-function.

\begin{theorem}[Two-Hopf descriptions of the connected coefficient]
\label{thm:two-hopf-omega}
For every finite mixed quiver $G$,
\begin{equation}
\label{eq:two-hopf-omega}
 \omega_c(G)
 =
 \bigl(\log_{\star_V}\zeta_V\bigr)(G).
\end{equation}
Thus the same connected graph function is obtained as the connected part of the convolution inverse of $e$ for edge contraction and as the convolution logarithm of $\zeta_V$ for weighted vertex splitting.
\end{theorem}

\begin{proof}
For $G=\emptyset$, both sides of \eqref{eq:two-hopf-omega} vanish, so
assume $V_G\neq\emptyset$.  Definition~\ref{def:general-master-function}
gives
\[
 \Omega_{|V_G|}(G)
 =
 \left(\prod_{u\in E_U(G)}\gamma_u\right)
 \Omega_{|V_G|}(G_D).
\]
Together with Theorem~\ref{thm:mixed-convolution-factorization} and
\eqref{eq:vertex-log-master}, it is therefore enough to prove the
identity for purely directed quivers.

Set $L(G):=(\log_{\star_V}\zeta_V)(G)$.  Since the logarithm of a
character on a connected graded Hopf algebra is an infinitesimal
character, $L$ vanishes on disconnected nonempty graphs, just as
$\omega_c$ does.  Moreover, \eqref{eq:vertex-log-master} gives, for
every directed edge $\varepsilon$,
\begin{equation}
\label{eq:vertex-log-PDE}
 2\frac{\partial}{\partial\beta_\varepsilon}L(G)
 =L(G_{\setminus\varepsilon}),
\end{equation}
because both $\beta_{\varepsilon,+}$ and
$\beta_{\varepsilon,-}$ have derivative $1/2$.
Proposition~\ref{prop:omega-PDE} gives the same deletion PDE for
$\omega_c$.  Finally, at $\boldsymbol\beta=\mathbf1$, only strictly
order-preserving colorings contribute to \eqref{eq:KG-map-short}, so
\[
 L(G;\mathbf1)=\chi_G'(0)=\omega_c(G;\mathbf1)
\]
by Lemma~\ref{lem:omega-beta-one}.  Thus $L$ and $\omega_c$ satisfy all
hypotheses of the PDE rigidity Lemma~\ref{lem:pde-rigidity}, and hence
$L(G)=\omega_c(G)$ for every directed quiver $G$.  The mixed identity
then follows from the first displayed factorization.
\end{proof} 

The equality in~\eqref{eq:two-hopf-omega} also admits a direct
structural explanation. On the unlocalized graph algebra,
define $\delta_c:=\tau\circ\Delta_c$, where $\tau$ is the flip map,
i.e., $\tau(a\otimes b)=b\otimes a$. Thus
\[
 \delta_c(G)=\sum_{H\preceq G}(G\cocont H)\otimes H,
\]
where the second tensor factor is the spanning cut-sub-mixed quiver. Then
\[
 (\Delta_V\otimes\id)\delta_c
 =m_{1,3,24}(\delta_c\otimes\delta_c)\Delta_V,
\]
where
$m_{1,3,24}(x_1\otimes a_1\otimes x_2\otimes a_2)
=x_1\otimes x_2\otimes a_1a_2$.
Indeed, a vertex split of $G\cocont H$ pulls back uniquely to a split
$V_G=A\sqcup B$ for which the selected edges of $H$ are internal to
$A$ or $B$.  The selected set then splits as $H_A\sqcup H_B$, the three resulting graph factors agree on the two sides, and every surviving crossing
edge has the same weight before and after contraction.  Thus the identity holds term by term.  Notice that this statement is made before localizing at
$\bullet$: the contraction coproduct is grouplike on $\bullet$, whereas
$\Delta_V(\bullet)=\bullet\otimes1+1\otimes\bullet$.
\\
\indent 
For a contraction character $a$, define
\[
 (f\act a)(G):=(f\otimes a)\delta_c(G)
 =\sum_{H\preceq G}f(G\cocont H)a(H).
\]
The preceding compatibility implies
\[
 (f\star_V g)\act a=(f\act a)\star_V(g\act a),
 \qquad
 (\log_{\star_V}f)\act a=\log_{\star_V}(f\act a).
\]
Let $\alpha$ be the infinitesimal character with
$\alpha(\bullet)=1$ and $\alpha(G)=0$ on every other basis graph.  For the
characters of the present paper one has
\[
 \zeta_V\act e=\exp_{\star_V}(\alpha).
\]
For completeness, when $G$ has no tadpoles the left-hand side is
\[
 \left(\prod_{u\in E_U(G)}\gamma_u\right)
 \sum_{S\subseteq E_D(G)}
 \frac{|\operatorname{LE}(G_S)|}{|V_G|!}
 \prod_{\varepsilon\in E_D(G)\setminus S}
 \beta_{\varepsilon,-}.
\]
Expanding the product over total vertex orders using
$\beta_{\varepsilon,+}=\beta_{\varepsilon,-}+1$ gives the same expression;
if $G$ has a tadpole, subsets can be paired according to whether that
tadpole is selected, and the two contributions cancel.  Hence the bridge
identity holds for all mixed quivers.

Now let $\omega=e^{*(-1)}$ as in Section~\ref{sec:graph functions}.  Taking the vertex-convolution
logarithm of the bridge gives
\[
 (\log_{\star_V}\zeta_V)\act e=\alpha.
\]
Acting by $\omega$ yields $\log_{\star_V}\zeta_V=\alpha\act\omega$.
The latter is $\omega(G)$ when $G$ is nonempty and connected and is zero
otherwise, because $\alpha(G\cocont H)\neq0$ forces $H=G$ and $G$ to be
connected.  Thus $\log_{\star_V}\zeta_V=\omega_c$,
which gives a second proof of Theorem~\ref{thm:two-hopf-omega}.  The general
character-action mechanism is standard in the theory of bialgebras in
cointeraction; the graph-specific input here is the displayed bridge
identity for the character $e$ used in this paper.

\begin{proof}[Proof of Theorem~\ref{thm:general-master-formula}]
Theorem~\ref{thm:two-hopf-omega} and
\eqref{eq:vertex-log-master} give
$\omega_c(G)=\Omega_{|V_G|}(G)$ for every finite mixed quiver, which is
\eqref{eq:general-master-equals-omega}.
\end{proof}

\subsection{Proof of the acyclic master formula}
\label{subsec:proof-special-master-formula}

Theorem~\ref{thm:special-master-formula} follows immediately from
Theorem~\ref{thm:general-master-formula} and the following proposition.  

\begin{proposition}[Acyclic specialization]
\label{prop:general-reduces-to-special}
Let $G$ be an acyclic quiver with $n\geq1$ vertices, labeled
so that $V_G=\{1,\ldots,n\}$ and every edge $j\to i$ satisfies $i<j$.
Then
\begin{equation}
\label{eq:general-reduces-to-special}
    \Omega_n(G)=\overline{\Omega}_n(G).
\end{equation}
\end{proposition}
\begin{proof}
For
$\mathcal B=(B_1|\cdots|B_\ell)\in\operatorname{OSP}(V_G)$, let
$\sigma(\mathcal B)\in S_n$ be obtained by writing the elements of each
block in decreasing order and then concatenating the blocks in the
order $B_1,\ldots,B_\ell$.

Consider an edge $\varepsilon:j\to i$, where $i<j$. If $i$ and $j$
lie in different blocks, then
$b_{\mathcal B}(i)<b_{\mathcal B}(j)$ if and only if $i$ occurs before
$j$ in $\sigma(\mathcal B)$, so
$s_{\mathcal B}(\varepsilon)=s_{\sigma(\mathcal B)}(\varepsilon)$.
If they lie in the same block, then
$s_{\mathcal B}(\varepsilon)=-$, while the decreasing order within the
block places $j$ before $i$, so again
$s_{\sigma(\mathcal B)}(\varepsilon)=-$. Hence
\[
    M_G(\mathcal B)
    =
    \prod_{\varepsilon\in E_G}
    \beta_{\varepsilon,s_{\sigma(\mathcal B)}(\varepsilon)}.
\]

Fix $\sigma\in S_n$ and set $k=k(\sigma)$. An ordered set
partition $\mathcal B$ satisfies $\sigma(\mathcal B)=\sigma$ if and
only if its block boundaries contain every ascent position
$a$ with $\sigma_a<\sigma_{a+1}$; each descent position may either be
a block boundary or lie inside a block.  Indeed, every block is written
in decreasing order, so an ascent cannot occur inside a block, and
conversely any such choice of cuts decomposes $\sigma$ uniquely into
consecutive decreasing blocks and hence determines a unique ordered set
partition in the fiber of $\sigma$.

There are $k$ descent positions.  If exactly $j$ of them are left
uncut, there are $\binom{k}{j}$ choices.  Since all $n-1-k$ ascent
positions are mandatory cuts, the total number of cuts is $n-1-j$ and
the resulting ordered set partition has $n-j$ blocks.  Moreover, by
the first part of the proof, every ordered set partition in this fiber
has the same edge-weight factor
$\prod_{\varepsilon\in E_G}
    \beta_{\varepsilon,s_\sigma(\varepsilon)}$.
Therefore the total contribution of the fiber
$\{\mathcal B:\sigma(\mathcal B)=\sigma\}$ to $\Omega_n(G)$ is this
factor multiplied by
$\sum_{j=0}^k
    \binom{k}{j}
    \frac{(-1)^{n-j-1}}{n-j}$.
Using
$\frac{1}{n-j}=\int_0^1 t^{n-j-1}\,dt$ and the binomial theorem, the
scalar factor equals
\begin{align*}
 \int_0^1
    \sum_{j=0}^k\binom{k}{j}(-1)^{n-j-1}t^{n-j-1}\,dt
   & =
    (-1)^{n-1-k}
    \int_0^1 t^{n-k-1}(1-t)^k\,dt
   \\
   & =
    (-1)^{n-1-k}\frac{(n-k-1)!\,k!}{n!}
    =
    \frac{(-1)^{n-1-k}}{n\binom{n-1}{k}}
    =
    d_{n,k}.
\end{align*} 
This computation is valid for the boundary cases $k=0$ and $k=n-1$
as well; when $n=1$ it gives $d_{1,0}=1$.  Nothing in the argument
requires $G$ to be simple: parallel edge occurrences have the same
sign determined by their common endpoints and simply contribute their
individual edge factors.

Grouping the ordered set partitions according to
$\sigma(\mathcal B)$ now gives
\[
    \Omega_n(G)
    =
    \sum_{\sigma\in S_n}
    d_{n,k(\sigma)}
    \prod_{\varepsilon\in E_G}
    \beta_{\varepsilon,s_\sigma(\varepsilon)}
    =
    \overline{\Omega}_n(G).
\qedhere \] 
\end{proof}

\subsection{Higher-vertex contraction rules}
\label{subsec:higher-vertex-contraction}

We conclude the section with higher-vertex contraction identities obtained
from Theorem~\ref{thm:general-master-formula}. 

Let $H$ be a finite mixed quiver and let
$I=\{i_1,\ldots,i_k\}\subseteq V_H$ consist of $k\geq2$
pairwise distinct marked vertices.
For each unordered pair $\{i,j\}\subseteq I$, fix a new
distinguished edge with two possible orientations
$\varepsilon_{ij}:j\to i$ and
$\varepsilon_{ji}:i\to j$.
Thus $\bar\varepsilon_{ij}=\varepsilon_{ji}$.
The two orientations carry independent parameters
$\beta_{\varepsilon_{ij}}$ and
$\beta_{\varepsilon_{ji}}$, giving $k(k-1)$ independent
parameters in total.
Only one orientation of each distinguished edge may be
adjoined to a given graph. The distinguished edges are not edges of $H$ and remain
distinct from all background edges, including any parallel
edges with the same orientation.  

For $\sigma\in S_k$, let $H_\sigma$ be the mixed quiver
obtained from $H$ by adjoining the $k-1$ distinguished edges
of the directed chain
\begin{equation}
\label{eq:oriented-higher-chain}
i_{\sigma(k)}
\xrightarrow{\ \varepsilon_{i_{\sigma(k-1)}i_{\sigma(k)}}\ }
i_{\sigma(k-1)}
\longrightarrow\cdots\longrightarrow
i_{\sigma(2)}
\xrightarrow{\ \varepsilon_{i_{\sigma(1)}i_{\sigma(2)}}\ }
i_{\sigma(1)}.
\end{equation}
Thus the edge
$\varepsilon_{i_{\sigma(t)}i_{\sigma(t+1)}}:
i_{\sigma(t+1)}\to i_{\sigma(t)}$
carries the parameter
$\beta_{\varepsilon_{i_{\sigma(t)}i_{\sigma(t+1)}}}$
for $1\le t<k$.

Let $\Pi(I)$ be the partition lattice of $I$, ordered by
refinement. For $\pi\in\Pi(I)$, let $H/\pi$ denote the
mixed quiver obtained from $H$ by identifying the vertices
in each block of $\pi$. All edges of $H$ are retained with
their original parameters and induced incidences. In particular,
edges whose endpoints are identified become tadpoles.

Let $\operatorname{OSP}(I)$ be the set of ordered set partitions
of $I$. For $\eta=(B_1|\cdots|B_q)\in\operatorname{OSP}(I)$,
write $\underline\eta=\{B_1,\ldots,B_q\}$ for the underlying
unordered partition. For $x,y\in I$, write $x\prec_\eta y$
if $x\in B_a$ and $y\in B_b$ for some $a<b$.
Define
\begin{equation}
\label{eq:higher-oriented-F}
\mathcal F_\eta(\boldsymbol\beta)
:=
\sum_{\sigma\in S_k}
\prod_{t=1}^{k-1}
\left(
\beta_{\varepsilon_{i_{\sigma(t)}i_{\sigma(t+1)}},-}
+
\mathbf 1_{\{
i_{\sigma(t)}\prec_\eta i_{\sigma(t+1)}
\}}
\right).
\end{equation}

Let $[B_r]$ denote the vertex of $H/\underline\eta$
obtained by identifying the vertices of $B_r$.
Define $H_\eta^{\mathrm{ord}}$ by adjoining to
$H/\underline\eta$ the directed chain
\[
[B_q]\longrightarrow[B_{q-1}]
\longrightarrow\cdots\longrightarrow[B_1],
\]
with every new edge having parameter $1$.
If $q=1$, set $H_\eta^{\mathrm{ord}}=H/\underline\eta$.

\begin{lemma}
\label{lem:higher-ordered-slices}
Let $\eta=(B_1|\cdots|B_q)\in\operatorname{OSP}(I)$.
Let $S_\eta(H)$ be the sum in the master formula for
$\omega_c(H)$ over all ordered partitions
$\mathcal B=(C_1|\cdots|C_r)$ of $V_H$ such that
the nonempty sets among
$C_1\cap I,\ldots,C_r\cap I$, in this order,
are precisely $B_1,\ldots,B_q$.
Then
\begin{equation}
\label{eq:higher-slice-identity}
S_\eta(H)
=
\omega_c\left(H_\eta^{\mathrm{ord}}\right).
\end{equation}
Consequently, for every $\pi\in\Pi(I)$,
\begin{equation}
\label{eq:higher-quotient-ordered-decomposition}
\omega_c(H/\pi)
=
\sum_{\substack{\eta\in\operatorname{OSP}(I)\\
                 \pi\leq\underline\eta}}
\omega_c\left(H_\eta^{\mathrm{ord}}\right).
\end{equation}
\end{lemma}

\begin{proof}
Let $\mathcal B=(C_1|\cdots|C_r)$ be an ordered
partition of $V_H$ contributing to $S_\eta(H)$.
Identifying the vertices in each $B_j$ gives an
ordered partition of $V(H/\underline\eta)$
in which $[B_1],\ldots,[B_q]$ belong to distinct
blocks appearing in this order.
This correspondence is bijective and preserves
the number of blocks. In the master formula, each new edge
$[B_{j+1}]\to[B_j]$ of $H_\eta^{\mathrm{ord}}$
has parameter $1$, and hence contributes
$\beta_+=1$, $\beta_-=0$.
Thus these edges select exactly the ordered
partitions in which $[B_1],\ldots,[B_q]$
appear in the prescribed order.

All edges of $H$ are retained under the quotient,
with their parameters unchanged.
If an edge becomes a tadpole, its endpoints
already belonged to the same block of
$\mathcal B$, so its weight remains
$\beta_{\varepsilon,-}$ for a directed edge,
or $\gamma_u$ for an undirected edge.
This also applies to parallel edges.
Therefore the corresponding terms in the
master formula have identical weights,
including the coefficient
$(-1)^{r-1}/r$.
This proves \eqref{eq:higher-slice-identity}.

Finally, an ordered partition of $V(H/\pi)$
corresponds to an ordered partition of $V_H$
in which the vertices of each block of $\pi$
belong to the same block.
Its restriction to $I$ is an ordered partition
$\eta$ satisfying $\pi\leq\underline\eta$.
These possibilities are disjoint, and the
corresponding weights agree.
Summing \eqref{eq:higher-slice-identity}
over such $\eta$ gives
\eqref{eq:higher-quotient-ordered-decomposition}.
\end{proof}

\begin{theorem}[Higher-vertex contraction rules]
\label{thm:higher-oriented-edge-contraction}
With the above notation, for any finite mixed quiver $H$,
\begin{equation}
\label{eq:higher-oriented-ordered}
\sum_{\sigma\in S_k}\omega_c(H_\sigma)
=
\sum_{\eta\in\operatorname{OSP}(I)}
\mathcal F_\eta(\boldsymbol\beta)\,
\omega_c\left(H_\eta^{\mathrm{ord}}\right).
\end{equation}

The formula can also be expressed in terms of the quotients
$H/\pi$. For $\rho,\pi\in\Pi(I)$, define
\begin{equation}
\label{eq:higher-oriented-averages}
\begin{aligned}
\overline{\mathcal F}_\rho
:=
\frac{1}{|\rho|!}
\sum_{\substack{\eta\in\operatorname{OSP}(I) \\ \underline\eta=\rho}}
\mathcal F_\eta(\boldsymbol\beta),\quad
C_\pi(\boldsymbol\beta)
:=
\sum_{\rho\leq\pi}
\mu_{\Pi(I)}(\rho,\pi)\,
\overline{\mathcal F}_\rho,
\end{aligned}
\end{equation}
where $\mu_{\Pi(I)}$ is the M\"obius function of
the partition lattice $\Pi(I)$. Then
\begin{equation}
\label{eq:higher-oriented-main}
\begin{aligned}
\sum_{\sigma\in S_k}\omega_c(H_\sigma)
=\sum_{\pi\in\Pi(I)}
C_\pi(\boldsymbol\beta)\,
\omega_c(H/\pi)+
\sum_{\eta\in\operatorname{OSP}(I)}
\left(
\mathcal F_\eta(\boldsymbol\beta)
-\overline{\mathcal F}_{\underline\eta}
\right)
\omega_c\left(H_\eta^{\mathrm{ord}}\right).
\end{aligned}
\end{equation}
\end{theorem}

\begin{proof}
Apply the ordered-set-partition master formula
(Theorem~\ref{thm:general-master-formula})
to each $H_\sigma$.
Fix an ordered partition $\mathcal B$ of $V_H$,
and let $\eta$ be the ordered partition of $I$
obtained by intersecting the blocks of $\mathcal B$
with $I$ and omitting empty intersections.

For $1\leq t<k$, the distinguished edge
$\varepsilon_{i_{\sigma(t)}i_{\sigma(t+1)}}:
i_{\sigma(t+1)}\to i_{\sigma(t)}$
contributes the factor
$\beta_{\varepsilon_{i_{\sigma(t)}i_{\sigma(t+1)}},-}
+
\mathbf 1_{\{
i_{\sigma(t)}\prec_\eta i_{\sigma(t+1)}
\}}$.
The factors corresponding to the edges of $H$
are independent of $\sigma$.
Summing the products of the distinguished-edge
factors over $\sigma\in S_k$ gives
$\mathcal F_\eta(\boldsymbol\beta)$. By Lemma~\ref{lem:higher-ordered-slices},
the summands of the master formula inducing $\eta$
sum to $\omega_c\left(H_\eta^{\mathrm{ord}}\right)$.
This proves \eqref{eq:higher-oriented-ordered}.

By M\"obius inversion on $\Pi(I)$, we have
$\overline{\mathcal F}_\rho
=
\sum_{\pi\leq\rho}C_\pi(\boldsymbol\beta)$.
Using \eqref{eq:higher-quotient-ordered-decomposition},
we obtain
\[
\begin{aligned}
\sum_{\pi\in\Pi(I)}
C_\pi(\boldsymbol\beta)\,\omega_c(H/\pi)=
\sum_{\eta\in\operatorname{OSP}(I)}
\left(
\sum_{\pi\leq\underline\eta}
C_\pi(\boldsymbol\beta)
\right)
\omega_c(H_\eta^{\mathrm{ord}})=
\sum_{\eta\in\operatorname{OSP}(I)}
\overline{\mathcal F}_{\underline\eta}\,
\omega_c(H_\eta^{\mathrm{ord}}).
\end{aligned}
\]
Consequently, the right-hand side of
\eqref{eq:higher-oriented-main} equals
$\sum_{\eta\in\operatorname{OSP}(I)}
\mathcal F_\eta(\boldsymbol\beta)\,
\omega_c(H_\eta^{\mathrm{ord}})$,
which is precisely the right-hand side of
\eqref{eq:higher-oriented-ordered}.
This proves \eqref{eq:higher-oriented-main}.

\end{proof}

\begin{remark}[The two-vertex case]
For $k=2$, the correction sum in
\eqref{eq:higher-oriented-main} vanishes and we recover
\begin{equation}
\label{eq:higher-two-oriented}
\omega_c(H+\varepsilon_{ij})+
\omega_c(H+\varepsilon_{ji})
=\frac{\beta_{\varepsilon_{ij}}+\beta_{\varepsilon_{ji}}}{2}
\omega_c(H)-\omega_c(H/[ij]).
\end{equation}
Equivalently, if $G=H+\varepsilon_{ij}$, this is
Theorem~\ref{thm:omega-two-vertex-contraction},
with $G\cocont{\varepsilon_{ij}}=H/[ij]$.
\end{remark}

\begin{corollary}[Three-vertex contraction rule]
\label{cor:higher-three-oriented}
Let $I=\{1,2,3\}$, and write
$\varepsilon:=\varepsilon_{12}:2\to1$,
$\delta:=\varepsilon_{23}:3\to2$,
$\gamma:=\varepsilon_{31}:1\to3$.
For $a\in\{\varepsilon,\delta,\gamma\}$, let $\bar a$ denote
its reverse and set
\[
b_a:=\frac{\beta_a+\beta_{\bar a}}2,\qquad
d_a:=\frac{\beta_a-\beta_{\bar a}}2,\qquad
D_a(H):=\omega_c(H+a)-\omega_c(H+\bar a),
\]
where $H+a$ is obtained by adjoining the distinguished edge
$a$ with its parameter $\beta_a$. Define
\[
\begin{gathered}
c_\varepsilon:=\frac{d_\delta+d_\gamma}{2},\qquad
c_\delta:=\frac{d_\gamma+d_\varepsilon}{2},\qquad
c_\gamma:=\frac{d_\varepsilon+d_\delta}{2},\\
C_\triangle:=\frac12\bigl(
-1+b_\varepsilon b_\delta+b_\delta b_\gamma
+b_\gamma b_\varepsilon
-d_\varepsilon d_\delta-d_\delta d_\gamma
-d_\gamma d_\varepsilon\bigr).
\end{gathered}
\]
Then
\begin{equation}
\label{eq:higher-three-vertex}
\begin{aligned}
\sum_{\sigma\in S_3}\omega_c(H_\sigma)
& = C_\triangle\,\omega_c(H)+2\omega_c(H/[123])-\frac12\bigl(
(b_\delta+b_\gamma)\omega_c(H/[12]) \\
&  +(b_\varepsilon+b_\gamma)\omega_c(H/[23]) +(b_\varepsilon+b_\delta)\omega_c(H/[13])\bigr)+c_\varepsilon D_\varepsilon(H)
+c_\delta D_\delta(H)+c_\gamma D_\gamma(H).
\end{aligned}
\end{equation}
\end{corollary}

\begin{proof}
We apply Theorem~\ref{thm:higher-oriented-edge-contraction}
with $k=3$. Put $z_a=\beta_{a,-}$ and
$T_a=\omega_c(H+a^{(1)})$, where $a^{(1)}$ denotes
a new edge with parameter $1$.
For each cyclic permutation $(a,b,c)$ of
$(\varepsilon,\delta,\gamma)$, set
$p_a=z_b+z_c$ and
$p_{\bar a}=z_{\bar b}+z_{\bar c}$.
Let $I_a(\eta)$ indicate whether the head of $a$
precedes its tail in $\eta$. Expanding \eqref{eq:higher-oriented-F} gives
\[
\mathcal F_\eta
=R+\mathbf 1_{\{\underline\eta=\hat0\}}
+\sum_{a\in\{\varepsilon,\delta,\gamma\}}
\bigl(p_aI_a(\eta)+p_{\bar a}I_{\bar a}(\eta)\bigr),
\]
where
$R=\sum_{\mathrm{cyc}(a,b,c)}
(z_az_b+z_{\bar a}z_{\bar b})$.
By Lemma~\ref{lem:higher-ordered-slices},
$\sum_\eta I_a(\eta)\omega_c(H_\eta^{\mathrm{ord}})=T_a$.
Moreover, M\"obius inversion applied to
\eqref{eq:higher-quotient-ordered-decomposition}
gives
$\sum_{\underline\eta=\hat0}
\omega_c(H_\eta^{\mathrm{ord}})
=\omega_c(H)-\sum_{i<j}\omega_c(H/[ij])
+2\omega_c(H/[123])$.
Hence \eqref{eq:higher-oriented-ordered} becomes
\[
\begin{aligned}
\sum_{\sigma\in S_3}\omega_c(H_\sigma)
=(1+R)\omega_c(H)
-\sum_{i<j}\omega_c(H/[ij])
+2\omega_c(H/[123])+\sum_{a\in\{\varepsilon,\delta,\gamma\}}
(p_aT_a+p_{\bar a}T_{\bar a}).
\end{aligned}
\]

For $a=\varepsilon_{ij}$, the two-vertex contraction
rule \eqref{eq:higher-two-oriented} at unit parameters
gives
$T_a+T_{\bar a}=\omega_c(H)-\omega_c(H/[ij])$.
The master formula also yields
$\omega_c(H+a)=T_a+z_a\omega_c(H)$, so
$T_a-T_{\bar a}=D_a(H)-d_a\omega_c(H)$.
Since $(p_a+p_{\bar a})/2=(b_b+b_c-2)/2$
and $(p_a-p_{\bar a})/2=c_a$, we obtain
\[
\begin{aligned}
p_aT_a+p_{\bar a}T_{\bar a}
=\frac{b_b+b_c-2}{2}
\bigl(\omega_c(H)-\omega_c(H/[ij])\bigr)+c_a\bigl(D_a(H)-d_a\omega_c(H)\bigr).
\end{aligned}
\]

Finally,
$R=\frac12\sum_{\mathrm{cyc}(a,b,c)}
((b_a-1)(b_b-1)+d_ad_b)$
and
$\sum_a c_ad_a
=d_\varepsilon d_\delta+d_\delta d_\gamma
+d_\gamma d_\varepsilon$.
Substituting these identities and collecting terms
gives \eqref{eq:higher-three-vertex}.
\end{proof}
 
\begin{corollary}[Contraction across distinct components]
\label{cor:higher-component-tree}
Suppose $H=H_1\sqcup\cdots\sqcup H_k$, where each $H_j$
is connected and $i_j\in V(H_j)$. For arbitrary
independent oriented-edge parameters,
\begin{equation}
\label{eq:higher-component-tree}
\sum_{\sigma\in S_k}\omega_c(H_\sigma)
=(-1)^{k-1}(k-1)!\,\omega_c(H/I).
\end{equation}
In particular, when $k=3$,
\begin{equation}
\label{eq:higher-three-tree}
\sum_{\sigma\in S_3}
\omega_c\left(H_{[\sigma(1)\leftarrow\sigma(2)
                         \leftarrow\sigma(3)]}\right)
=2\omega_c(H/[123]),
\end{equation}
where $H_{[\sigma(1)\leftarrow\sigma(2)
\leftarrow\sigma(3)]}$ denotes the corresponding chain graph.
This includes the case in which the $H_j$ are trees.
\end{corollary}

\begin{proof}
Each new edge of $H_\sigma$ is a bridge between
components of $H$. Deleting it disconnects $H_\sigma$.
Hence, by Proposition~\ref{prop:omega-PDE} and the
vanishing of $\omega_c$ on disconnected graphs,
$\omega_c(H_\sigma)$ is independent of all newly adjoined
edge parameters. We may therefore set these parameters to $1$.
At this specialization, $\mathcal F_\eta$ in
\eqref{eq:higher-oriented-F} equals $1$ if every block
of $\eta$ is a singleton, and $0$ otherwise.

Let $\hat0$ and $\hat1$ denote the discrete and
indiscrete partitions of $I$, respectively, and let
$\mu_{\Pi(I)}$ be the M\"obius function of the partition
lattice $\Pi(I)$. By M\"obius inversion
(see \cite[Chapter~3]{richard2011enumerative}),
\[
\sum_{\pi\leq\rho}\mu_{\Pi(I)}(\hat0,\pi)
=
\begin{cases}
1, & \rho=\hat0,\\
0, & \rho\ne\hat0.
\end{cases}
\]
Hence, the main theorem and
\eqref{eq:higher-quotient-ordered-decomposition}
give
\[
\sum_{\sigma\in S_k}\omega_c(H_\sigma)
=
\sum_{\pi\in\Pi(I)}
\mu_{\Pi(I)}(\hat0,\pi)\,\omega_c(H/\pi).
\]

For $\pi\ne\hat1$, the quotient $H/\pi$ has
$|\pi|\geq2$ connected components, and hence
$\omega_c(H/\pi)=0$. Since $H/I=H/\hat1$ and
\[
\mu_{\Pi(I)}(\hat0,\hat1)
=(-1)^{k-1}(k-1)!
\]
(see \cite[Section~3.10]{richard2011enumerative}),
the result follows.
\end{proof}

\begin{remark} 
\label{rem:higher-three-rooted}
For distinct $i,j,k\in\{1,2,3\}$, let
$H_{[i\leftarrow j,k]}$ be obtained by adjoining both
$j\to i$ and $k\to i$, and let
$H_{[j\cdot k\leftarrow i]}$ be obtained by identifying $j,k$
and adjoining an edge from $i$ to the merged vertex.
When $H$ is a disjoint union of three trees, the two-vertex
tree rule and \eqref{eq:higher-three-tree} yield
\begin{equation}
\label{eq:higher-rooted-cyclic}
 \sum_{\operatorname{cyc}(i,j,k)}
 \left(\omega_c(H_{[i\leftarrow j,k]})
       +\omega_c(H_{[j\cdot k\leftarrow i]})\right)
 =-\omega_c(H/[123]).
\end{equation}
To verify it, write $F_i$ for the coefficient of the fork
$j\to i\leftarrow k$ and $C_i$ for the coefficient obtained by
identifying $j,k$ and adjoining $i\to[j\cdot k]$, where
$\{i,j,k\}=I$.  Apply the two-vertex tree contraction rule
\eqref{eq:omega_contraction-tree} separately to either edge of
the fork. Reversing one of its edges produces a directed chain;
contracting that edge produces the corresponding $C_j$ or $C_k$.
When these six identities are added, each fork $F_i$ and each
$C_i$ occurs twice, and the six directed chains occur once.
Thus
\[
 2\sum_i F_i+\sum_{\sigma\in S_3}\omega_c(H_\sigma)
 =-2\sum_i C_i.
\]
By \eqref{eq:higher-three-tree}, the middle sum is
$2\omega_c(H/[123])$. Dividing by $2$ yields
\eqref{eq:higher-rooted-cyclic}.

In the notation of \cite{Chartier_2010}, suppose that the three
components $H_1,H_2,H_3$ correspond to rooted trees $u,v,w$,
respectively. Then
$H_{[1\leftarrow 2,3]}$, $H_{[2\leftarrow 3,1]}$, and
$H_{[3\leftarrow 1,2]}$ correspond to
$u\circ(vw)$, $v\circ(wu)$, and $w\circ(uv)$, respectively,
while $H_{[2\cdot 3\leftarrow 1]}$,
$H_{[3\cdot 1\leftarrow 2]}$, and
$H_{[1\cdot 2\leftarrow 3]}$ correspond to
$(v\times w)\circ u$, $(w\times u)\circ v$, and
$(u\times v)\circ w$, respectively.
Finally, $H/[123]$ corresponds to $u\times v\times w$.
Thus, when restricted to rooted trees,
\eqref{eq:higher-rooted-cyclic} recovers precisely
the relation in \cite[Proposition~4.3]{Chartier_2010}.
\end{remark}

\section{Magnus expansion and operator products}
\label{sec:operator-products-Hopf}

We now develop the operator-product realization.  Time ordering
fixes edge directions, operator ordering selects $W_{ij}$ versus
$W_{ji}$, and the Wick expansion produces parallel RB edges.  In this
section $\beta_\varepsilon=\pm1$; independent edge parameters enter in
Section~\ref{sec:RB-BW-bases-reorg}.  The independent operator-product
derivation of the explicit acyclic formula is given later in
Section~\ref{sec:operator-acyclic-master}.  All series are treated
formally.

\subsection{Magnus expansion and Magnusian} \label{subsec:Magnus expansion}

The Magnus expansion arises from a first-order differential equation
on a Lie group $\mathcal G$ with Lie algebra $\mathfrak g$.  
Given a time-dependent element $A(t) \in \mathfrak{g}$, 
the differential equation is given by 
\begin{align} \label{eq:magnus_Lie_group_DE}
    \left[\frac{d}{dt}U(t)\right] U(t)^{-1} = A(t) \,, 
    \quad 
    U(t) \in \mathcal{G} \,,
    \quad 
    U(0) = I \,, 
\end{align}
where $I \in \mathcal{G}$ is the identity element. 
The Magnusian $\Omega(t) \in  \mathfrak{g}$ is defined by  
\begin{align} \label{eq:magnus_Lie_group_Lie_algebra_Exp}
    U(t) = \exp\left[ \Omega(t) \right] \,,
    \quad 
    \Omega(t) \in \mathfrak{g} \,,
    \quad 
    \Omega(0) = 0 \,. 
\end{align}
We use the bracket notation for the Lie algebra: $A,B \in \mathfrak{g} \,\Longrightarrow \, [A, B] \in \mathfrak{g}$. 
The Lie algebra structure allows us to express the Magnus series
entirely in terms of Lie brackets. Writing
\begin{equation}
    \Omega(t)=\sum_{n=1}^{\infty}\Omega_n(t),
\end{equation}
the first two terms are
\begin{align}
\begin{split}
    \Omega_1(t)
    &=
    \int_0^t A(t_1)\,dt_1,
    \\
    \Omega_2(t)
    &=
    \frac12
    \int_0^t\int_0^t
    [A(t_1),A(t_2)]
    \theta(t_1-t_2)\,dt_1dt_2,
\end{split}
\end{align}
where $\theta$ is the Heaviside step function,
\[
    \theta(x)
    =
    \begin{cases}
        1,&x>0,\\
        0,&x<0.
    \end{cases}
\]
Its value at $x=0$ is irrelevant for the integrals considered here.
 
(\ref{eq:magnus_Lie_group_DE})-(\ref{eq:magnus_Lie_group_Lie_algebra_Exp})
are written intrinsically on the Lie group~$G$.
For the operator-product realization used below, we realize
$\mathfrak g$ as a Lie subalgebra of the commutator Lie algebra
associated with the associative algebra
$(\mathcal A,\star,\mathbf 1)$, with bracket
\begin{equation}
\label{eq:star-commutator}
    [A,B]=A\star B-B\star A.
\end{equation}
Under this realization, the group exponential is represented formally
by the $\star$-exponential.  
Thus every term in the Magnus expansion is a linear combination of
noncommutative operator words.  We write
\[
    A_i:=A(t_i),
    \qquad
    A_{(\sigma)}
    :=
    A_{\sigma_1}\star\cdots\star A_{\sigma_n},
    \qquad
    \sigma=(\sigma_1,\ldots,\sigma_n)\in S_n.
\]
Then the second and third Magnus terms take the form
\begin{align}
    \Omega_2
    &=
    \int
    \frac12\bigl(A_{(12)}-A_{(21)}\bigr)
    \theta_{12}\,dt_{12}, \label{eq:Omega2-operator-words}
    \\
    \Omega_3
    &=
    \int
    \left[
        \frac13\bigl(A_{(123)}+A_{(321)}\bigr)
        -
        \frac16\bigl(
            A_{(132)}+A_{(213)}+A_{(231)}+A_{(312)}
        \bigr)
    \right]
    \theta_{123}\,dt_{123}.  \label{eq:Omega3-operator-words}
\end{align}
Here
\[
    \theta_{i_1\cdots i_r}
    :=
    \prod_{a=1}^{r-1}
    \theta(t_{i_a}-t_{i_{a+1}}),
    \qquad
    dt_{i_1\cdots i_r}
    :=
    dt_{i_1}\cdots dt_{i_r}.
\]
The coefficients are the $d_{n,k}$ in (\ref{eq:d-nk-master}) (up to a sign convention); see, for example, \cite{Mielnik_1970,Burum1981,Strichartz_1987,Ebrahimi_Fard_2014,arnal2017,bandiera2017,bandiera2020}. 

In the rest of the section, we will explain how a fixed operator word becomes a graph.

\subsection{Star products and ordered contractions}
\label{subsec:star-products-ordered-contractions}

We use a simple oscillator realization to extract the graph data from
operator products.
Let $a$ and $a^\dagger$ satisfy
\[
    [a,a^\dagger]=1.
\] 
Recall that normal ordering is defined such that for any function of $a$, $a^\dagger$, 
defined as a (possibly infinite) power series, 
all $a^\dagger$'s are pushed to the left and $a$'s are pushed to the right. For example, 
\begin{align}
    : a a^\dagger a^2: \;\; = \;\; a^\dagger a^3 \,.
\end{align}
For normally ordered functions, define the star product by
\begin{equation}
\label{eq:operator-star-product}
    F\star G
    :=
    F
    \exp\left(
        \frac{\overleftarrow{\partial}}{\partial a}
        \frac{\overrightarrow{\partial}}{\partial a^\dagger}
    \right)
    G.
\end{equation}
Here the arrows indicate the factors on which the derivatives act:
$\frac{\overleftarrow{\partial}}{\partial a}$ acts on the function to its left, whereas
$\frac{\overrightarrow{\partial}}{\partial a^\dagger}$ acts on the function to its
right. Equivalently,
\[
    F\star G
    =
    \sum_{r\geq0}
    \frac{1}{r!}
    \frac{\partial^r F}{\partial a^r}
    \frac{\partial^r G}{\partial (a^\dagger)^r}.
\]
This is the symbol-level realization of the operator product in Section~\ref{subsec:Magnus expansion}.
This star product arises from
the deformation quantization of K\"ahler manifolds
given in refs.~\cite{Karabegov1996,BordemannWaldmann1997},
the use of which to explore the quantum Magnusian
is due to ref.~\cite{Kim:2025ebl}.

The diagrammatics of the quantum Magnusian is independent of the
choice of vertex symbol. A particularly convenient choice is the
time-dependent normally ordered symbol
\begin{equation}
\label{eq:operator-vertex}
    A(t)
    := \;\;
    :\exp\left(
        f(t)a^\dagger+\bar f(t)a
    \right): \;\;
    =
    e^{f(t)a^\dagger}e^{\bar f(t)a}.
\end{equation}
For time variables $t_1,\ldots,t_n$, we use the shorthand
\[
    A_i:=A(t_i),
    \qquad
    f_i:=f(t_i),
    \qquad
    \bar f_i:=\bar f(t_i).
\]
We define the ordered contraction by
\begin{equation}
\label{eq:operator-Wij}
    W_{ij}
    :=
    \bar f_i f_j
    =
    \bar f(t_i)f(t_j).
\end{equation}
In general, $W_{ij}\neq W_{ji}$.

We take
$(\mathcal A,\star,\mathbf1)$ to be the algebra of normally ordered
oscillator symbols, with multiplication given by
\eqref{eq:operator-star-product}.  This provides a concrete realization
of the associative algebra introduced in
Section~\ref{subsec:Magnus expansion}. We work in a suitable completion of
$\mathcal A$ whenever formal exponentials such as $A(t)$ are involved.

Before stating the Wick rule, we fix one piece of notation. On the
right-hand side of Lemma~\ref{lem:exponentiated-Wick}, juxtaposition
denotes the ordinary commuting product of normally ordered symbols.
Thus
\begin{align} \label{eq:commuting product of Ai}
\prod_{i=1}^n A_i
    =
    \exp\left(
        \sum_{i=1}^n f_i\,a^\dagger
    \right)
    \exp\left(
        \sum_{i=1}^n \bar f_i\,a
    \right).
\end{align} 
\begin{lemma}[Exponentiated Wick rule]
\label{lem:exponentiated-Wick}
For every permutation $\sigma=(\sigma_1,\ldots,\sigma_n)\in S_n$,
one has
\begin{equation}
\label{eq:exponentiated-Wick}
\begin{split}
    A_{\sigma_1}
    \star\cdots\star
    A_{\sigma_n}
    &=
    \exp\left(
        \sum_{1\leq a<b\leq n}
        W_{\sigma_a\sigma_b}
    \right)
    \prod_{i=1}^n A_i.
\end{split}
\end{equation}
\end{lemma}

\begin{proof}
For two vertices, \eqref{eq:operator-star-product} gives
\[
    A_i\star A_j
    =
    e^{\bar f_i f_j}A_i A_j
    =
    e^{W_{ij}}A_iA_j.
\]
Suppose the formula holds for $n-1$ vertices.  Multiplication on the
right by $A_{\sigma_n}$ produces the cross-contraction
$W_{\sigma_a\sigma_n}$ for every $a<n$.  Therefore
\begin{align*}
    A_{\sigma_1}\star\cdots\star A_{\sigma_n}
    &=
    \exp\left(
        \sum_{1\leq a<b<n}
        W_{\sigma_a\sigma_b}
    \right)
    \exp\left(
        \sum_{a=1}^{n-1}
        W_{\sigma_a\sigma_n}
    \right)
    \prod_{i=1}^n A_i
=
    \exp\left(
        \sum_{1\leq a<b\leq n}
        W_{\sigma_a\sigma_b}
    \right)
    \prod_{i=1}^n A_i.
\end{align*}
The result follows by induction.
\end{proof}

Thus, once an operator ordering has been fixed, the ordered
contraction between two labels is determined completely:
if $i$ occurs before $j$, the contraction is $W_{ij}$, whereas if
$j$ occurs before $i$, it is $W_{ji}$.
Moreover, expanding the exponential in
\eqref{eq:exponentiated-Wick} turns powers of the same contraction into
parallel edges.

\subsection{Red-Blue graphs and their \texorpdfstring{$\omega$}{omega}-function}
\label{subsec:double-color-graphs-reorg}

The operator ordering records the relative order of the operators in
the noncommutative product, whereas the time ordering records the
relative order of their time variables. Thus, for a pair of vertices
$i,j$, the factor $\theta_{ij}=\theta(t_i-t_j)$ imposes
$t_i>t_j$, while $W_{ij}$ and $W_{ji}$ correspond to the two possible
relative operator orderings.
As in refs.~\cite{Kim:2025ebl,Guo:2026xaw}, 
we introduce two types of colored edges (blue/red) 
to distinguish the pairing between the operator ordering and the time ordering. 

Our convention in this paper for colored directed edges reads 
\begin{align}
\label{eq:operator-color-convention}
    W_{ij} \,\theta_{ij} \; \rightarrow \; \raisebox{-2.5mm}{\includegraphics[width=1.2cm]{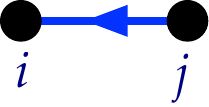}} \,,
    \qquad 
    W_{ji} \,\theta_{ij} \; \rightarrow \; \raisebox{-2.5mm}{\includegraphics[width=1.2cm]{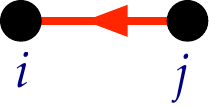}} \,. 
\end{align}
That is, $\text{blue}_{ij} = W_{ij}\theta_{ij}$, $\text{red}_{ij}=W_{ji}\theta_{ij}$ (we use the convention that an edge $j \to i$ is denoted by $\varepsilon_{ij}$ for some $\varepsilon$). Our red and blue labels are interchanged relative to those of
\cite{Kim:2025ebl,Guo:2026xaw}; see Remark~\ref{rem:RB-physical-convention} below.

\eqref{eq:Omega2-operator-words} gives the simplest example:
on the sector $t_1>t_2$, the word $A_1\star A_2$ produces blue
contractions, while $A_2\star A_1$ produces red contractions.  Thus
red and blue graphs are not an additional graph-theoretic input; they
are the Wick images of the operator words occurring in the Magnus
expansion.

For a fixed blue contraction, the exponentiated Wick rule produces the
factor $e^{W_{ij}}$. Expanding this exponential gives
\[
    e^{W_{ij}}\theta_{ij}
    =
    \sum_{m\geq0}\frac{W_{ij}^m}{m!}\theta_{ij}.
\]
The term of degree $m$ is represented by $m$ parallel blue edges
between the two vertices. In particular, $m=1$ gives a single edge,
whereas $m\geq2$ gives a banana graph of positive loop number.
Diagrammatically,
\begin{align}
    e^{W_{ij}} \theta_{ij} = \;\raisebox{-2.5mm}{\includegraphics[width=8cm]{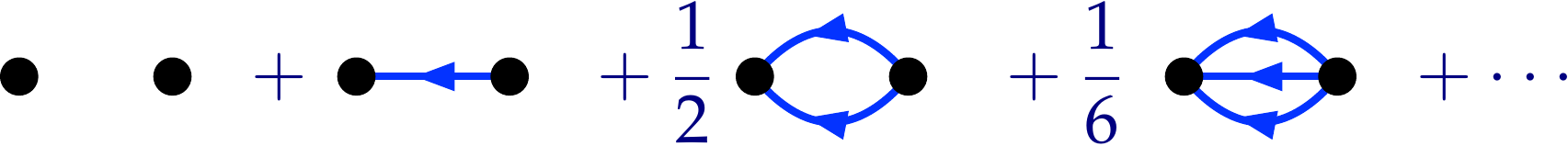}} \; \,.
\end{align}

We now specialize the general \mquiver formalism to quivers
with two edge colors.

\begin{definition}
\label{def:double-color-graph-reorg}
A \emph{red-blue graph}, or \emph{RB graph}, is a quiver
$G=(V_G,E_G,\rho_G)$ together with a coloring
\[
    \kappa_G:E_G\longrightarrow\{r,b\}.
\]
We write
\[
    E_r(G):=\kappa_G^{-1}(r),
    \qquad
    E_b(G):=\kappa_G^{-1}(b),
    \qquad
    b(G):=|E_b(G)|.
\]
An RB graph is called \emph{acyclic} if its underlying quiver is
acyclic.
\end{definition}

Time ordering produces acyclic RB graphs (with possible parallel
edges), although the Hopf functions below are defined on arbitrary RB
graphs.

For an RB graph $G$, we specialize the edge parameters by
\begin{equation}
\label{eq:double-color-parameters-reorg}
    \beta_\varepsilon
    =
    \begin{cases}
        1,&\varepsilon\in E_r(G),\\
        -1,&\varepsilon\in E_b(G).
    \end{cases}
\end{equation}

\begin{definition}[RB $e$- and $\omega$-functions]
\label{def:RB-e-omega-reorg}
For an RB graph $G$, define
\begin{equation}
\label{eq:RB-e-omega-reorg}
    e_{RB}(G)
    :=
    (-1)^{b(G)}
    \left.e(G;\boldsymbol\beta)\right|_{\beta_r=1,\beta_b=-1},
    \qquad
    \omega_{RB}(G)
    :=
    (-1)^{b(G)}
    \left.\omega(G;\boldsymbol\beta)\right|_{\beta_r=1,\beta_b=-1}.
\end{equation}
For every RB graph $G$, define its connected projection by
\begin{equation}
\label{eq:RB-omega-c-reorg}
    \omega_{RB,c}(G):=(-1)^{b(G)}
    \left.\omega_c(G;\boldsymbol\beta)\right|_{\beta_r=1,\beta_b=-1}.
\end{equation}
Thus $\omega_{RB,c}(G)=0$ whenever $G$ is disconnected.
\end{definition}

Cut-sub-\mquivers and contractions inherit the colors of all surviving
edges. Hence, for every $K\preceq G$,
\begin{equation}
\label{eq:RB-blue-splitting-reorg}
    b(G)=b(K)+b(G\cocont K).
\end{equation}

\begin{proposition}
\label{prop:RB-convolution-reorg}
The functions $e_{RB}$ and $\omega_{RB}$ are multiplicative and are
convolution inverses:
\begin{equation}
\label{eq:RB-convolution-reorg}
    e_{RB}*\omega_{RB}
    =
    \omega_{RB}*e_{RB}
    =
    \epsilon.
\end{equation}
\end{proposition}

\begin{proof}
Multiplicativity follows from
$b(G_1\sqcup G_2)=b(G_1)+b(G_2)$ and the multiplicativity of the
general functions $e$ and $\omega$. Moreover, using
\eqref{eq:RB-blue-splitting-reorg},
\[
    (e_{RB}*\omega_{RB})(G)
    =
    (-1)^{b(G)}
    \left.(e*\omega)(G;\boldsymbol\beta)
    \right|_{\beta_r=1,\beta_b=-1}.
\]
If $G$ has at least one edge, the right-hand side is zero. If $G$ is
edgeless, then $b(G)=0$. Thus
\[
    (e_{RB}*\omega_{RB})(G)=\epsilon(G).
\]
The proof of $\omega_{RB}*e_{RB}=\epsilon$ is similar.
\end{proof}

\begin{remark}
\label{rem:RB-physical-convention}
Our red/blue labels are interchanged relative to \cite{Guo:2026xaw}.
The factor $(-1)^{b(G)}$ compensates for the sign of each blue
$\beta_-$ factor at $(\beta_r,\beta_b)=(1,-1)$ and matches the
normalization used there.
\end{remark}

\section{Hopf--operator compatibility in the RB basis}
\label{sec:global-RB-compatibility}

Let $\Omega_{\mathrm{RB}}$ be the connected RB series and $U$ the
Dyson series.  We prove
$\omega_{RB}*\widetilde e=\mathbf b$, where $\widetilde e$ is the
redification ordering character and $\mathbf b$ the pure-blue
character, and deduce
$\exp_\star(\Omega_{\mathrm{RB}})=U$, hence
$\Omega_{\mathrm{RB}}=\Omega$.

For later use, we define the color graph integral of an acyclic RB graph $G$ by
\begin{equation}
\label{eq:I-color}
    I_{\mathrm{color}}(G;t)
:=
\int_{[0,t]^{|V_G|}}
\left(
\prod_{\varepsilon\ {\rm blue}}
W_{ij}\theta_{ij}
\right)
\left(
\prod_{\varepsilon\ {\rm red}}
W_{ji}\theta_{ij}
\right)
\left(
\prod_{v\in V_G}A(t_v)
\right)
\prod_{v\in V_G}dt_v,
\end{equation}
where for each edge the ordered pair $(i,j)$ is fixed by its
time-ordering factor $\theta_{ij}$, equivalently by the edge
$j\to i$, and the product $\prod_{v\in V_G}A(t_v)$ is commutative, see (\ref{eq:commuting product of Ai}).

\subsection{Dyson series as a pure-blue graph series}
\label{subsec:Dyson-pure-blue}

Introduce a formal parameter $\lambda$ recording the homogeneous degree in
$A(t)$ and write
\begin{equation}
\label{eq:operator-Dyson-series}
    U(\lambda;t)
    :=
    \mathbf1+
    \sum_{n\geq1}\lambda^nU_n(t),
    \qquad
    U_n(t)
    :=
    \int_{\Delta_n(t)}
    A(t_1)\star\cdots\star A(t_n)\,
    dt_1\cdots dt_n,
\end{equation}
where
\begin{equation}
\label{eq:operator-simplex}
    \Delta_n(t)
    :=
    \{(t_1,\ldots,t_n)\mid 0<t_n<\cdots<t_1<t\}.
\end{equation}
Likewise set
\begin{equation}
\label{eq:lambda-Magnus-series}
    \Omega(\lambda;t)
    :=
    \sum_{n\geq1}\lambda^n\Omega_n(t).
\end{equation}
In the completed algebra $\mathcal A[[\lambda]]$,
\begin{equation}
\label{U-vs-Omega_star}
    U(\lambda;t)
    =
    \exp_\star\bigl(\Omega(\lambda;t)\bigr),
    \qquad
    \Omega(\lambda;t)
    =
    \log_\star U(\lambda;t),
\end{equation}
where $\exp_\star$ and $\log_\star$ are understood as formal power series.

If $G$ is purely blue, write $I_{\mathrm{blue}}(G;t):=I_{\mathrm{color}}(G;t)$ (see \eqref{eq:I-color}).
Equivalently,
\begin{equation}
\label{eq:I-blue-cube}
    I_{\mathrm{blue}}(G;t)
    =
    \int_{[0,t]^{|V_G|}}
    \left(
        \prod_{\varepsilon\in E_G}
        W_\varepsilon\theta_\varepsilon
    \right)
    \left(\prod_{v\in V_G}A(t_v)\right)
    \prod_{v\in V_G}dt_v.
\end{equation}

For a directed colored multigraph $G$, let
$\operatorname{Aut}(G)$ be the group of pairs of bijections
\[
    \phi_V:V_G\longrightarrow V_G,
    \qquad
    \phi_E:E_G\longrightarrow E_G,
\]
such that, for every edge $\varepsilon:j\to i$,
the edge $\phi_E(\varepsilon)$ is directed from
$\phi_V(j)$ to $\phi_V(i)$ and has the same color as $\varepsilon$.
The \emph{symmetry factor} of $G$ is
\begin{equation}
\label{eq:graph-symmetry-factor}
    \sigma(G):=|\operatorname{Aut}(G)|.
\end{equation} 

\begin{proposition}[Pure-blue graph expansion of the Dyson series]
\label{prop:Dyson-blue-graph-expansion}
The Dyson series has the graph expansion
\begin{equation}
\label{eq:Dyson-blue-graph-expansion}
    U(\lambda;t)
    =
    \sum_G
    \frac{\lambda^{|V_G|}}{\sigma(G)}
    I_{\mathrm{blue}}(G;t),
\end{equation}
where the sum runs over isomorphism classes of finite purely blue acyclic 
quivers, including the empty graph.
\end{proposition}

\begin{proof}
At degree $n$, the Dyson integral is taken over the simplex
$0<t_n<\cdots<t_1<t$ with operator word $A_1\star\cdots\star A_n$.
Lemma~\ref{lem:exponentiated-Wick} gives
\[
    A_1\star\cdots\star A_n
    =
    \exp\left(\sum_{1\leq i<j\leq n}W_{ij}\right)
    \prod_{i=1}^nA_i.
\]
The operator ordering agrees with the time ordering; hence every contraction is
$W_{ij}\theta_{ij}$ and is blue.  Expanding the exponential chooses a
multiplicity $m_{ij}\geq0$ for each ordered pair $i<j$ and therefore produces
all labeled purely blue acyclic quivers compatible with the
total order $1<\cdots<n$, with the factor $1/m_{ij}!$ accounting for the
permutation of parallel contractions.

Now group the labeled terms by graph isomorphism class.  The compatible
vertex labelings of a fixed acyclic graph are its linear extensions, while
vertex automorphisms and permutations of parallel edges form precisely its
automorphism group.  Orbit--stabilizer therefore converts the labeled sum
into $\sigma(G)^{-1}I_{\mathrm{blue}}(G;t)$.  Summing over $n$ proves
\eqref{eq:Dyson-blue-graph-expansion}.
\end{proof}

\subsection{RB convolution identity}
\label{subsec:RB-convolution}

Recall from Proposition~\ref{prop:RB-convolution-reorg} that
\begin{equation}
\label{eq:RB-convolution-recall}
    \omega_{RB}*e_{RB}
    =
    e_{RB}*\omega_{RB}
    =
    \epsilon.
\end{equation}
To describe the ordering constraints arising when connected RB contributions
are multiplied by the star product, we introduce the following operation.

\begin{definition}[Redification]
\label{def:redification}
For an RB graph $G$, its \emph{redification} $\mathfrak r(G)$ is the quiver 
obtained by keeping every red-edge orientation, reversing every
blue-edge orientation, and then forgetting the colors:
\begin{equation}
\label{eq:redification}
    \text{red }(i\to j)\longmapsto i\to j,
    \qquad
    \text{blue }(i\to j)\longmapsto j\to i.
\end{equation}
Define
\begin{equation}
\label{eq:tilde-e-definition}
    \widetilde e(G)
    :=
    e\bigl(\mathfrak r(G);\mathbf1\bigr),
\end{equation}
where $\mathbf1$ means that all edge parameters are set equal to $1$.
Since redification commutes with disjoint union and $e$ is a character,
$\widetilde e$ is multiplicative.
\end{definition}

\begin{definition}[Pure-blue character]
\label{def:pure-blue-character}
Define
\begin{equation}
\label{eq:pure-blue-character}
    \mathbf b(G)
    :=
    \begin{cases}
        1,&E_r(G)=\emptyset,\\
        0,&E_r(G)\neq\emptyset.
    \end{cases}
\end{equation}
Thus $\mathbf b(G)=1$ precisely when $G$ is purely blue; in particular, the
edgeless graph is purely blue.  Since a disjoint union is purely blue if and
only if each factor is purely blue, $\mathbf b$ is multiplicative.
\end{definition}

The operator meaning of $\widetilde e$ explains the reversal in
\eqref{eq:redification}.  Let $Q$ be a tadpole-free RB graph on $m$
vertices, and let
$    \pi=(\pi_1,\ldots,\pi_m)$
be an ordering of its vertices, regarded as an operator word.  Write
$\pi^{\mathrm{rev}}=(\pi_m,\ldots,\pi_1)$
for the reversed ordering.  For an edge $j\to i$, a blue color requires
$i$ to occur before $j$ in $\pi$, whereas a red color requires $j$ to
occur before $i$.  After reversing the operator word, these conditions
are precisely the order relations encoded by the redification
$\mathfrak r(Q)$.  Hence
\begin{equation}
\label{eq:RB-ordering-redification}
    \pi\text{ produces the prescribed RB coloring of }Q
    \quad\Longleftrightarrow\quad
    \pi^{\mathrm{rev}}
    \in
    \operatorname{LE}\bigl(\mathfrak r(Q)\bigr).
\end{equation}
Here $\operatorname{LE}(H)$ denotes the set of linear extensions of
the ordering constraints defined by $H$; in particular, it is empty
when these constraints are inconsistent.

Since reversal is a
bijection on $S_m$, the fraction of operator orderings producing $Q$ is
\begin{equation}
\label{eq:tilde-e-ordering-fraction}
    \frac{|\operatorname{LE}(\mathfrak r(Q))|}{m!}
    =
    \widetilde e(Q).
\end{equation}
If $\mathfrak r(Q)$ contains a directed cycle, both sides are zero; if $Q$
contains a tadpole, no ordering can realize it and
$\widetilde e(Q)=e(\mathfrak r(Q);\mathbf1)=0$.  Thus
\eqref{eq:tilde-e-ordering-fraction} remains the correct normalized ordering
factor for every quotient graph that occurs below.

\begin{proposition}[Redification identity]
\label{prop:redification-character-identity}
On RB graphs, $\widetilde e
    =
    e_{RB}*\mathbf b$. 
Equivalently, for every RB graph $G$,
\begin{equation}
\label{eq:redification-character-expanded}
    e\bigl(\mathfrak r(G);\mathbf1\bigr)
    =
    \sum_{K\preceq G}
    e_{RB}(K)\,\mathbf b(G\cocont K).
\end{equation}
\end{proposition}

\begin{proof}
Assume first that $G$ has no tadpoles.  Expanding the convolution gives
\[
    (e_{RB}*\mathbf b)(G)
    =
    \sum_{K\preceq G}
    e_{RB}(K)\,\mathbf b(G\cocont K).
\]
By definition, $\mathbf b(G\cocont K)$ is nonzero precisely when
$G\cocont K$ is purely blue.  Since the edges of $G\cocont K$ are
exactly the edges of $G$ not contained in $K$, this means that every
red edge of $G$ must belong to $K$.  Thus the contributing
cut-sub-\mquivers are precisely
$K_S$, $E_{K_S}=E_r(G)\sqcup S$,
$S\subseteq E_b(G)$.
For such a $K_S$, we have
$e_{RB}(K_S)
    =
    (-1)^{|S|}e(K_S;\mathbf1)$,
and hence
\[
    (e_{RB}*\mathbf b)(G)
    =
    \sum_{S\subseteq E_b(G)}
    (-1)^{|S|}e(K_S;\mathbf1).
\]

We now identify this alternating sum with
$e(\mathfrak r(G);\mathbf1)$.  For a total ordering $\pi$ of $V_G$
and a non-tadpole edge $\varepsilon:u\to v$, define
\[
    X_\varepsilon(\pi)
    :=
    \begin{cases}
        1, & \text{if $v$ occurs before $u$ in $\pi$},\\
        0, & \text{otherwise}.
    \end{cases}
\]
Then $X_{\bar{\varepsilon}}(\pi)
    =
    1-X_\varepsilon(\pi)$, where $\bar{\varepsilon}$ is obtained from $\varepsilon$ by reversing the direction.
Since redification preserves the orientation of every red edge and
reverses that of every blue edge, we obtain
\[
    \mathbf1_{\{\pi\in\operatorname{LE}(\mathfrak r(G))\}}
    =
    \prod_{\varepsilon\in E_r(G)}
    X_\varepsilon(\pi)
    \prod_{\eta\in E_b(G)}
    \bigl(1-X_\eta(\pi)\bigr).
\]
Expanding the product over the blue edges gives
\[
    \mathbf1_{\{\pi\in\operatorname{LE}(\mathfrak r(G))\}}
    =
    \sum_{S\subseteq E_b(G)}
    (-1)^{|S|}
    \prod_{\varepsilon\in E_r(G)\sqcup S}
    X_\varepsilon(\pi).
\]
Summing over all total orderings $\pi$ and dividing by $|V_G|!$
therefore yields
\[
    e\bigl(\mathfrak r(G);\mathbf1\bigr)
    =
    \sum_{S\subseteq E_b(G)}
    (-1)^{|S|}e(K_S;\mathbf1)
    =
    (e_{RB}*\mathbf b)(G).
\]

It remains to consider tadpoles.  If $G$ contains a red tadpole, then
$\mathfrak r(G)$ contains a tadpole, so
$e(\mathfrak r(G);\mathbf1)=0$; on the convolution side, every
nonzero contribution would have to contain that red tadpole in $K$,
and hence also vanishes.  If $G$ contains a blue tadpole, the terms in
the convolution sum cancel in pairs according to whether that tadpole
is included in $K$: the two terms have the same quotient contribution
and opposite RB signs.  Hence the identity holds for arbitrary RB
graphs.
\end{proof}

\begin{theorem}[RB convolution identity]
\label{thm:two-color-convolution}
For every RB graph $G$,
\begin{equation}
\label{eq:two-color-convolution}
    \sum_{K\preceq G}
    \omega_{RB}(K)\,
    e\bigl(\mathfrak r(G\cocont K);\mathbf1\bigr)
    =
    \mathbf b(G).
\end{equation}
Equivalently, 
$\omega_{RB}*\widetilde e
    =
    \mathbf b$. 
\end{theorem}

\begin{proof}
By Proposition~\ref{prop:redification-character-identity},
$\widetilde e=e_{RB}*\mathbf b$.  Associativity of convolution gives
\[
    \omega_{RB}*\widetilde e
    =
    \omega_{RB}*(e_{RB}*\mathbf b)
    =
    (\omega_{RB}*e_{RB})*\mathbf b
    =
    \epsilon*\mathbf b
    =
    \mathbf b.
\]
\end{proof}

\subsection{Identification of the RB series with the Magnusian}
\label{subsec:RB-Magnus-identification}

Define the connected RB series by
\begin{equation}
\label{eq:Omega-RB-definition}
    \Omega_{\mathrm{RB}}(\lambda;t)
    :=
    \sum_{\substack{H\ \mathrm{conn}\\H\ \mathrm{acyclic}}}
    \frac{\lambda^{|V_H|}\omega_{RB,c}(H)}{\sigma(H)}
    I_{\mathrm{color}}(H;t),
\end{equation}
where the sum runs over isomorphism classes of finite connected
acyclic RB graphs.  For a homogeneous graph series
\[
    \mathcal F(\lambda;t)
    =
    \sum_G\frac{\lambda^{|V_G|}c_G}{\sigma(G)}I_{\mathrm{color}}(G;t),
\]
define its normalized graph coefficient by
\begin{equation}
\label{eq:normalized-coefficient}
    [G]\mathcal F:=c_G,
\end{equation}
so that $[G]$ removes the universal homogeneous factor
$\lambda^{|V_G|}/\sigma(G)$.

\begin{proposition}[Coefficient of the star exponential]
\label{prop:star-exponential-convolution-coefficient}
Let $G$ be an acyclic RB graph.  Then
\begin{equation}
\label{eq:star-exponential-convolution-coefficient}
    [G]\exp_\star\bigl(\Omega_{\mathrm{RB}}\bigr)
    =
    \sum_{K\preceq G}
    \omega_{RB}(K)\,
    \widetilde e(G\cocont K).
\end{equation}
\end{proposition}

\begin{proof}
Expand
\[
    \exp_\star(\Omega_{\mathrm{RB}})
    =
    \sum_{m\geq0}\frac1{m!}\Omega_{\mathrm{RB}}^{\star m}.
\]
Fix a target RB graph $G$.  A term in $\Omega_{\mathrm{RB}}^{\star m}$
partitions its vertices among $m$ connected factors.  The union $K$ of the
edges internal to those factors is a spanning cut-sub-\mquiver with exactly
$m$ connected components, isolated vertices included.  Since
$\omega_{RB,c}=\omega_{RB}$ on connected graphs and $\omega_{RB}$ is
multiplicative, the internal coefficient is $\omega_{RB}(K)$.

Set $Q:=G\cocont K$.  Its vertices are the $m$ factors and its surviving
edges are the cross-contractions.  The $m!$ orders of the factor slots occur
in $\Omega_{\mathrm{RB}}^{\star m}$, and the prefactor $1/m!$ therefore turns
the number of orders producing the prescribed RB quotient into the normalized
fraction $\widetilde e(Q)$ by
\eqref{eq:RB-ordering-redification}--\eqref{eq:tilde-e-ordering-fraction}.
It is harmless to sum over every $K\preceq G$: if an edge of $G$ is omitted
from $K$ although its endpoints lie in one connected component of $K$, that
edge becomes a tadpole in $Q$, so $\widetilde e(Q)=0$.

For the symmetry factors, first label all vertices, edge occurrences, and
factor slots.  Then the preceding correspondence is literal; forgetting the
labels and applying orbit--stabilizer divides by the stabilizer
$\operatorname{Aut}(G)$, giving the weight $1/\sigma(G)$.  Hence the normalized
contribution of $K$ is $\omega_{RB}(K)\widetilde e(G\cocont K)$, and summing
over $K$ proves the formula.
\end{proof}

Combining Proposition~\ref{prop:star-exponential-convolution-coefficient} with
Theorem~\ref{thm:two-color-convolution} gives
\begin{equation}
\label{eq:star-exponential-pure-blue-coefficient}
    [G]\exp_\star\bigl(\Omega_{\mathrm{RB}}\bigr)
    =
    \mathbf b(G).
\end{equation}
With the convention \eqref{eq:normalized-coefficient}, the right-hand side is
exactly the normalized graph coefficient of $G$ in the Dyson expansion
\eqref{eq:Dyson-blue-graph-expansion}: it is $1$ for a purely blue graph and
$0$ otherwise.  We therefore obtain the structural identification underlying
the RB expansion.

\begin{theorem}[RB star-exponential identity]
\label{thm:global-RB-Hopf-operator-compatibility}
The connected RB series $\Omega_{\mathrm{RB}}$, defined by the
coefficients $\omega_{RB,c}$, satisfies
\begin{equation}
\label{eq:U-star-connected}
    \exp_\star\bigl(\Omega_{\mathrm{RB}}(\lambda;t)\bigr)
    =
    U(\lambda;t).
\end{equation}
Consequently,
$\Omega_{\mathrm{RB}}(\lambda;t)=\Omega(\lambda;t)$.
\end{theorem}

\begin{proof}
\eqref{eq:star-exponential-pure-blue-coefficient} and
Proposition~\ref{prop:Dyson-blue-graph-expansion} show that
$\exp_\star(\Omega_{\mathrm{RB}})$ and $U$ have the same normalized coefficient
for every graph in the time-ordered Wick expansion.  Hence
$\exp_\star(\Omega_{\mathrm{RB}})=U$.  Since both series have constant term
$\mathbf1$, applying $\log_\star$ and using \eqref{U-vs-Omega_star} gives
$\Omega_{\mathrm{RB}}=\Omega$.
\end{proof}

In particular, the Magnusian has the RB expansion
\begin{equation}
\label{eq:Omega-RB-expansion-proved}
    \Omega(\lambda;t)
    =
    \sum_{\substack{H\ \mathrm{conn}\\H\ \mathrm{acyclic}}}
    \frac{\lambda^{|V_H|}\omega_{RB,c}(H)}{\sigma(H)}
    I_{\mathrm{color}}(H;t).
\end{equation}

Finally, since $I_{\mathrm{blue}}$ is multiplicative under disjoint union and
\[
    \sigma\bigl(H_1^{\sqcup m_1}\sqcup\cdots\sqcup H_r^{\sqcup m_r}\bigr)
    =
    \prod_{a=1}^r m_a!\,\sigma(H_a)^{m_a},
\]
the standard exponential formula gives the ordinary connected expansion of
the Dyson series.

\begin{corollary}[Star versus ordinary exponential]
\label{thm:star-vs-ordinary-exponential}
One has
\begin{equation}
\label{U-exp-easier}
    \exp_\star\bigl(\Omega(\lambda;t)\bigr)
    =
    \exp\left[
        \sum_{\substack{H\ \mathrm{conn}\\ H\ \mathrm{acyclic}\\H\ \mathrm{blue}}}
        \frac{\lambda^{|V_H|}}{\sigma(H)}
        I_{\mathrm{blue}}(H;t)
    \right].
\end{equation}
\end{corollary}

\section{UC change of basis and the BW specialization} 
\label{sec:RB-BW-bases-reorg}

We pass directly from the RB basis to the UC basis with
independent parameters $\beta_\varepsilon$ and $\gamma_u$.  The
resulting UC graphs are mixed quivers, and their coefficients satisfy
directed-core factorization.  The physical black--white basis is a
specialization discussed at the end of the section.

\subsection{The UC change of basis}
\label{subsec:turn-on-beta}

In the Hopf-algebraic part of this paper, each directed edge
$\varepsilon$ may carry an independent parameter $\beta_\varepsilon$,
with no restriction on its value.  By contrast, the operator-product
description of Section~\ref{sec:operator-products-Hopf} naturally
produces the two ordered-contraction kernels $W_{ji}$ and $W_{ij}$.
The directed red and blue edges arise only after these kernels are
multiplied by the time-ordering factor $\theta_{ij}$. We introduce the edge parameters at the level of the kernels, before imposing time ordering.

Choose a reference orientation
$\varepsilon:j\to i$ for an edge, and let $u=|\varepsilon|$ denote the
corresponding unoriented edge.  Introduce an antisymmetric directed
kernel $\mathrm U_\varepsilon=\mathrm U_{ij}$ ($\mathrm U_{\bar\varepsilon}=-\mathrm U_\varepsilon$) and a symmetric undirected kernel
$\mathrm C_u$, and set
$\beta_{\varepsilon,\pm}
    :=
    \frac{\beta_\varepsilon\pm1}{2}$.
We define the kernel-level deformation by
\begin{equation}
\label{basis-change-beta-gamma-reorg}
    W_{ji}
    \longmapsto
    \beta_{\varepsilon,+}\,\mathrm U_\varepsilon
    +\gamma_u\,\mathrm C_u,
    \qquad
    W_{ij}
    \longmapsto
    \beta_{\varepsilon,-}\,\mathrm U_\varepsilon
    +\gamma_u\,\mathrm C_u.
\end{equation}
Here $\beta_\varepsilon$ is attached to the directed edge
$\varepsilon$, while $\gamma_u$ is attached to the underlying
undirected edge $u$.

The determinant of the local transformation
\eqref{basis-change-beta-gamma-reorg} is $\gamma_u$, so it is
invertible whenever $\gamma_u\neq0$.  All coefficient identities below
are polynomial in the edge parameters and therefore extend to
$\gamma_u=0$ by specialization.  The physical BW specialization is
recorded in Subsection~\ref{subsec:BW graphs and edge parameters-reorg}.

Let us emphasize the distinction between the kernel-level deformation
and the graph edges after time ordering. Section~\ref{subsec:double-color-graphs-reorg}
defines
\[
    \mathrm{red}_{ij}=W_{ji}\theta_{ij},
    \qquad
    \mathrm{blue}_{ij}=W_{ij}\theta_{ij}.
\]
Hence \eqref{basis-change-beta-gamma-reorg} gives
\begin{align} \label{eq:basis-change-RB-to-UC-with-arrows}
\mathrm{red}_{ij}
    \longmapsto
    \beta_{\varepsilon,+}\,
    \mathrm U_\varepsilon\theta_{ij}
    +\gamma_u\,\mathrm C_u\theta_{ij},
    \qquad
    \mathrm{blue}_{ij}
    \longmapsto
    \beta_{\varepsilon,-}\,
    \mathrm U_\varepsilon\theta_{ij}
    +\gamma_u\,\mathrm C_u\theta_{ij}.
\end{align} 
The two terms on the right-hand side have different roles.
The term $U_\varepsilon\theta_{ij}$ represents the directed $U$-edge
$\varepsilon:j\to i$, whereas $C_u\theta_{ij}$ is the contribution
of the symmetric $C$-kernel $C_u$ in the time-ordering sector
$t_i>t_j$.  The complementary time-ordering sectors will recombine as
$C_u\theta_{ij}+C_u\theta_{ji}=C_u$.

\begin{definition}[UC graph]
\label{def:UC-graph-reorg}
A \emph{UC graph} is a mixed quiver in the UC basis:
each directed $U$-edge $\varepsilon\in E_D(G)$ carries an independent
parameter $\beta_\varepsilon$, and each undirected $C$-edge
$u\in E_U(G)$ carries an independent dressing parameter $\gamma_u$.
We use the notation $E_D(G)$, $E_U(G)$, and $G_D$ introduced in
Subsection~\ref{subsec:mixed-convolution-extension}.

A UC graph is called \emph{acyclic} if it is tadpole-free and its
directed core $G_D$ is acyclic.
\end{definition}

\subsection{UC graph coefficients}
\label{subsec:BW-omega-from-basis-reorg}

We now keep both terms in
\eqref{eq:basis-change-RB-to-UC-with-arrows} and determine the resulting
graph coefficients.

Recall the RB expansion established in
Theorem~\ref{thm:global-RB-Hopf-operator-compatibility}.  Setting the
bookkeeping parameter $\lambda=1$, we obtain
\begin{equation}
\label{eq:Omega-RB-expansion-reorg}
    \Omega(t)
    =
    \sum_{G_{RB}}
    \frac{
        \omega_{RB,c}(G_{RB})
    }{
        \sigma(G_{RB})
    }
    I_{\mathrm{color}}(G_{RB};t),
\end{equation}
where the sum runs over isomorphism classes of finite connected
acyclic RB graphs.
The physical specialization
$(\beta_r,\beta_b)=(1,-1)$ and the factor
$(-1)^{b(G_{RB})}$ are already included in
$\omega_{RB,c}$.

Apply \eqref{eq:basis-change-RB-to-UC-with-arrows} edge by edge.
For each edge occurrence, choosing the $\mathrm U$-term produces a
directed $U$-edge, whereas choosing the $\mathrm C$-term produces a
time-ordered piece
$\mathrm C_u\theta_{ij}$ of an undirected edge.
Thus a fixed UC graph is obtained together with a choice of
time-ordering sector for each of its undirected edges.  We use
\emph{sector sum} for the sum over these admissible time-ordering
assignments, equivalently over the admissible directed completions below.

Let $G_{\mathrm{UC}}$ be an acyclic UC graph, and let $G_D$ denote its
directed core.  An \emph{admissible directed completion}
$\widetilde G_{\mathrm{UC}}$ is obtained by orienting every edge in
$E_U(G_{\mathrm{UC}})$ so that the resulting quiver is
acyclic and extends the orientations already present in $G_D$.
Such a completion exists: choose a linear extension of the partial
order defined by $G_D$ and orient all undirected edges consistently
with it.

For a fixed admissible completion
$\widetilde G_{\mathrm{UC}}$, let
\[
    X(\widetilde G_{\mathrm{UC}})
    :=
    \left\{
        \widetilde G_{\mathrm{UC}}^\kappa
        \,\middle|\,
        \kappa:
        E_{\widetilde G_{\mathrm{UC}}}
        \longrightarrow
        \{r,b\}
    \right\},
\]
where
$\widetilde G_{\mathrm{UC}}^\kappa$ is the RB graph obtained by
coloring every edge occurrence red or blue.  For
$G_{RB}\in X(\widetilde G_{\mathrm{UC}})$, set
\[
    E_{D,r}(G_{RB})
    :=
    E_D(G_{\mathrm{UC}})\cap E_r(G_{RB}),
    \qquad
    E_{D,b}(G_{RB})
    :=
    E_D(G_{\mathrm{UC}})\cap E_b(G_{RB}).
\]
Define the coefficient of the sector
$\widetilde G_{\mathrm{UC}}$ by
\begin{equation}
\label{eq:UC-sector-coefficient-reorg}
\mathcal C_{\widetilde G_{\mathrm{UC}}}(G_{\mathrm{UC}};\boldsymbol\beta,\boldsymbol\gamma)
:=\left(\prod_{u\in E_U(G_{\mathrm{UC}})}\gamma_u\right)
\sum_{G_{RB}\in X(\widetilde G_{\mathrm{UC}})}
\left(\prod_{\varepsilon\in E_{D,r}(G_{RB})}\beta_{\varepsilon,+}\right)
\left(\prod_{\varepsilon\in E_{D,b}(G_{RB})}\beta_{\varepsilon,-}\right)
\omega_{RB,c}(G_{RB}).
\end{equation}
Here an edge that remains directed contributes
$\beta_{\varepsilon,+}$ or $\beta_{\varepsilon,-}$ according to
whether its RB preimage is red or blue.  By contrast, an edge that is
to become undirected contributes the same factor $\gamma_u$ for
either color.  It is this equality of the two prefactors that leads to
the cancellation below.

\begin{example}
\label{ex:UC-two-parallel-edges}
Consider two vertices joined by two parallel edges, with one directed
edge $\varepsilon:2\to1$ and one undirected edge $u$.  Choose the
auxiliary orientation $2\to1$ for $u$, and order the RB colors so that
the first color refers to $\varepsilon$ and the second to the auxiliary
edge.  The four RB preimages $rr,rb,br,bb$ have connected coefficients
$-\frac12,0,0,\frac12$, respectively.  Thus, for each fixed color of
$\varepsilon$, summing over the two colors of the auxiliary edge gives
$-\frac12+0=-\frac12$ and $0+\frac12=\frac12$, which are precisely the
connected RB coefficients obtained after deleting the auxiliary edge.
Since both color choices of the auxiliary edge carry the same factor
$\gamma_u$,
\[
    \mathcal C_{\widetilde G_{\mathrm{UC}}}
    =
    \gamma_u
    \left(
        -\frac12\beta_{\varepsilon,+}
        +
        \frac12\beta_{\varepsilon,-}
    \right)
    =
    -\frac{\gamma_u}{2}
    =
    \gamma_u\,\omega_c(G_D),
\]
where $G_D$ is the one-edge directed core.  If both parallel edges are
undirected, summing over the RB colors deletes both auxiliary oriented
representatives, leaving the disconnected directed core
$\bullet\sqcup\bullet$, so the connected coefficient vanishes.
\end{example}

This example exhibits the basic mechanism.  For a target undirected
edge, summing over the two colors of its auxiliary oriented
representative has the effect of deleting that auxiliary edge from the
connected RB coefficient.  The following lemma formalizes this
edgewise identity.

\begin{lemma}[RB deletion identity]
\label{lem:RB-white-edge-sum-reorg}
Fix an oriented edge $\varepsilon$ and keep all other RB data fixed.
Let $G_\varepsilon^r$ and $G_\varepsilon^b$ be the two RB graphs
obtained by coloring $\varepsilon$ red or blue.  Then
\begin{equation}
\label{eq:RB-white-edge-sum-reorg}
    \omega_{RB,c}(G_\varepsilon^r)
    +
    \omega_{RB,c}(G_\varepsilon^b)
    =
    \omega_{RB,c}(G_{\setminus\varepsilon}),
\end{equation}
where all remaining RB colors are retained, and the right-hand side is
understood to be zero if $G_{\setminus\varepsilon}$ is disconnected.
\end{lemma}

\begin{proof}
Let $b_0$ be the number of blue edges among the fixed remaining
edges.  By \eqref{eq:RB-omega-c-reorg},
\[
    \omega_{RB,c}(G_\varepsilon^r)
    +
    \omega_{RB,c}(G_\varepsilon^b)
    =
    (-1)^{b_0}
    \left(
        \omega_c(G;\beta_\varepsilon=1)
        -
        \omega_c(G;\beta_\varepsilon=-1)
    \right),
\]
with all remaining edge parameters fixed at their RB values.
Since $\omega_c(G)$ is affine in $\beta_\varepsilon$ and
\[
    2\partial_{\beta_\varepsilon}\omega_c(G)
    =
    \omega_c(G_{\setminus\varepsilon}),
\]
we have
\[
    \omega_c(G;1)
    -
    \omega_c(G;-1)
    =
    \omega_c(G_{\setminus\varepsilon}).
\]
Restoring the factor $(-1)^{b_0}$ gives
\eqref{eq:RB-white-edge-sum-reorg}.
\end{proof}

\begin{example}
\label{ex:UC-transitive-triangle}
Consider the transitive triangle with directed edges
\[
    \varepsilon_{12}:2\to1,
    \qquad
    \varepsilon_{23}:3\to2,
    \qquad
    \varepsilon_{13}:3\to1.
\]
For the purely directed graph,
$    \omega_c(G;\boldsymbol\beta)
    =
    \frac{\beta_{\varepsilon_{12}}}{12}
    +
    \frac{\beta_{\varepsilon_{23}}}{12}
    +
    \frac{\beta_{\varepsilon_{13}}}{6}$.
Direct summation of the RB preimages gives
$    \frac{\gamma_{|\varepsilon_{13}|}}{3}$
when $\varepsilon_{13}$ is made undirected, and respectively
$    \frac{\gamma_{|\varepsilon_{12}|}}{6}$,
$    \frac{\gamma_{|\varepsilon_{23}|}}{6}$
when $\varepsilon_{12}$ or $\varepsilon_{23}$ is made undirected.
If two or three edges are made undirected, the directed core is
disconnected and the connected coefficient vanishes.
In the uniform specialization
$\beta_\varepsilon=\beta$ and $\gamma_u=\gamma$, these are the
coefficients
$    \frac{\beta}{3}$, 
$\frac{\gamma}{3}$, 
$    \frac{\gamma}{6}$, 
$    \frac{\gamma}{6}$
shown in Fig.~\ref{fig:triangle-beta-re}.
\end{example}

\begin{figure}[htbp]
    \centering
    \includegraphics[width=0.75\linewidth]{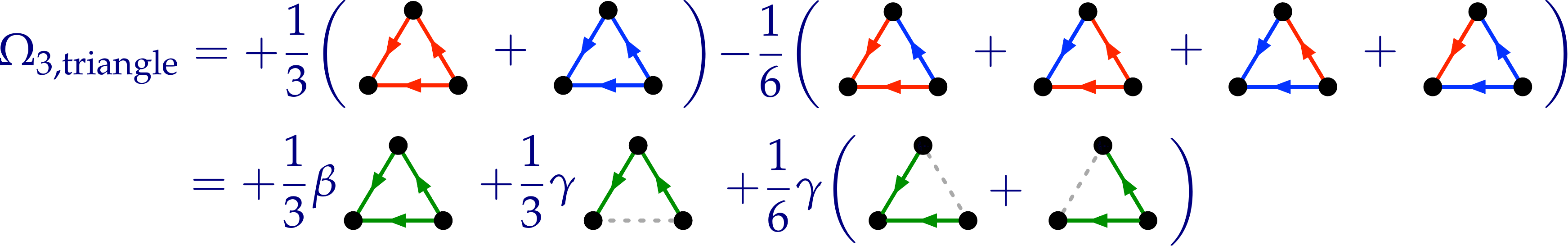}
    \caption{The change of basis for the transitive triangle in the
    uniform specialization.}
    \label{fig:triangle-beta-re}
\end{figure}


\begin{theorem}[Directed-core factorization for UC graphs]
\label{thm:UC-directed-core-factorization-reorg}
Let $G_{\mathrm{UC}}$ be an acyclic UC graph, and let $G_D$ denote its
directed core.  For every admissible directed completion
$\widetilde G_{\mathrm{UC}}$,
\begin{equation}
\label{eq:UC-directed-core-factorization-reorg}
    \mathcal C_{\widetilde G_{\mathrm{UC}}}
    \bigl(
        G_{\mathrm{UC}};
        \boldsymbol\beta,\boldsymbol\gamma
    \bigr)
    =
    \left(
        \prod_{u\in E_U(G_{\mathrm{UC}})}
        \gamma_u
    \right)
    \omega_c(G_D;\boldsymbol\beta),
\end{equation}
where the right-hand side is understood to be zero when $G_D$ is
disconnected.  In particular, the sector coefficient is independent
of the admissible directed completion.
\end{theorem}

\begin{proof}
Start from \eqref{eq:UC-sector-coefficient-reorg}.  For
$u\in E_U(G_{\mathrm{UC}})$, let $\varepsilon$ be its auxiliary
orientation in $\widetilde G_{\mathrm{UC}}$.  Its two RB preimages have
the same prefactor $\gamma_u$; summing over their colors and applying
Lemma~\ref{lem:RB-white-edge-sum-reorg} deletes $\varepsilon$.  Repeating
this for all undirected edges leaves
\begin{equation}
\label{eq:UC-core-RB-sum-reorg}
\mathcal C_{\widetilde G_{\mathrm{UC}}}
=\left(\prod_{u\in E_U(G_{\mathrm{UC}})}\gamma_u\right)
\sum_{\kappa:E_D(G_{\mathrm{UC}})\to\{r,b\}}
\left(\prod_{\varepsilon\in E_r(G_D^\kappa)}\beta_{\varepsilon,+}\right)
\left(\prod_{\varepsilon\in E_b(G_D^\kappa)}\beta_{\varepsilon,-}\right)
\omega_{RB,c}(G_D^\kappa).
\end{equation}
This identity also covers disconnected $G_D$, since then every
$\omega_{RB,c}(G_D^\kappa)$ is zero.

For a coloring $\kappa$, put $\eta_\varepsilon=1$ on red edges and
$\eta_\varepsilon=-1$ on blue edges.  By
\eqref{eq:RB-omega-c-reorg}, the sum in
\eqref{eq:UC-core-RB-sum-reorg} equals
\[
\sum_{\boldsymbol\eta\in\{\pm1\}^{E_D(G_{\mathrm{UC}})}}
\omega_c(G_D;\boldsymbol\eta)
\prod_{\varepsilon\in E_D(G_{\mathrm{UC}})}
\frac{1+\eta_\varepsilon\beta_\varepsilon}{2},
\]
because a blue edge contributes
$-\beta_{\varepsilon,-}=(1-\beta_\varepsilon)/2$.  By
Corollary~\ref{cor:omega-multiaffine}, this is precisely the
multiaffine interpolation of $\omega_c(G_D;\boldsymbol\beta)$ from its
$\{\pm1\}$ corner values.  This proves
\eqref{eq:UC-directed-core-factorization-reorg}.  No division by
$\gamma_u$ occurs, so the identity remains valid at $\gamma_u=0$ by
specialization.
\end{proof}

For an acyclic UC graph $G$, define its graph integral by
\begin{equation}
\label{eq:I-UC}
\begin{split}
    I_{\mathrm{UC}}(G;t)
    :={}&
    \int_{[0,t]^{|V_G|}}
    \left(
        \prod_{\substack{\varepsilon:j\to i\\
        \varepsilon\in E_D(G)}}
        \mathrm U_\varepsilon\theta_{ij}
    \right)
    \left(
        \prod_{u\in E_U(G)}
        \mathrm C_u
    \right)
    \left(
        \prod_{v\in V_G}A(t_v)
    \right)
    \prod_{v\in V_G}dt_v.
\end{split}
\end{equation}
Thus a directed edge $\varepsilon:j\to i$ contributes
$\mathrm U_\varepsilon\theta_{ij}$, whereas an undirected edge $u$
contributes the symmetric kernel $\mathrm C_u$ without a preferred
time orientation.

\begin{corollary}[UC expansion]
\label{cor:UC-expansion-reorg}
The RB expansion \eqref{eq:Omega-RB-expansion-reorg} can be rewritten as
\begin{equation}
\label{eq:Omega-UC-expansion-reorg}
    \Omega(t)
    =
    \sum_{G_{\mathrm{UC}}}
    \frac{
        \omega_{\mathrm{UC},c}
        (G_{\mathrm{UC}};
        \boldsymbol\beta,\boldsymbol\gamma)
    }{
        \sigma(G_{\mathrm{UC}})
    }
    I_{\mathrm{UC}}(G_{\mathrm{UC}};t),
\end{equation}
where the sum runs over isomorphism classes of finite connected acyclic UC graphs, and
\begin{equation}
\label{eq:UC-core-factorization-final}
    \omega_{\mathrm{UC},c}
    \bigl(
        G_{\mathrm{UC}};
        \boldsymbol\beta,\boldsymbol\gamma
    \bigr)
    =
    \left(
        \prod_{u\in E_U(G_{\mathrm{UC}})}
        \gamma_u
    \right)
    \omega_c(G_D;\boldsymbol\beta).
\end{equation}
\end{corollary} 

\begin{proof}
For each undirected edge $u=|\varepsilon|$, the time-ordering
functions satisfy $\theta_{ij}+\theta_{ji}=1$
away from the measure-zero boundary $t_i=t_j$, and hence
$\mathrm C_u
=
\mathrm C_u
\bigl(
\theta_{ij}+\theta_{ji}
\bigr)$.
Expanding these factors decomposes the UC graph integral into
time-ordering sectors indexed by orientations of the undirected
edges.  Away from the measure-zero boundaries $t_i=t_j$, each time
configuration determines a unique orientation of every undirected
edge.  A cyclic orientation gives an empty sector, while an acyclic
orientation has a linear extension and hence gives precisely an
admissible directed completion.  Thus the sector sum has no duplicate
counting.

By Theorem~\ref{thm:UC-directed-core-factorization-reorg}, every admissible
sector has the same coefficient
\[
    \left(
        \prod_{u\in E_U(G_{\mathrm{UC}})}
        \gamma_u
    \right)
    \omega_c(G_D;\boldsymbol\beta).
\]
This common coefficient can therefore be factored out, and the
complementary sectors recombine edgewise as
$\mathrm C_u\theta_{ij}+\mathrm C_u\theta_{ji}=\mathrm C_u$.
For the symmetry factors, first label all vertices and edge occurrences;
the edgewise change of basis and sector decomposition are then literal and
equivariant under relabeling.  Forgetting the labels and applying
orbit--stabilizer gives the factor $1/\sigma(G_{\mathrm{UC}})$.  This
reconstructs $I_{\mathrm{UC}}(G_{\mathrm{UC}})$ and proves
\eqref{eq:Omega-UC-expansion-reorg}.
\end{proof}

\begin{corollary}[Compatibility of direct convolution and the UC basis change]
\label{cor:UC-mixed-Hopf-compatibility}
For every finite connected acyclic UC graph $G$,
\begin{equation}
\label{eq:UC-mixed-Hopf-compatibility}
    \omega_{\mathrm{UC},c}
    (G;\boldsymbol\beta,\boldsymbol\gamma)
    =
    \omega_c
    (G;\boldsymbol\beta,\boldsymbol\gamma).
\end{equation}
\end{corollary}

\begin{proof}
By the directed-core factorization
\eqref{eq:UC-core-factorization-final}, the left-hand side is
\[
    \left(\prod_{u\in E_U(G)}\gamma_u\right)
    \omega_c(G_D;\boldsymbol\beta).
\]
Theorem~\ref{thm:mixed-convolution-factorization} identifies this
expression with the right-hand side.
\end{proof}

For a purely directed target, every edge is extracted from the
$\mathrm U$-part of \eqref{eq:basis-change-RB-to-UC-with-arrows}.
For an edge
$\varepsilon:j\to i$,
\begin{equation}
\label{eq:pure-directed-beta-evaluation}
    W_{ji}
    \longmapsto
    \beta_{\varepsilon,+},
    \qquad
    W_{ij}
    \longmapsto
    \beta_{\varepsilon,-}.
\end{equation}
This is precisely the local evaluation rule used in
Section~\ref{sec:operator-acyclic-master} to derive the acyclic master
formula from operator products.

\subsection{Physical BW specialization}
\label{subsec:BW graphs and edge parameters-reorg}

At the physical point
$(\boldsymbol\beta,\boldsymbol\gamma)=(\mathbf0,\mathbf1)$, the UC
kernels reduce to the black--white basis of~\cite{Guo:2026xaw}.  Namely,
\begin{equation}
\label{eq:change-of-basis-BW-RB}
    B_{ij}:=W_{ji}-W_{ij},
    \qquad
    C_{ij}:=\frac12(W_{ji}+W_{ij}),
\end{equation}
and the UC kernels specialize as
\[
    \mathrm U_\varepsilon=B_{ij},
    \qquad
    \mathrm C_u=C_{ij}.
\]
Thus the directed $U$-edges are drawn as black edges and the
undirected $C$-edges as white edges.  Multiplying by the time-ordering factors gives
\[
    \mathrm{black}_{ij}
    =\mathrm{red}_{ij}-\mathrm{blue}_{ij},
    \qquad
    \mathrm{white}_{ij}
    =\frac12\bigl(
        \mathrm{red}_{ij}+\mathrm{blue}_{ij}
        +\mathrm{red}_{ji}+\mathrm{blue}_{ji}
    \bigr).
\]
Equivalently, a BW graph is a UC
graph at $(\boldsymbol\beta,\boldsymbol\gamma)=(\mathbf0,\mathbf1)$,
and
\[
    \omega_{\mathrm{BW},c}(G)
    :=
    \omega_{\mathrm{UC},c}(G;\mathbf0,\mathbf1)
    =
    \omega_c(G;\mathbf0,\mathbf1).
\]
In physics parlance the black edges correspond to retarded propagators
and the white edges to cut propagators.

For example, the tree Magnusian in the RB basis and its BW
specialization are shown in Figs.~\ref{fig:Dyson-Magnus-2-reorg}
and~\ref{fig:Dyson-Magnus-3-reorg}.
\begin{figure}[htbp]
    \centering
    \includegraphics[width=0.64\linewidth]{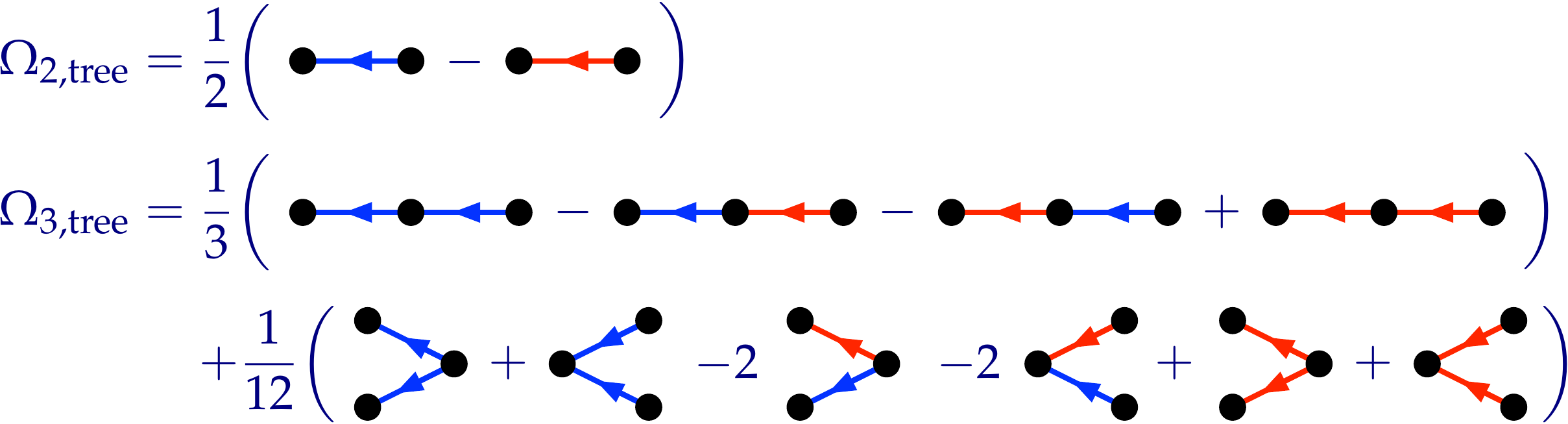}
    \caption{Tree Magnusian with 2 or 3 vertices in the RB basis.}
    \label{fig:Dyson-Magnus-2-reorg}
\end{figure}
\begin{figure}[htbp]
    \centering
    \includegraphics[width=0.46\linewidth]{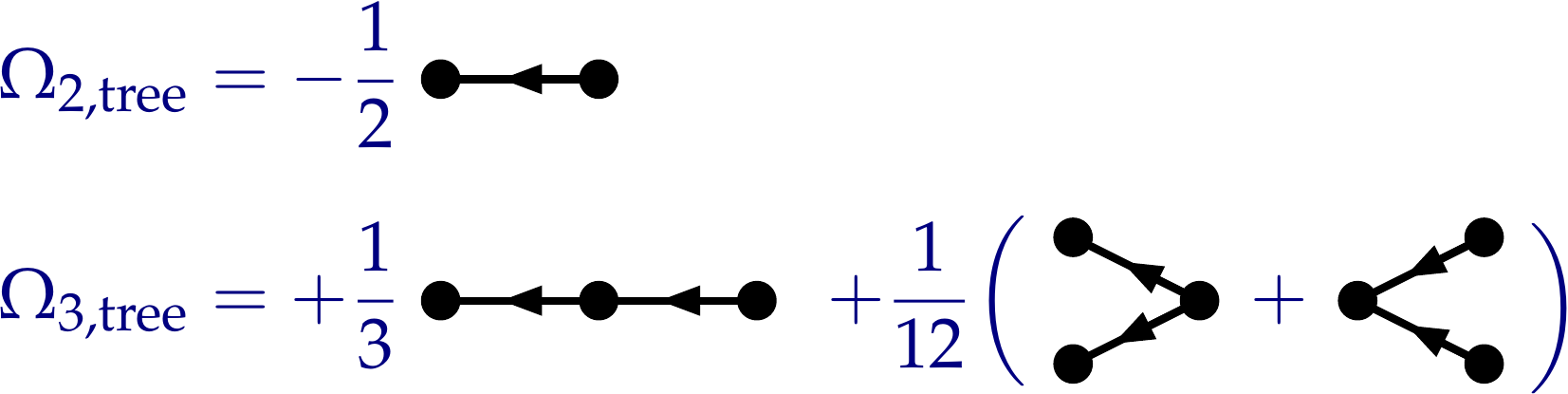}
    \caption{Tree Magnusian with 2 or 3 vertices in the BW basis.}
    \label{fig:Dyson-Magnus-3-reorg}
\end{figure}

\begin{remark}\label{rem:RB-BW-physics-convention-2}
Relative to~\cite{Guo:2026xaw}, we suppress the factors of $\pm i$
used there to display Hermiticity; the conventions are related by
$(i\hbar)\Omega_{\mathrm{here}}=\chi_{\mathrm{there}}$.
\end{remark}

\section{Acyclic master formula from Magnus expansion}
\label{sec:operator-acyclic-master}

We now derive the acyclic master formula independently from the
Magnus permutation formula, using
$W_{ji}\mapsto\beta_{\varepsilon,+}$ and
$W_{ij}\mapsto\beta_{\varepsilon,-}$ for each directed edge
$\varepsilon:j\to i$.

\subsection{Permutation coefficients in the Magnus expansion}
\label{subsec:Magnus-permutation-coefficients}

For $\sigma=(\sigma_1,\ldots,\sigma_n)\in S_n$, write
\[
    A_{(\sigma)}
    :=
    A(t_{\sigma_1})\star\cdots\star A(t_{\sigma_n}),
\]
and define
\begin{equation}
\label{eq:operator-descent-set}
    \operatorname{Des}(\sigma)
    :=
    \{a\in\{1,\ldots,n-1\}\mid \sigma_a>\sigma_{a+1}\},
    \qquad
    k(\sigma):=|\operatorname{Des}(\sigma)|.
\end{equation}
Set
\begin{equation}
\label{eq:operator-cnk}
    c_{n,k}
    :=
    \frac{(-1)^k}{n\binom{n-1}{k}}.
\end{equation}

The following classical permutation form of the Magnus expansion is due to
Mielnik--Pleba\'nski and Strichartz; see
\cite{Mielnik_1970,Strichartz_1987} and Eq. (24) in \cite{Ebrahimi_Fard_2014}.

\begin{theorem}[Mielnik--Pleba\'nski--Strichartz formula]
\label{thm:Magnus-permutation-form}
For every $n\geq1$,
\begin{equation}
\label{eq:Magnus-permutation-form}
    \Omega_n(t)
    =
    \int_{\Delta_n(t)}
    \left(
        \sum_{\sigma\in S_n}
        c_{n,k(\sigma)}A_{(\sigma)}
    \right)
    dt_1\cdots dt_n.
\end{equation}
Thus the coefficient of an operator word depends only on the number of descents
of the corresponding permutation.
\end{theorem}

\subsection{From operator orderings to edge-parameter monomials}
\label{subsec:operator-orderings-edge-weights}

Let $G$ be an acyclic quiver with $V_G=[n]$, topologically labeled so
that every edge has the form $\varepsilon:j\to i$ with $i<j$; fix this
labeling throughout the section. Recall that
$\beta_{\varepsilon,\pm}=(\beta_\varepsilon\pm1)/2$, that
$p_\sigma(i)=\sigma^{-1}(i)$ is the position of $i$ in the operator
word $A_{(\sigma)}$, and that
\begin{equation}
\label{eq:operator-edge-sign}
    s_\sigma(\varepsilon)
    =
    \begin{cases}
        +,&p_\sigma(i)<p_\sigma(j),\\
        -,&p_\sigma(i)>p_\sigma(j).
    \end{cases}
\end{equation}
Thus $s_\sigma(\varepsilon)$ is precisely the edge sign appearing in
the acyclic master formula.

We now apply the purely directed rule of
\eqref{eq:pure-directed-beta-evaluation}.  For an edge
$\varepsilon:j\to i$, it gives
\begin{equation}
\label{eq:operator-beta-evaluation}
    W_{ij}\longmapsto\beta_{\varepsilon,-},
    \qquad
    W_{ji}\longmapsto\beta_{\varepsilon,+}.
\end{equation}
The Wick expansion determines which ordered contraction occurs, and
\eqref{eq:operator-beta-evaluation} assigns the corresponding
edge parameter.  For parallel edges, the rule is applied independently
to each edge.  Define
\begin{equation}
\label{eq:operator-graph-monomial}
    M_G^{\mathrm{op}}(\sigma)
    :=
    \prod_{\varepsilon\in E_G}
    \beta_{\varepsilon,-s_\sigma(\varepsilon)}.
\end{equation}

\begin{proposition}[Edge-parameter monomial of an operator ordering]
\label{prop:operator-ordering-weight}
For the fixed topologically labeled acyclic quiver $G$ and every
$\sigma\in S_n$, the contribution of the operator word $A_{(\sigma)}$
to the edge parameters of $G$ is $M_G^{\mathrm{op}}(\sigma)$.
\end{proposition}

\begin{proof}
Let $\varepsilon:j\to i$.  If $p_\sigma(i)<p_\sigma(j)$, then $i$
occurs before $j$ in the operator word, so the Wick contraction is
$W_{ij}$ and contributes $\beta_{\varepsilon,-}$.  Since
$s_\sigma(\varepsilon)=+$, this is
$\beta_{\varepsilon,-s_\sigma(\varepsilon)}$.  If
$p_\sigma(i)>p_\sigma(j)$, the contraction is $W_{ji}$ and contributes
$\beta_{\varepsilon,+}$; since $s_\sigma(\varepsilon)=-$, the same
expression is obtained.  The Wick rule factorizes over the edges, and
multiplying these contributions gives
\eqref{eq:operator-graph-monomial}.
\end{proof}

\subsection{Derivation of the acyclic master formula}
\label{subsec:operator-special-formula}

Recall the acyclic master coefficients
$
    d_{n,k}
    =
    \frac{(-1)^{n-1-k}}{n\binom{n-1}{k}}$
and the acyclic master function
\[
    \overline{\Omega}_n(G)
    =
    \sum_{\sigma\in S_n}
    d_{n,k(\sigma)}
    \prod_{\varepsilon\in E_G}
    \beta_{\varepsilon,s_\sigma(\varepsilon)}.
\]
On the operator-product side, Proposition~\ref{prop:operator-ordering-weight} gives
\begin{equation}
\label{eq:operator-coefficient-G}
    \Omega_n^{\mathrm{op}}(G)
    :=
    \sum_{\sigma\in S_n}
    c_{n,k(\sigma)}
    \prod_{\varepsilon\in E_G}
    \beta_{\varepsilon,-s_\sigma(\varepsilon)}.
\end{equation}

For $\sigma=(\sigma_1,\ldots,\sigma_n)$, let
$\sigma^{\mathrm{rev}}=(\sigma_n,\ldots,\sigma_1)$.

\begin{lemma}[Reversal of operator orderings]
\label{lem:operator-permutation-reversal}
For every $\sigma\in S_n$,
\[
    k(\sigma^{\mathrm{rev}})
    =
    n-1-k(\sigma),
    \qquad
    s_{\sigma^{\mathrm{rev}}}(\varepsilon)
    =
    -s_\sigma(\varepsilon),
\]
and consequently
$c_{n,k(\sigma)}
    =
    d_{n,k(\sigma^{\mathrm{rev}})}$.
\end{lemma}

\begin{proof}
A descent of $\sigma^{\mathrm{rev}}$ at position $a$ is equivalent to
an ascent of $\sigma$ at position $n-a$; hence
$k(\sigma^{\mathrm{rev}})=n-1-k(\sigma)$. Also
$p_{\sigma^{\mathrm{rev}}}(i)=n+1-p_\sigma(i)$, so every endpoint
order is reversed and $s_{\sigma^{\mathrm{rev}}}(\varepsilon)
=-s_\sigma(\varepsilon)$. Finally, using
$\binom{n-1}{n-1-k}=\binom{n-1}{k}$,
\[
 d_{n,k(\sigma^{\mathrm{rev}})}=d_{n,n-1-k(\sigma)}
 =\frac{(-1)^{k(\sigma)}}{n\binom{n-1}{k(\sigma)}}=c_{n,k(\sigma)}.
\]
\end{proof}

\begin{theorem}[Operator-product derivation of the acyclic master formula]
\label{thm:operator-special-master-formula}
Let $G$ be an acyclic quiver with a fixed topological labeling
$V_G=[n]$ as above. Then
\begin{equation}
\label{eq:operator-special-equality}
    \Omega_n^{\mathrm{op}}(G)
    =
    \overline{\Omega}_n(G).
\end{equation}
In particular, if $G$ is connected, then
$\Omega_n^{\mathrm{op}}(G)
    =
    \omega_c(G;\boldsymbol\beta)$.
\end{theorem}

\begin{proof}
By \eqref{eq:operator-coefficient-G} and
Lemma~\ref{lem:operator-permutation-reversal},
\[
 \Omega_n^{\mathrm{op}}(G)=\sum_{\sigma\in S_n}
 d_{n,k(\sigma^{\mathrm{rev}})}
 \prod_{\varepsilon\in E_G}\beta_{\varepsilon,s_{\sigma^{\mathrm{rev}}}(\varepsilon)}.
\]
Since $\sigma\mapsto\sigma^{\mathrm{rev}}$ is a bijection of $S_n$,
reindexing gives
\[
 \Omega_n^{\mathrm{op}}(G)=\sum_{\tau\in S_n}d_{n,k(\tau)}
 \prod_{\varepsilon\in E_G}\beta_{\varepsilon,s_\tau(\varepsilon)}
 =\overline{\Omega}_n(G).
\]
For connected $G$, the acyclic master formula gives
$\overline{\Omega}_n(G)=\omega_c(G;\boldsymbol\beta)$.
\end{proof}

Theorem~\ref{thm:operator-special-master-formula} shows that operator-product and Hopf-algebraic routes agree at the level
of connected acyclic coefficients.

\section{Quantum Murua formula}
\label{sec:quantum-Murua-reorg}
We extend the quantum Murua formula of \cite{Guo:2026xaw} to
independent edge parameters and prove that the resulting recursion
equals $\omega_c$ on connected acyclic quivers.  The RB and BW
specializations follow at the end.

\subsection{Cut-subquivers and almost-rooted spider webs}
\label{subsec:almost-rooted-spider-webs-reorg}

Let $G$ be a connected acyclic quiver.

\begin{definition}
A vertex $s\in V_G$ is called a \emph{semi-root} if it has no outgoing edge.
\end{definition}

Since $G$ is acyclic, it has at least one semi-root. Fix one and denote it by $s$. Recall that $K\preceq G$ denotes a cut-subquiver of $G$. Set
\begin{equation}
\label{eq:Murua-cut-subquivers-reorg}
    \mathcal K_s(G)
    :=
    \{K\preceq G\mid \text{$s$ is isolated in $K$}\}.
\end{equation}
For $K\in\mathcal K_s(G)$, contract every connected component of $K$ and denote the resulting edge-labeled quiver by ${G}\cocont{K}$. Its edges are precisely the edges of $G$ not contained in $K$, with their labels retained. Define
\begin{equation}
\label{eq:Murua-effective-vertices-reorg}
    \nu_G(K)
    :=
    |V_{{G}\cocont{K}}|
    =
    c(K),
\end{equation}
where $c(K)$ is the number of connected components of $K$. The isolated component $\{s\}$ determines a distinguished root vertex of ${G}\cocont{K}$. Although $G$ is acyclic, ${G}\cocont{K}$ need not be acyclic and may contain tadpoles.

\begin{remark}[Notation of \cite{Guo:2026xaw}]
\label{rem:Murua-notation-dictionary-reorg}
Their $p$ is $E_G\setminus E_K$, their $G\setminus p$ is our $K$, and
their quotient is ${G}\cocont K$.  Their $|p|$ equals
$\nu_G(K)=|V_{{G}\cocont K}|$, while
$\omega(G\backslash p)=\omega(K)$.  The condition
$p\in P_s(G)$ is equivalent to $K\in\mathcal K_s(G)$.
\end{remark}

We next give an intrinsic graph-theoretic formulation of the ``almost-rooted'' condition used in \cite{Guo:2026xaw}.

\begin{definition}[Almost-rooted orientation]
\label{def:almost-rooted-reorg}
Let $H$ be a finite quiver with a distinguished vertex $r$.
We say that $(H,r)$ is \emph{almost rooted} if $r$ has no outgoing
edge, every vertex $v\neq r$ has at least one outgoing edge, and every
choice of one outgoing edge at each $v\neq r$ yields a rooted spanning
tree oriented toward $r$.
\end{definition}

\begin{lemma}
\label{lem:almost-rooted-equivalence-reorg}
Let $H$ be a finite quiver with a distinguished vertex $r$. Then
$(H,r)$ is almost rooted if and only if $H$ is acyclic and $r$ is its
unique sink.
\end{lemma}

\begin{proof}
Assume first that $H$ is acyclic with unique sink $r$. After choosing one outgoing edge at each $v\neq r$, repeatedly following the unique outgoing edge cannot produce a directed cycle and must therefore terminate at a sink, necessarily $r$. Hence every vertex is connected to $r$. The resulting spanning graph has $|V_H|-1$ edges and is therefore a rooted tree directed toward $r$.

Conversely, suppose that $(H,r)$ is almost rooted. Then $r$ is a sink and every other vertex has positive outdegree. If $H$ contained a directed cycle, choosing at each vertex of the cycle the edge lying on that cycle, and choosing arbitrarily at the remaining non-root vertices, would produce a graph containing the same directed cycle, contrary to the definition. Thus $H$ is acyclic, and $r$ is its unique sink.
\end{proof}

\begin{definition}[Spider web]
\label{def:spider-realization-reorg}
Let $Q$ be an edge-labeled quiver with distinguished vertex $s$. A \emph{spider web of $Q$ rooted at $s$} is an orientation $\Gamma$ of the underlying edge-labeled graph of $Q$ such that $(\Gamma,s)$ is almost rooted. We denote the set of such spider webs by $\mathcal W_s(Q)$.
\end{definition}

If $Q$ contains a tadpole, then $\mathcal W_s(Q)=\emptyset$. Parallel edges remain distinct throughout the construction: they may impose the same relation in the partial order of $\Gamma$, but they remain distinct in all edge weights. 

For $\Gamma\in\mathcal W_s(Q)$ and $v\neq s$, let $h_v(\Gamma)$ be the set of outgoing edges from $v$ in $\Gamma$; we call it the \emph{spider shot} at $v$.

\subsection{Refined spider factors}
\label{subsec:refined-spider-factors-reorg}

Recall that $\beta_{\varepsilon,\pm}=(\beta_\varepsilon\pm1)/2$. Let $Q$ be an edge-labeled quiver with root $s$ and let $\Gamma\in\mathcal W_s(Q)$. The orientations of $\Gamma$ and $Q$ can be compared edge by edge. For a spider shot $h=h_v(\Gamma)$, set
\begin{equation}
\label{eq:spider-reversal-set-reorg}
    F_h(Q,\Gamma)
    :=
    \{\varepsilon\in h\mid \text{$\varepsilon$ has opposite orientations in $\Gamma$ and $Q$}\}.
\end{equation}

\begin{definition}[Refined spider factor]
\label{def:refined-spider-factor-reorg}
For a spider shot $h$ and $F\subseteq h$, define
\begin{equation}
\label{eq:refined-spider-factor-reorg}
    s_{\boldsymbol\beta}(h,F)
    :=
    \prod_{\varepsilon\in h\setminus F}\beta_{\varepsilon,+}
    \prod_{\varepsilon\in F}\beta_{\varepsilon,-}
    -
    \prod_{\varepsilon\in h\setminus F}\beta_{\varepsilon,-}
    \prod_{\varepsilon\in F}\beta_{\varepsilon,+}.
\end{equation}
For $\Gamma\in\mathcal W_s(Q)$, set
\begin{equation}
\label{eq:refined-g-reorg}
    g_{\boldsymbol\beta}(Q,\Gamma)
    :=
    \prod_{v\neq s}
    s_{\boldsymbol\beta}\bigl(h_v(\Gamma),F_{h_v(\Gamma)}(Q,\Gamma)\bigr).
\end{equation}
\end{definition}

\subsection{Refined Murua formula}
\label{subsec:refined-murua-formula-reorg}

\begin{definition}
\label{def:refined-murua-expression-reorg}
For a connected acyclic quiver $G$ with semi-root $s$, define
\begin{equation}
\label{eq:refined-murua-expression-reorg}
    \mathcal M_s(G;\boldsymbol\beta)
    :=
    \sum_{K\in\mathcal K_s(G)}
    \nu_G(K)B_{\nu_G(K)-1}\,
    \omega(K;\boldsymbol\beta)
    \sum_{\Gamma\in\mathcal W_s({G}\cocont{K})}
    e(\Gamma)\,g_{\boldsymbol\beta}({G}\cocont{K},\Gamma), 
\end{equation}
where $B_n$ is the Bernoulli number with $B_1 = -\frac12$.
\end{definition}

\begin{theorem}[Refined quantum Murua formula]
\label{thm:refined-quantum-murua-reorg}
Let $G$ be a connected acyclic quiver. Then, for every semi-root $s$ and arbitrary edge parameters $\boldsymbol\beta=(\beta_\varepsilon)_{\varepsilon\in E_G}$,
\begin{equation}
\label{eq:refined-murua-theorem-reorg}
    \omega_c(G;\boldsymbol\beta)
    =
    \mathcal M_s(G;\boldsymbol\beta).
\end{equation}
In particular, the right-hand side is independent of the choice of semi-root.
\end{theorem}

\subsection{Proof of the refined Murua formula}
\label{subsec:proof-refined-murua-reorg}

Recall that a polynomial is \emph{multiaffine} if it has degree at most one in each variable separately. We first show that both sides of \eqref{eq:refined-murua-theorem-reorg} are multiaffine and then compare them at the corners of $\{\pm1\}^{E_G}$.

\begin{lemma}[Multiaffine interpolation]
\label{lem:murua-multiaffine-interpolation-reorg}
If $F\in\Bbbk[x_1,\ldots,x_n]$ is multiaffine, then
\begin{equation}
\label{eq:murua-multiaffine-interpolation-reorg}
    F(x_1,\ldots,x_n)
    =
    \frac1{2^n}
    \sum_{\boldsymbol\eta\in\{\pm1\}^n}
    F(\eta_1,\ldots,\eta_n)
    \prod_{i=1}^n(1+\eta_i x_i).
\end{equation}
Hence $F$ is uniquely determined by its values on $\{\pm1\}^n$.
\end{lemma}

\begin{proof}
Apply successively the one-variable identity
$f(x)=\frac{1+x}{2}f(1)+\frac{1-x}{2}f(-1)$.
\end{proof}

\begin{lemma}
\label{lem:murua-both-multiaffine-reorg}
Both $\omega_c(G;\boldsymbol\beta)$ and $\mathcal M_s(G;\boldsymbol\beta)$ are multiaffine in the edge parameters.
\end{lemma}

\begin{proof}
The multiaffinity of $\omega_c(G;\boldsymbol\beta)$ is
Corollary~\ref{cor:omega-multiaffine}.  Fix
$K\in\mathcal K_s(G)$ and
$\Gamma\in\mathcal W_s({G}\cocont{K})$.  Since $s$ is the sink of
$\Gamma$, the spider shots $h_v(\Gamma)$, $v\neq s$, partition
$E_{{G}\cocont{K}}$, so every parameter attached to an edge outside
$K$ occurs exactly once in
$g_{\boldsymbol\beta}({G}\cocont{K},\Gamma)$.  The remaining
parameters are precisely those of $E_K$ and occur only in
$\omega(K;\boldsymbol\beta)$.  The latter is multiaffine because
$\omega$ is multiplicative under disjoint union and, on each connected
component, equals the corresponding connected coefficient covered by
Corollary~\ref{cor:omega-multiaffine}.  Hence each summand, and
therefore $\mathcal M_s(G;\boldsymbol\beta)$, is multiaffine.
\end{proof}

Fix $\boldsymbol\eta=(\eta_\varepsilon)_{\varepsilon\in E_G}\in\{\pm1\}^{E_G}$. For a spider shot $h$, write
\begin{equation}
\label{eq:corner-edge-sets-reorg}
    E_+(h):=\{\varepsilon\in h\mid\eta_\varepsilon=+1\},
    \qquad
    E_-(h):=\{\varepsilon\in h\mid\eta_\varepsilon=-1\}.
\end{equation}

\begin{lemma}[Corner values of the spider factor]
\label{lem:murua-spider-corners-reorg}
At $\boldsymbol\beta=\boldsymbol\eta$,
\begin{equation}
\label{eq:murua-spider-corners-reorg}
    s_{\boldsymbol\eta}(h,F)
    =
    \begin{cases}
        (-1)^{|E_-(h)|},&F=E_-(h),\\
        -(-1)^{|E_-(h)|},&F=E_+(h),\\
        0,&\text{otherwise}.
    \end{cases}
\end{equation}
The first two cases are disjoint because every spider shot is nonempty.
\end{lemma}

\begin{proof}
If $\eta_\varepsilon=+1$, then $\beta_{\varepsilon,+}=1$ and $\beta_{\varepsilon,-}=0$; if $\eta_\varepsilon=-1$, then $\beta_{\varepsilon,+}=0$ and $\beta_{\varepsilon,-}=-1$. Hence the first product in \eqref{eq:refined-spider-factor-reorg} is nonzero exactly for $F=E_-(h)$ and equals $(-1)^{|E_-(h)|}$, while the second is nonzero exactly for $F=E_+(h)$ and has the same value before the overall minus sign.
\end{proof}

Since $F_h(Q,\Gamma)$ records the edges whose orientations in $\Gamma$
and $Q$ are opposite, the two terms in
\eqref{eq:refined-spider-factor-reorg} are precisely the two local
operator orderings occurring in the corresponding commutator.  We
make this correspondence explicit before proving the corner identity.

\begin{lemma}[Nested-adjoint decomposition by spider webs]
\label{lem:adjoint-spider-correspondence-reorg}
Fix $K\in\mathcal K_s(G)$ and put
$Q:={G}\cocont{K}$, $r:=|V_Q|=\nu_G(K)$.
Write
$K=\{s\}\sqcup C_1\sqcup\cdots\sqcup C_{r-1}$.
In the multilinear part of
$\ad_\Omega^{r-1}(A)$ in which $A$ supplies the root $s$ and the
$r-1$ copies of $\Omega$ supply the connected contributions of
$C_1,\ldots,C_{r-1}$, the contribution of the remaining edges can be
reorganized as
\begin{equation}
\label{eq:adjoint-spider-correspondence-reorg}
\sum_{\Gamma\in\mathcal W_s(Q)}
|\operatorname{LE}(\Gamma)|\,
g_{\boldsymbol\eta}(Q,\Gamma).
\end{equation}
More precisely, the terms are indexed by pairs
$(\Gamma,\tau)$ with
$\Gamma\in\mathcal W_s(Q)$ and
$\tau\in\operatorname{LE}(\Gamma)$, and for fixed $\Gamma$ the
edge factor is $g_{\boldsymbol\eta}(Q,\Gamma)$, independently of
$\tau$.
\end{lemma}

\begin{proof}
We work first with fixed vertex, edge, and component labels, so no
symmetry factors enter the following bookkeeping.  If $Q$ contains a
tadpole, then an edge of $G\setminus K$ has both endpoints in the same
connected component of $K$. Such an edge cannot be produced between
two distinct factors in the multilinear part of the nested adjoint.
Hence this contribution vanishes, in agreement with
$\mathcal W_s(Q)=\emptyset$. We may therefore assume that $Q$ is
tadpole-free.

A choice of the order in which the factors associated with
$C_1,\ldots,C_{r-1}$ occupy the successive adjoint positions gives a
total order $\tau$ of the vertices of $Q$, with the root $s$ first.
Orient each edge of $Q$ from the vertex inserted later to the vertex
inserted earlier, and denote the resulting orientation by $\Gamma$.
It is acyclic and $s$ is a sink.

Suppose that, when a vertex $v\neq s$ is inserted, no edge joins it
to a previously inserted vertex. Then the prescribed edge monomial
occurs with the same coefficient in the two terms of the commutator
and cancels. Thus a nonzero contribution requires every
$v\neq s$ to have at least one outgoing edge in $\Gamma$. Therefore
$s$ is the unique sink of $\Gamma$, and
Lemma~\ref{lem:almost-rooted-equivalence-reorg} gives
$\Gamma\in\mathcal W_s(Q)$. By construction, the nesting order
$\tau$ is a linear extension of $\Gamma$.

Conversely, let $\Gamma\in\mathcal W_s(Q)$ and
$\tau\in\operatorname{LE}(\Gamma)$. Since $s$ is the unique sink, it
is first in $\tau$. Moreover, every edge in the spider shot
$h_v(\Gamma)$ joins $v$ to a vertex preceding it in $\tau$.
Successively inserting the component factors in the order prescribed
by $\tau$ therefore reconstructs the corresponding nested-adjoint
term and induces precisely the orientation $\Gamma$. Hence the
nesting orders producing a fixed $\Gamma$ are exactly
$\operatorname{LE}(\Gamma)$.

It remains to compute the edge factor. Fix $v\neq s$, and set
$h=h_v(\Gamma)$ and $F=F_h(Q,\Gamma)$. In the first term of the local
commutator, the factor associated with $v$ lies to the left of the
previously inserted factors; in the second term it lies to their
right. By the operator edge rule \eqref{eq:operator-beta-evaluation},
their contributions are respectively
\[
\prod_{\varepsilon\in h\setminus F}\beta_{\varepsilon,+}
\prod_{\varepsilon\in F}\beta_{\varepsilon,-},
\qquad
\prod_{\varepsilon\in h\setminus F}\beta_{\varepsilon,-}
\prod_{\varepsilon\in F}\beta_{\varepsilon,+}.
\]
Their difference is therefore
$s_{\boldsymbol\beta}(h,F)$ by
\eqref{eq:refined-spider-factor-reorg}. At
$\boldsymbol\beta=\boldsymbol\eta$ this becomes
$s_{\boldsymbol\eta}(h,F)$, whose sign is given by
Lemma~\ref{lem:murua-spider-corners-reorg}.

Finally, the spider shots $h_v(\Gamma)$, $v\neq s$, partition $E_Q$.
Multiplying the local commutator factors therefore gives
\[
\prod_{v\neq s}
s_{\boldsymbol\eta}
\bigl(h_v(\Gamma),F_{h_v(\Gamma)}(Q,\Gamma)\bigr)
=
g_{\boldsymbol\eta}(Q,\Gamma)
\]
by \eqref{eq:refined-g-reorg}. This factor depends only on $\Gamma$.
Thus a fixed $\Gamma$ occurs with multiplicity
$|\operatorname{LE}(\Gamma)|$, proving
\eqref{eq:adjoint-spider-correspondence-reorg}.
\end{proof}

\begin{lemma}[Spider-web expansion for sign specializations]
\label{lem:magnus-spider-decomposition-reorg}
For every $\boldsymbol\eta\in\{\pm1\}^{E_G}$,
\begin{equation}
\label{eq:murua-corner-identity-reorg}
\omega_c(G;\boldsymbol\eta)
=
\mathcal M_s(G;\boldsymbol\eta).
\end{equation}
\end{lemma}

\begin{proof}
By Theorem~\ref{thm:operator-special-master-formula},
$\omega_c(G;\boldsymbol\eta)$ is the operator-product coefficient
obtained from the Magnus expansion after the specialization
$\boldsymbol\beta=\boldsymbol\eta$.  We compare coefficients at the
level of fixed vertex and edge labels.  Differentiating a time-ordered
graph integral with respect to its upper limit gives boundary terms
indexed by sinks: the boundary $t_v=t$ is compatible with the edge
inequalities exactly when $v$ has no outgoing edge.  Hence the boundary
term with the chosen semi-root $s$ carries the same graph coefficient
$\omega_c(G;\boldsymbol\eta)$ and is matched on the right-hand side by
the distinguished factor $A(t)$.  Thus no extra vertex multiplicity or
symmetry factor is introduced by fixing $s$; forgetting the labels at
the end restores the usual automorphism normalization.

At $\boldsymbol\beta=\boldsymbol\eta$, for each edge $\varepsilon$
exactly one of $\beta_{\varepsilon,+}$ and
$\beta_{\varepsilon,-}$ is nonzero, so the local rule
\eqref{eq:operator-beta-evaluation} selects one of the two ordered
contractions, with a factor $-1$ precisely when
$\eta_\varepsilon=-1$. We evaluate the fixed-$s$ boundary coefficient
using the Magnus differential equation
\begin{equation}
\label{eq:magnus-recursion-corner-reorg}
\dot\Omega
=
\frac{\ad_\Omega}{e^{\ad_\Omega}-1}(A)
=
\sum_{q\geq0}\frac{B_q}{q!}\ad_\Omega^q(A).
\end{equation}

Fix $K\in\mathcal K_s(G)$ and put
$r:=\nu_G(K)=|V_{{G}\cocont{K}}|$. Since $s$ is isolated in $K$, write
$K=\{s\}\sqcup C_1\sqcup\cdots\sqcup C_{r-1}$.
The corresponding contribution comes from $q=r-1$: the factor $A$
supplies the root, while the $r-1$ copies of $\Omega$ supply the
connected contributions of $C_1,\ldots,C_{r-1}$. The Bernoulli
coefficient is $B_{r-1}/(r-1)!$.  The adjoint slots are ordered, and
Lemma~\ref{lem:adjoint-spider-correspondence-reorg} already sums over
their admissible assignments, so no additional factorial is inserted.

By Lemma~\ref{lem:adjoint-spider-correspondence-reorg}, the contribution
of the remaining edges is
\[
\sum_{\Gamma\in\mathcal W_s({G}\cocont{K})}
|\operatorname{LE}(\Gamma)|\,
g_{\boldsymbol\eta}({G}\cocont{K},\Gamma).
\]
Since every such $\Gamma$ is acyclic on $r$ vertices,
$|\operatorname{LE}(\Gamma)|=r!e(\Gamma)$, and hence
\[
\frac{B_{r-1}}{(r-1)!}
|\operatorname{LE}(\Gamma)|
=
rB_{r-1}e(\Gamma).
\]

The edge parameters belonging to $K$ are contained in the connected
factors $\omega_c(C_a;\boldsymbol\eta)$. The isolated vertex $s$
contributes $1$, so multiplicativity gives their product as
$\omega(K;\boldsymbol\eta)$. Therefore the contribution associated
with $K$ is
\[
rB_{r-1}\,\omega(K;\boldsymbol\eta)
\sum_{\Gamma\in\mathcal W_s({G}\cocont{K})}
e(\Gamma)\,
g_{\boldsymbol\eta}({G}\cocont{K},\Gamma).
\]
Summing over $K\in\mathcal K_s(G)$ gives
\eqref{eq:murua-corner-identity-reorg}.
\end{proof}

\begin{proof}[Proof of Theorem~\ref{thm:refined-quantum-murua-reorg}]
By Lemma~\ref{lem:murua-both-multiaffine-reorg}, both sides are multiaffine, and by Lemma~\ref{lem:magnus-spider-decomposition-reorg} they agree at every $\boldsymbol\eta\in\{\pm1\}^{E_G}$. Lemma~\ref{lem:murua-multiaffine-interpolation-reorg} therefore proves the identity for arbitrary $\boldsymbol\beta$. Independence of $s$ follows from the left-hand side.
\end{proof}

\subsection{RB and BW specializations}
\label{subsec:Murua-specializations-reorg}

We now compare the refined formula with the two quantum Murua formulas of \cite{Guo:2026xaw}.

\subsubsection{The RB specialization}

Let $G$ be a connected acyclic RB graph and let $s$ be a semi-root. Specialize
\begin{equation}
\label{eq:Murua-RB-corner-reorg}
    \beta_\varepsilon
    =
    \begin{cases}
        1,&\varepsilon\in E_r(G),\\
        -1,&\varepsilon\in E_b(G).
    \end{cases}
\end{equation}
For $K\in\mathcal K_s(G)$, put $Q:={G}\cocont{K}$ and let $\mathcal F_s(Q)\subseteq\mathcal W_s(Q)$ consist of the spider webs $\Gamma$ such that, for every $v\neq s$,
\begin{equation}
\label{eq:RB-admissible-shot-reorg}
    F_{h_v(\Gamma)}(Q,\Gamma)
    =
    h_v(\Gamma)\cap E_b(Q)
    \quad\text{or}\quad
    F_{h_v(\Gamma)}(Q,\Gamma)
    =
    h_v(\Gamma)\cap E_r(Q).
\end{equation}
For $\Gamma\in\mathcal F_s(Q)$, let $\ell(\Gamma)$ be the number of vertices $v\neq s$ for which the second alternative occurs. This is an RB vertex statistic; it should not be confused with the BW edge-reversal statistic $\ell(\Gamma;Q)$ below. Lemma~\ref{lem:murua-spider-corners-reorg} gives
\begin{equation}
\label{eq:RB-spider-factor-reorg}
    g_{\boldsymbol\beta}(Q,\Gamma)
    =
    \begin{cases}
        (-1)^{b(Q)+\ell(\Gamma)},&\Gamma\in\mathcal F_s(Q),\\
        0,&\Gamma\notin\mathcal F_s(Q),
    \end{cases}
\end{equation}
where $b(Q)$ is the number of blue edges of $Q$.

\begin{proposition}[RB quantum Murua formula]
\label{prop:Murua-RB-specialization-reorg}
Let $G$ be a connected acyclic RB graph and let $s$ be a semi-root. Then
\begin{equation}
\label{eq:Murua-RB-result-reorg}
    \omega_{RB,c}(G)
    =
    \sum_{K\in\mathcal K_s(G)}
    \sum_{\Gamma\in\mathcal F_s({G}\cocont{K})}
    (-1)^{\ell(\Gamma)}
    e(\Gamma)\,
    \nu_G(K)B_{\nu_G(K)-1}\,
    \omega_{RB}(K).
\end{equation}
\end{proposition}

\begin{proof}
Under \eqref{eq:Murua-RB-corner-reorg}, Theorem~\ref{thm:refined-quantum-murua-reorg} and \eqref{eq:RB-spider-factor-reorg} give the factor $(-1)^{b({G}\cocont{K})+\ell(\Gamma)}$. Moreover,
\[
    \omega_{RB,c}(G)=(-1)^{b(G)}\omega_c(G;\boldsymbol\beta),
    \qquad
    \omega_{RB}(K)=(-1)^{b(K)}\omega(K;\boldsymbol\beta).
\]
Since the labeled edge sets of $K$ and ${G}\cocont{K}$ partition $E_G$,
$b(G)=b(K)+b({G}\cocont{K})$. The blue-edge signs therefore cancel, leaving $(-1)^{\ell(\Gamma)}$.
\end{proof}

 Our red and blue labels are interchanged relative to those of \cite{Guo:2026xaw}; see Remark~\ref{rem:RB-physical-convention}. Since the factor $(-1)^{b(G)}$ is built into the definition of $\omega_{RB}$, the resulting formula is nevertheless unchanged. Thus \eqref{eq:Murua-RB-result-reorg} is exactly \cite[Eq.~(3.4)]{Guo:2026xaw}.

\subsubsection{The BW specialization}

At the physical BW specialization $(\boldsymbol\beta,\boldsymbol\gamma)=(\mathbf0,\mathbf1)$, directed-core factorization shows that the undirected edges contribute only trivial dressing factors. If the directed core is disconnected, the connected coefficient vanishes. Thus the only nontrivial case is a connected directed core, and it suffices to derive its Murua formula.

At $\boldsymbol\beta=\mathbf0$, the refined spider factor becomes
\begin{equation}
\label{eq:spider-zero-reorg}
    s_{\mathbf0}(h,F)
    =
    \frac{(-1)^{|F|}-(-1)^{|h|-|F|}}{2^{|h|}}
    =
    \begin{cases}
        0,&|h|\text{ even},\\
        \displaystyle\frac{(-1)^{|F|}}{4^k},&|h|=2k+1.
    \end{cases}
\end{equation}
Thus only spider webs of odd outdegree at every non-root vertex contribute. For such a spider web, write
\begin{equation}
\label{eq:m-Gamma-BW-reorg}
    \deg_\Gamma^+(v)=2k_v+1\quad(v\neq s),
    \qquad
    m(\Gamma):=\sum_{v\neq s}k_v,
\end{equation}
and define
\begin{equation}
\label{eq:ell-Gamma-BW-reorg}
    \ell(\Gamma;Q)
    :=
    \sum_{v\neq s}
    \bigl|F_{h_v(\Gamma)}(Q,\Gamma)\bigr|.
\end{equation}
Thus $\ell(\Gamma;Q)$ is exactly the number of edges whose directions in $\Gamma$ are opposite to their directions in $Q$. Hence
\begin{equation}
\label{eq:g-zero-BW-reorg}
    g_{\mathbf0}(Q,\Gamma)
    =
    \frac{(-1)^{\ell(\Gamma;Q)}}{4^{m(\Gamma)}}.
\end{equation}

\begin{proposition}[Undressed BW quantum Murua formula]
\label{prop:undressed-BW-Murua-reorg}
Let $G_D$ be a connected acyclic quiver and let $s$ be a semi-root of $G_D$. Then
\begin{equation}
\label{eq:undressed-BW-Murua-reorg}
    \omega_c(G_D;\mathbf0)
    =
    \sum_{K\in\mathcal K_s(G_D)}
    \nu_{G_D}(K)B_{\nu_{G_D}(K)-1}\,
    \omega(K;\mathbf0)
    \left(
        \sum_{\substack{\Gamma\in\mathcal W_s\bigl((G_D)\cocont K\bigr)\\
        \deg_\Gamma^+(v)\text{ odd for all }v\neq s}}
        \frac{(-1)^{\ell\bigl(\Gamma;(G_D)\cocont K\bigr)}}{4^{m(\Gamma)}}\,e(\Gamma)
    \right).
\end{equation}
\end{proposition}
\begin{proof}
This is Theorem~\ref{thm:refined-quantum-murua-reorg} specialized to $\boldsymbol\beta=\mathbf0$, together with \eqref{eq:spider-zero-reorg} and \eqref{eq:g-zero-BW-reorg}.
\end{proof}

We next compare \eqref{eq:undressed-BW-Murua-reorg} with the BW quantum Murua formula of \cite{Guo:2026xaw}. The only additional ingredient is the difference between our mathematical normalization and the physical BW normalization used there.

\begin{lemma}[Loop counting]
\label{lem:Murua-BW-loop-counting-reorg}
Let $K\in\mathcal K_s(G_D)$ and let
$\Gamma\in\mathcal W_s\bigl((G_D)\cocont K\bigr)$ contribute nontrivially to \eqref{eq:undressed-BW-Murua-reorg}. Then
\begin{equation}
\label{eq:loop-Gamma-reorg}
    L(\Gamma)=2m(\Gamma).
\end{equation}
If
$K=C_0\sqcup\cdots\sqcup C_{r-1}$,
where $C_0=\{s\}$ and $r=\nu_{G_D}(K)$, then
\begin{equation}
\label{eq:loop-decomposition-BW-reorg}
    L(G_D)
    =
    2m(\Gamma)+\sum_{a=0}^{r-1}L(C_a).
\end{equation}
\end{lemma}
\begin{proof}
Because $\Gamma$ is a spider web, its underlying graph is connected, so its loop number is defined. Since every edge of $\Gamma$ has a unique tail at a non-root vertex,
\[
    |E_\Gamma|
    =
    \sum_{v\neq s}\deg_\Gamma^+(v)
    =
    \sum_{v\neq s}(2k_v+1)
    =
    2m(\Gamma)+r-1.
\]
As $|V_\Gamma|=r$, this gives $L(\Gamma)=2m(\Gamma)$. Contracting the connected components of $K$ gives $(G_D)\cocont K$, and hence
\[
    L(G_D)
    =
    L\bigl((G_D)\cocont K\bigr)
    +
    \sum_{a=0}^{r-1}L(C_a).
\]
Since $(G_D)\cocont K$ and $\Gamma$ have the same underlying graph,
$L\bigl((G_D)\cocont K\bigr)=L(\Gamma)$, which proves
\eqref{eq:loop-decomposition-BW-reorg}.
\end{proof}

As explained in Remark~\ref{rem:RB-BW-physics-convention-2}, our mathematical convention omits the factors of $\pm i$ appearing in the physical ordered contractions of \cite{Guo:2026xaw}. After translating the red-blue labels, this does not change the physical RB coefficients, but it produces a loop-dependent sign in the BW basis. By Corollary~\ref{cor:beta-reversal}, $\omega_c(H;\mathbf0)=0$ for a connected acyclic directed graph $H$ of odd loop number. On the even-loop sector the convention translation is
\begin{equation}
\label{eq:BW-directed-convention-Murua-reorg}
    \omega_c(H;\mathbf0)=(-1)^{L(H)/2}\omega_{\mathrm{BW},c}^{\mathrm{phys}}(H).
\end{equation}
We set $\omega_{\mathrm{BW},c}^{\mathrm{phys}}(H)=0$ in the odd-loop case and, for a possibly disconnected acyclic directed graph $K=\bigsqcup_a C_a$, define
\begin{equation}
\label{eq:BW-full-physical-coefficient-reorg}
    \omega_{\mathrm{BW}}^{\mathrm{phys}}(K):=\prod_a\omega_{\mathrm{BW},c}^{\mathrm{phys}}(C_a),
\end{equation}
with the isolated-vertex factor equal to $1$.

Consider a nonzero term of \eqref{eq:undressed-BW-Murua-reorg} and write $K=C_0\sqcup\cdots\sqcup C_{r-1}$. Since $\omega(K;\mathbf0)=\prod_a\omega_c(C_a;\mathbf0)$, nonvanishing implies that every $L(C_a)$ is even. Lemma~\ref{lem:Murua-BW-loop-counting-reorg} then shows that $L(G_D)$ is even as well. Hence \eqref{eq:BW-directed-convention-Murua-reorg} applies to $G_D$ and every connected component of $K$, and the relative normalization factor between the left-hand side and $\omega(K;\mathbf0)$ is
\[
    (-1)^{\frac12\sum_{a=0}^{r-1}L(C_a)-\frac12L(G_D)}
    =
    (-1)^{-m(\Gamma)}
    =
    (-1)^{m(\Gamma)},
\]
where the first equality follows from
Lemma~\ref{lem:Murua-BW-loop-counting-reorg}, and the second uses
$m(\Gamma)\in\mathbb Z$. Thus the additional normalization factor $(-1)^{m(\Gamma)}$ gives
\begin{equation}
\label{eq:BW-Murua-convention-factor-reorg}
    (-1)^{m(\Gamma)}
    \frac{(-1)^{\ell\bigl(\Gamma;(G_D)\cocont K\bigr)}}{4^{m(\Gamma)}}
    =
    \left(-\frac14\right)^{m(\Gamma)}
    (-1)^{\ell\bigl(\Gamma;(G_D)\cocont K\bigr)}.
\end{equation}
Therefore \eqref{eq:undressed-BW-Murua-reorg}, expressed in the physical normalization of \cite{Guo:2026xaw}, becomes
\begin{equation}
\label{eq:physical-BW-Murua-derived-reorg}
    \omega_{\mathrm{BW},c}^{\mathrm{phys}}(G_D)
    =
    \sum_{K\in\mathcal K_s(G_D)}
    \nu_{G_D}(K)B_{\nu_{G_D}(K)-1}\,
    \omega_{\mathrm{BW}}^{\mathrm{phys}}(K)
    \left(
        \sum_{\substack{\Gamma\in\mathcal W_s\bigl((G_D)\cocont K\bigr)\\
        \deg_\Gamma^+(v)\text{ odd for all }v\neq s}}
        \left(-\frac14\right)^{m(\Gamma)}
        (-1)^{\ell\bigl(\Gamma;(G_D)\cocont K\bigr)}
        e(\Gamma)
    \right).
\end{equation}
Under the notation dictionary of
Remark~\ref{rem:Murua-notation-dictionary-reorg}, namely
\[
    p=E_{G_D}\setminus E_K,
    \qquad
    G_D\setminus p=K,
    \qquad
    |p|=\nu_{G_D}(K),
\]
\eqref{eq:physical-BW-Murua-derived-reorg} is precisely the BW quantum Murua formula \cite[Eq.~(3.9)]{Guo:2026xaw}.

\begin{remark}
For a BW graph with undirected edges, no additional Murua recursion is needed. If its directed core is disconnected, the connected coefficient is zero; otherwise directed-core factorization shows that the undirected edges contribute only through dressing factors, equal to $1$ at $(\boldsymbol\beta,\boldsymbol\gamma)=(\mathbf0,\mathbf1)$. Thus the full BW quantum Murua formula is determined by the directed-core formula above.
\end{remark}

\section{Sum rules}
\label{sec:sum-rules}

We sum $e$ and $\omega$ over all orientation assignments of a fixed
tadpole-free undirected graph and relate the results to Tutte and
chromatic polynomials.

Let $\widehat G=(V,E)$ be such a graph. We denote by
$\operatorname{Ori}(\widehat G)$ the set of quivers obtained by
choosing independently one of the two orientations of every edge.
Parallel edges are regarded as distinct, so
$|\operatorname{Ori}(\widehat G)|=2^{|E|}$.

For an undirected edge $u\in E$, let $\varepsilon$ and
$\bar\varepsilon$ denote its two orientations and attach independent
parameters $\beta_\varepsilon$ and $\beta_{\bar\varepsilon}$. Define
\begin{equation}
\label{eq:beta-average}
    \overline\beta_u
    :=
    \frac{\beta_\varepsilon+\beta_{\bar\varepsilon}}{2}.
\end{equation}
This definition is independent of the choice of $\varepsilon$. We
write $\boldsymbol\beta$ for the collection of all directional
parameters and
$\overline{\boldsymbol\beta}
=(\overline\beta_u)_{u\in E}$.
 
For an undirected graph $\widehat G=(V,E)$, let
$\operatorname{Aut}(\widehat G)$ denote the group of pairs of
bijections
\[
    \phi_V:V\longrightarrow V,
    \qquad
    \phi_E:E\longrightarrow E,
\]
such that, if an edge $u\in E$ joins $i$ and $j$, then
$\phi_E(u)$ joins $\phi_V(i)$ and $\phi_V(j)$. We set
$\sigma(\widehat G):=|\operatorname{Aut}(\widehat G)|$.
For $G\in\operatorname{Ori}(\widehat G)$, let
$\operatorname{Aut}(G)$ denote the group of pairs of bijections of
the vertex and edge sets preserving the source and target of every
edge, and set $\sigma(G):=|\operatorname{Aut}(G)|$.

Let $R$ be a commutative ring and let
$F:\operatorname{Ori}(\widehat G)\to R$ be a function that is constant
on isomorphism classes of quivers. The group
$\operatorname{Aut}(\widehat G)$ acts on
$\operatorname{Ori}(\widehat G)$ by transporting edge orientations.
The stabilizer of $G\in\operatorname{Ori}(\widehat G)$ under this
action is precisely $\operatorname{Aut}(G)$. Hence the
orbit--stabilizer theorem gives
\begin{equation}
\label{eq:orientation-orbit-reduction}
    \sum_{[G]\in\operatorname{Ori}(\widehat G)/\cong}
    \frac{F(G)}{\sigma(G)}
    =
    \frac{1}{\sigma(\widehat G)}
    \sum_{G\in\operatorname{Ori}(\widehat G)}F(G),
\end{equation}
where $[G]$ denotes the isomorphism class of $G$. For trees, this
specializes to the orbit--stabilizer argument used in
\cite[Eqs.~(4.19)--(4.21)]{Kim:2024svw}.

\subsection{\texorpdfstring{$e$}{e}-sum rule}
\label{subsec:e-sum-rule}

\begin{proposition}
\label{prop:e-orientation-sum}
For every tadpole-free undirected graph $\widehat G$,
$\sum_{G\in\operatorname{Ori}(\widehat G)}e(G)=1$.
\end{proposition}

\begin{proof}
Let $n=|V|$. For an acyclic orientation $G$,
$e(G)=|\operatorname{LE}(G)|/n!$, while $e(G)=0$ for a cyclic
orientation. Hence
\[
    \sum_{G\in\operatorname{Ori}(\widehat G)}e(G)
    =
    \frac1{n!}
    \sum_{\sigma\in S_n}
    \#\{G\in\operatorname{Ori}(\widehat G)
        \mid \sigma\in\operatorname{LE}(G)\}.
\]
For every total order $\sigma$, each edge of $\widehat G$ has a unique
orientation compatible with $\sigma$. Thus the cardinality in the
last expression is one, and the sum equals $1$.
\end{proof}

Applying \eqref{eq:orientation-orbit-reduction} to $F=e$ and using
Proposition~\ref{prop:e-orientation-sum}, we obtain
\begin{equation}
\label{eq:e-orientation-sum-isomorphism}
    \sum_{[G]\in\operatorname{Ori}(\widehat G)/\cong}
    \frac{e(G)}{\sigma(G)}
    =
    \frac{1}{\sigma(\widehat G)}.
\end{equation}
For trees, this recovers the second equality in
\cite[Eq.~(4.19)]{Kim:2024svw}.

\subsection{\texorpdfstring{$\omega$}{omega}-sum rule}
\label{subsec:omega-sum-rule}

Define
\begin{equation}
\label{eq:S-omega-definition}
    S_\omega(\widehat G;\boldsymbol\beta)
    :=
    \sum_{G\in\operatorname{Ori}(\widehat G)}
    \omega_c(G;\boldsymbol\beta),
\end{equation}
where each oriented edge uses the parameter associated with its chosen
orientation. If $\widehat G$ is disconnected, then
$S_\omega(\widehat G;\boldsymbol\beta)=0$.

\begin{theorem}[Orientation sum for the $\omega$-function]
\label{thm:omega-orientation-sum}
Let $\widehat G=(V,E)$ be a tadpole-free undirected graph with
$V\neq\emptyset$. Then
\begin{equation}
\label{eq:omega-sum-connected-subgraphs}
    S_\omega(\widehat G;\boldsymbol\beta)
    =
    \sum_{\substack{F\subseteq E\\(V,F)\ {\rm connected}}}
    (-1)^{|F|}
    \prod_{u\in E\setminus F}\overline\beta_u.
\end{equation}
Thus the orientation sum depends on the two directional parameters of
each edge only through their average \eqref{eq:beta-average}.
\end{theorem}

\begin{proof}
If $\widehat G$ is disconnected, both sides vanish, so assume that it
is connected. For every orientation $G$ of $\widehat G$, the general
master formula (Theorem~\ref{thm:general-master-formula}) gives
\[
    \omega_c(G)
    =
    \sum_{\mathcal B\in\operatorname{OSP}(V)}
    c_{|\mathcal B|}
    \prod_{\varepsilon\in E_G}
    \beta_{\varepsilon,s_{\mathcal B}(\varepsilon)},
    \qquad
    c_\ell=\frac{(-1)^{\ell-1}}{\ell}.
\]
Fix $\mathcal B$ and an undirected edge $u$ with orientations
$\varepsilon$ and $\bar\varepsilon$. If the endpoints of $u$ lie in
different blocks, one orientation contributes a ``$+$'' factor and the
other a ``$-$'' factor, so their sum is
$\beta_{\varepsilon,+}+\beta_{\bar\varepsilon,-}$ or
$\beta_{\varepsilon,-}+\beta_{\bar\varepsilon,+}$, both equal to
$\overline\beta_u$. If the endpoints lie in the same block, both
orientations contribute a ``$-$'' factor and their sum is
$\overline\beta_u-1$. Hence
\begin{equation}
\label{eq:omega-sum-OSP}
    S_\omega(\widehat G;\boldsymbol\beta)
    =
    \sum_{\mathcal B\in\operatorname{OSP}(V)}
    c_{|\mathcal B|}
    \prod_{u\in E_{\neq}(\mathcal B)}\overline\beta_u
    \prod_{u\in E_{=}(\mathcal B)}
    (\overline\beta_u-1),
\end{equation}
where $E_{=}(\mathcal B)$ consists of the edges whose endpoints lie in
the same block and
$E_{\neq}(\mathcal B)=E\setminus E_{=}(\mathcal B)$.

Expanding the last product and interchanging the sums gives
\[
    S_\omega(\widehat G;\boldsymbol\beta)
    =
    \sum_{F\subseteq E}
    (-1)^{|F|}
    \prod_{u\in E\setminus F}\overline\beta_u
    \sum_{\substack{\mathcal B\in\operatorname{OSP}(V)\\
                    F\subseteq E_{=}(\mathcal B)}}
    c_{|\mathcal B|}.
\]
Let $\kappa(F)$ be the number of connected components of $(V,F)$.
The ordered set partitions in the inner sum are precisely the ordered
set partitions of these $\kappa(F)$ components. Since
\[
    q^{\kappa(F)}
    =
    \sum_{\ell=1}^{\kappa(F)}
    \ell!\,S(\kappa(F),\ell)\binom q\ell
\]
where $S(m,\ell)$ denotes the Stirling number of the second kind, and
$\left.\frac{d}{dq}\binom q\ell\right|_{q=0}=c_\ell$,
the inner sum equals
$\left.\frac{d}{dq}q^{\kappa(F)}\right|_{q=0}$, which is $1$ if
$\kappa(F)=1$ and $0$ otherwise. Thus only connected spanning
subgraphs survive, proving \eqref{eq:omega-sum-connected-subgraphs}.
\end{proof}

If the directional parameter assignment is invariant under
$\operatorname{Aut}(\widehat G)$, then
$G\mapsto\omega_c(G;\boldsymbol\beta)$ is constant on the
$\operatorname{Aut}(\widehat G)$-orbits in
$\operatorname{Ori}(\widehat G)$. Therefore
\eqref{eq:orientation-orbit-reduction} and
Theorem~\ref{thm:omega-orientation-sum} give
\begin{equation}
\label{eq:omega-orientation-sum-isomorphism}
    \sum_{[G]\in\operatorname{Ori}(\widehat G)/\cong}
    \frac{\omega_c(G;\boldsymbol\beta)}{\sigma(G)}
    =
    \frac{1}{\sigma(\widehat G)}
    \sum_{\substack{F\subseteq E\\(V,F)\ {\rm connected}}}
    (-1)^{|F|}
    \prod_{u\in E\setminus F}\overline\beta_u.
\end{equation}
For trees, this specializes to the second equality in
\cite[Eq.~(4.20)]{Kim:2024svw}. 
We expect the sum rule \eqref{eq:omega-orientation-sum-isomorphism} to
play an important role in extending generalized unitarity methods
\cite{Bern:1994cg,Bern:2011qt} from scattering amplitudes to Magnus
amplitudes \cite{Brandhuber:2025igz}.

\begin{example}
For the cycle $C_r$, a connected spanning subgraph is either the
whole cycle or a spanning tree obtained by deleting one edge.  Thus
\begin{equation}
\label{eq:cycle-sum-general}
    S_\omega(C_r;\boldsymbol\beta)
    =
    (-1)^{r-1}
    \left(
        \sum_{u\in E(C_r)}\overline\beta_u-1
    \right).
\end{equation}

For a graph $\widehat G$ consisting of two vertices joined by parallel
edges $u_1,\ldots,u_m$, one obtains
\begin{equation}
\label{eq:parallel-edge-sum}
    S_\omega(\widehat G;\boldsymbol\beta)
    =
    \prod_{a=1}^m
    (\overline\beta_{u_a}-1)
    -
    \prod_{a=1}^m
    \overline\beta_{u_a}.
\end{equation}

Finally, let $K_4$ have edge parameters
$b_{ij}:=\overline\beta_{\{i,j\}}$, $1\leq i<j\leq4$.  Then
\begin{equation}
\label{eq:K4-sum}
\begin{split}
    S_\omega(K_4;\boldsymbol\beta)
    ={}&
    1
    -\sum_{1\leq i<j\leq4}b_{ij}
    +\sum_{\substack{u,v\in E(K_4)\\u<v}}b_ub_v
    -\sum_{T\in\operatorname{ST}(K_4)}
    \prod_{u\in E(K_4)\setminus T}b_u,
\end{split}
\end{equation}
where $\operatorname{ST}(K_4)$ denotes the set of spanning trees of
$K_4$.
\end{example}

\subsection{Specializations of the \texorpdfstring{$\omega$}{omega}-sum rule}
\label{subsec:omega-sum-specializations}

\begin{corollary}[Tree sum rule]
\label{cor:omega-tree-sum}
For every undirected tree $\widehat\tau$,
\begin{equation}
\label{eq:omega-sum-tree}
    \sum_{\tau\in\operatorname{Ori}(\widehat\tau)}
    \omega_c(\tau;\boldsymbol\beta)
    =(-1)^{|V|-1}.
\end{equation}
\end{corollary}

\begin{proof}
The only connected spanning subgraph of a tree is the tree itself, so
Theorem~\ref{thm:omega-orientation-sum} gives
\[
    S_\omega(\widehat\tau;\boldsymbol\beta)
    =(-1)^{|E|}=(-1)^{|V|-1}. \qedhere
\]
\end{proof}
Suppose next that the two orientations of every edge carry the same
parameter,
$\beta_\varepsilon
    =
    \beta_{\bar\varepsilon}
    =:
    \beta_u$.
Then $\overline\beta_u=\beta_u$, and
\begin{equation}
\label{eq:omega-sum-symmetric-beta}
    S_\omega(\widehat G;\boldsymbol\beta)
    =
    \sum_{\substack{
        F\subseteq E\\
        (V,F)\ {\rm connected}
    }}
    (-1)^{|F|}
    \prod_{u\in E\setminus F}\beta_u.
\end{equation}
Under the further uniform specialization
$\beta_u=\beta$ for every $u\in E$,
\begin{equation}
\label{eq:omega-sum-uniform}
    S_\omega(\widehat G;\beta)
    =
    \sum_{\substack{
        F\subseteq E\\
        (V,F)\ {\rm connected}
    }}
    (-1)^{|F|}
    \beta^{|E|-|F|}.
\end{equation}
The same formula holds whenever
$\overline\beta_u=\beta$ for every $u$, without requiring the two
directional parameters themselves to be equal.

\begin{corollary}
\label{cor:omega-sum-zero-average}
If $\widehat G$ is connected and
$\overline\beta_u=0$ for every $u\in E$, then
\begin{equation}
\label{eq:omega-sum-zero-average}
    S_\omega(\widehat G;\boldsymbol\beta)
    =
    (-1)^{|E|}.
\end{equation}
\end{corollary}

\begin{proof}
In \eqref{eq:omega-sum-connected-subgraphs}, every term with
$F\neq E$ contains a vanishing factor $\overline\beta_u$, so only
$F=E$ survives.
\end{proof}

\subsection{Deletion--contraction}
\label{subsec:sum-deletion-contraction}

For an undirected graph $\widehat G=(V,E)$ with $V\neq\emptyset$,
define
\begin{equation}
\label{eq:C-G-definition}
    \mathcal C_{\widehat G}
    (\overline{\boldsymbol\beta})
    :=
    \sum_{\substack{
        F\subseteq E\\
        (V,F)\ {\rm connected}
    }}
    (-1)^{|F|}
    \prod_{u\in E\setminus F}\overline\beta_u.
\end{equation}
Here $\widehat G$ is allowed to contain tadpoles, so that the class of
graphs is closed under contraction.  If $\widehat G$ is disconnected,
the sum is empty and $\mathcal C_{\widehat G}=0$.

Let $\widehat G$ be an undirected graph and let $u$ be an edge of $\widehat G$. Recall that $\widehat G_{\setminus u}$ denotes the graph obtained by deleting $u$. Denote by
$\widehat G/u$ the graph obtained by contracting $u$. In the notation of Section~\ref{sec:hopf}, $\widehat G/u = \widehat G\cocont{H_u}$,
where $H_u\preceq\widehat G$ is the cut-sub-\mquiver with edge set $E_{H_u}=\{u\}$.

\begin{proposition}[Deletion--contraction]
\label{prop:C-deletion-contraction}
For every non-tadpole edge $u$,
\begin{equation}
\label{eq:C-deletion-contraction}
    \mathcal C_{\widehat G}
    =
    -\mathcal C_{\widehat G/u}
    +
    \overline\beta_u\,
    \mathcal C_{\widehat G_{\setminus u}},
\end{equation}
with the parameters of all remaining edges inherited under deletion
and contraction.
\end{proposition}

\begin{proof}
Split the connected spanning sets $F\subseteq E$ according to whether
$u\in F$.  If $u\in F$, contraction gives a connected spanning set of
$\widehat G/u$ and changes $(-1)^{|F|}$ by a factor $-1$.  If
$u\notin F$, then $F$ is a connected spanning set of
$\widehat G_{\setminus u}$ and the complementary product contains
$\overline\beta_u$.  Adding the two contributions gives
\eqref{eq:C-deletion-contraction}.
\end{proof}

By Theorem~\ref{thm:omega-orientation-sum},
\[
    S_\omega(\widehat G;\boldsymbol\beta)
    =
    \mathcal C_{\widehat G}(\overline{\boldsymbol\beta})
\]
for tadpole-free $\widehat G$.  Hence, whenever
$\widehat G/u$ is also tadpole-free,
\begin{equation}
\label{eq:S-omega-deletion-contraction}
    S_\omega(\widehat G;\boldsymbol\beta)
    =
    -S_\omega(\widehat G/u;\boldsymbol\beta)
    +
    \overline\beta_u\,
    S_\omega(\widehat G_{\setminus u};\boldsymbol\beta).
\end{equation}
If $u$ is a bridge, the deletion term vanishes, giving
$S_\omega(\widehat G;\boldsymbol\beta)=-S_\omega(\widehat G/u;\boldsymbol\beta)$.

\subsection{Relation with the Tutte and chromatic polynomials}
\label{subsec:sum-rule-Tutte}

Let $\widehat G$ be a connected tadpole-free undirected graph, with
$n=|V|$, $m=|E|$, and
$L(\widehat G)=m-n+1$. Recall the standard spanning-subgraph expansion
of the Tutte polynomial; see, for example, \cite{bollobas1998modern, tutte1954contribution}:
\[
    T_{\widehat G}(1,y)
    =
    \sum_{\substack{F\subseteq E\\(V,F)\ {\rm connected}}}
    (y-1)^{|F|-n+1}.
\]
Under the uniform specialization
$\overline\beta_u=\beta$, comparison with
\eqref{eq:omega-sum-uniform} gives, for $\beta\neq0$,
\begin{equation}
\label{eq:S-omega-Tutte}
    S_\omega(\widehat G;\beta)
    =
    (-1)^{n-1}\beta^{L(\widehat G)}
    T_{\widehat G}\left(1,1-\frac1\beta\right).
\end{equation}
After simplification the right-hand side is polynomial in $\beta$, so
the identity extends to $\beta=0$.

To see the relation to chromatic polynomials, define
\begin{equation}
\label{eq:P-G-q-beta}
    \mathcal P_{\widehat G}
    (q;\overline{\boldsymbol\beta})
    :=
    \sum_{f:V\to[q]}
    \prod_{u=\{i,j\}\in E}
    \left(
        \overline\beta_u
        -
        \mathbf 1_{\{f(i)=f(j)\}}
    \right).
\end{equation}
Expanding the product gives
\begin{equation}
\label{eq:P-G-spanning}
    \mathcal P_{\widehat G}
    (q;\overline{\boldsymbol\beta})
    =
    \sum_{F\subseteq E}
    (-1)^{|F|}
    \left(
        \prod_{u\in E\setminus F}\overline\beta_u
    \right)
    q^{\kappa(F)},
\end{equation}
where $\kappa(F)$ is the number of connected components of $(V,F)$.
Consequently,
\begin{equation}
\label{eq:S-omega-P-derivative}
    S_\omega(\widehat G;\boldsymbol\beta)
    =
    \left.
    \frac{d}{dq}
    \mathcal P_{\widehat G}
    (q;\overline{\boldsymbol\beta})
    \right|_{q=0}.
\end{equation}

Let $\operatorname{Chr}_{\widehat G}(q)$ denote the chromatic
polynomial of $\widehat G$. At
$\overline\beta_u=1$ for every $u\in E$,
$\mathcal P_{\widehat G}(q;\mathbf1)
    =
    \operatorname{Chr}_{\widehat G}(q)$,
and hence $S_\omega(\widehat G;\mathbf1)
    =
    \operatorname{Chr}_{\widehat G}'(0)$. 
Thus the orientation sum is the linear coefficient in $q$ of the
multivariate graph polynomial
$\mathcal P_{\widehat G}(q;\overline{\boldsymbol\beta})$.


\subsection*{Acknowledgements}

JWK and SL are grateful to the organizers of the Amplitudes 2026 conference, during which significant progress toward the master formulas was made. 
JL was supported by the Beatriz Galindo Senior grant BG24/00114 from the Spanish Ministry of Science, Innovation and Universities.
JHK was supported by the Department of Energy (Grant No.~DE-SC0011632) and by the Walter Burke Institute for Theoretical Physics.
SK and SL were supported by the National Research Foundation of Korea (NRF) grant NRF RS-2024-00351197 and KIAS grant PG006002.

\subsection*{Declaration on the Use of Generative AI} 
Generative AI models (ChatGPT by OpenAI and Claude by Anthropic) were used to assist with language editing, improve readability, and help reorganize the presentation of some proofs. It did not generate the underlying mathematical or physical ideas. All mathematical and physical content, including the validity of every proof, remains the sole responsibility of the authors.

\subsection*{Conflict of Interest}
The authors declare that they have no conflicts of interest relevant to this work.

\subsection*{Data Availability Statement}
No datasets were generated or analyzed during the current study. 


\bibliographystyle{alpha}  
\bibliography{math-references}

\end{document}